\documentclass[preprint]{imsart} %%aap wieder in [] einfügen !!

\RequirePackage{amsthm,amsmath,amsfonts,amssymb}
\RequirePackage[numbers]{natbib}
\RequirePackage[colorlinks,citecolor=blue,urlcolor=blue]{hyperref}

\startlocaldefs
\theoremstyle{plain}

\usepackage{amsmath}
\usepackage{amsthm}
\usepackage{bigints}
\usepackage{amssymb}
\usepackage{bm}
\usepackage{color}
\usepackage{graphicx}
\usepackage{float}
\usepackage{tikz}

\newcommand{\pp}{\mathbb{P}}
\newcommand{\ee}{\mathbb{E}}
\newcommand{\lk}{\left[ }
\newcommand{\rk}{\right] }
\newcommand{\lc}{\left(}
\newcommand{\rc}{\right)}
\newcommand{\RR}{\mathbb{R}}
\newcommand{\bR}{\overline{\mathbb{R}}}
\newcommand{\fb}{\overline{F}}
\newcommand{\gb}{\overline{G}}
\newcommand{\QQ}{\mathbb{Q}}
\newcommand{\N}{\mathbb{N}}

\newcommand{\bc}{\mathcal{B}}
\newcommand{\id}{\mathbf{1}}
\newcommand{\ac}{\mathcal{A}}
\newcommand{\nlim}{\lim_{n\to\infty}}

\newcommand{\bmx}{{\bm{x}}}
\newcommand{\bmX}{{\bm{X}}}

\newcommand{\bmt}{{\bm{t}}}

\newcommand{\bell}{{\bm{\ell}}}
\newcommand{\binfty}{{\bm{\infty}}}

\newcommand{\sumn}{\sum_{i=1}^n}
\newcommand{\sumd}{\sum_{i=1}^d}
\newcommand{\prodd}{\prod_{i=1}^d}

\newcommand{\leqn}{1\leq i\leq n}
\newcommand{\minf}{(-\infty,\infty]}

\newtheorem{thm}{Theorem}[section]
\newtheorem*{thm*}{Theorem}

\newtheorem{ex}[thm]{Example}
\newtheorem{lem}[thm]{Lemma}
\newtheorem{rem}{Remark}
\newtheorem{defn}[thm]{Definition}
\newtheorem{cor}[thm]{Corollary}
\newtheorem{prop}[thm]{Proposition}
\newtheorem*{cond}{Condition $(\Diamond)$}
\newtheorem*{condprime}{Condition $(\Diamond^\prime)$}
\newtheoremstyle{plain}
  {\topsep}   % ABOVESPACE
  {\topsep}   % BELOWSPACE
  {\itshape}  % BODYFONT
  {0pt}       % INDENT (empty value is the same as 0pt)
  {\bfseries} % HEADFONT
  {}         % HEADPUNCT
  {\newline} % HEADSPACE
   {}

\allowdisplaybreaks

\endlocaldefs

\begin{document}

\begin{frontmatter}
\title{Exchangeable min-id sequences: Characterization, exponent measures and non-decreasing id-processes}
%\title{A sample article title with some additional note\thanksref{T1}}
\runtitle{Exch min-id sequences: Characterization, exponent measures nnnd id-processes}
%\thankstext{T1}{A sample of additional note to the title.}

\begin{aug}
%%%%%%%%%%%%%%%%%%%%%%%%%%%%%%%%%%%%%%%%%%%%%%
%%Only one address is permitted per author. %%
%%Only division, organization and e-mail is %%
%%included in the address.                  %%
%%Additional information can be included in %%
%%the Acknowledgments section if necessary. %%
%%%%%%%%%%%%%%%%%%%%%%%%%%%%%%%%%%%%%%%%%%%%%%
\author[A]{\fnms{Florian} \snm{Br{\"u}ck}\ead[label=e1,mark]{florian.brueck@tum.de} } %\ead[label=e1]{florian.brueck@tum.de}  
\author[B]{\fnms{Jan-Frederik} \snm{Mai}\ead[label=e2,mark]{jan-frederik.mai@xaia.com}}
% %\and
\author[A]{\fnms{Matthias} \snm{Scherer}\ead[label=e3,mark]{scherer@tum.de}}

% %%%%%%%%%%%%%%%%%%%%%%%%%%%%%%%%%%%%%%%%%%%%%%
% %% Addresses                                %%
% %%%%%%%%%%%%%%%%%%%%%%%%%%%%%%%%%%%%%%%%%%%%%%
\address[A]{Lehrstuhl f{\"u}r Finanzmathematik, Technische Universit{\"a}t M{\"u}nchen, Garching, Germany, \printead{e1,e3}}

\address[B]{XAIA Investment GmbH, M{\"u}nchen, Germany, \printead{e2}}

\end{aug}

\begin{abstract}
We establish a one-to-one correspondence between (i) exchangeable sequences of random variables whose finite-dimensional distributions are minimum (or maximum) infinitely divisible and (ii) non-negative, non-decreasing, infinitely divisible stochastic processes. The exponent measure of an exchangeable minimum infinitely divisible sequence is shown to be the sum of a very simple ``drift measure'' and a mixture of product probability measures, which uniquely corresponds to the L\'evy measure of a non-negative and non-decreasing infinitely divisible process. The latter is shown to be supported on non-negative and non-decreasing functions. In probabilistic terms, the aforementioned infinitely divisible process is equal to the conditional cumulative hazard process associated with the exchangeable sequence of random variables with minimum (or maximum) infinitely divisible marginals. Our results provide an analytic umbrella which embeds the de Finetti subfamilies of many interesting classes of multivariate distributions, such as exogenous shock models, exponential and geometric laws with lack-of-memory property, min-stable multivariate exponential and extreme-value distributions, as well as reciprocal Archimedean copulas with completely monotone generator and Archimedean copulas with log-completely monotone generator.
\end{abstract}



\end{frontmatter}
%%%%%%%%%%%%%%%%%%%%%%%%%%%%%%%%%%%%%%%%%%%%%%
%% Please use \tableofcontents for articles %%
%% with 50 pages and more                   %%
%%%%%%%%%%%%%%%%%%%%%%%%%%%%%%%%%%%%%%%%%%%%%%
%\tableofcontents

%%%%%%%%%%%%%%%%%%%%%%%%%%%%%%%%%%%%%%%%%%%%%%
%%%% Main text entry area:

\section{Introduction}
The present article bridges the gap between two well-established theories: exchangeable sequences of minimum/maximum infinitely divisible (min-id/max-id) random variables and non-negative and non-decreasing (nnnd) infinitely divisible stochastic processes. So far, these topics are studied in (mostly) separate communities and we seek to address both. On the one hand, our results can be understood as a particular application of the theory of infinitely divisible (id) processes. General id-processes are studied in \cite{kabluchko2016,BarndorffNielsen2006InfiniteDF,rosinskiinfdivproc}, related literature concerned with subfamilies comprises \cite{hakassou2013idt,koppmolchanov,maischereridtsubordinators2019,mansuy2005processes,hakassou2012alphaidt,bertoinsubordinators,satolevyinfinitely,skorohod1991random}. The article \cite{rosinskiinfdivproc} unifies the literature by establishing a general analytical apparatus to deal with id-processes by means of a general L\'evy measure on the path space $\RR^\RR$. We refine these results by restricting our attention to processes with nnnd c\`adl\`ag paths, but remain totally general aside from this assumption. On the other hand, finite-dimensional probability distributions that are max- (or min-)id naturally arise as limit laws of suitably scaled maxima of independent random vectors, see \cite{husler1989limit}, a textbook account being \cite{resnickextreme}. Such distributions and prominent subfamilies, like max- (or min-) stable laws, are well-established in the applied probability and statistics literature, see e.g.\ \cite{balkema1977max,marshallolkin90,alzaid_proschan_1994,joe1996multivariate,mulinacci2015marshall,genestreciparchim}, and have recently gained interest in the modeling of spatial extremes, see \cite{huser2018penultimate, BoppShabyHuser2020, PADOAN20131, Huser2018MaxinfinitelyDM}. In analytical terms, such probability distributions are canonically described by a so-called exponent measure and the work of \cite{vatanmaxidinfinitedim} generalizes this framework to infinite sequences of random variables.\\
In our article, we apply de Finetti's Theorem~\cite[Chapter 1.3]{aldousexchangeable} to derive a correspondence between nnnd c\`adl\`ag id-processes and exchangeable sequences of random variables all of whose finite-dimensional distributions are min- (or max-)id. In fact, given a nnnd c\`adl\`ag id-process $H=(H_t)_{t \in \RR}$, we may define an infinite exchangeable sequence $\bmX=(X_1,X_2,\ldots)$ of min-id random variables via the almost sure relation
\begin{align}
 \pp\Bigg(\bigcap_{i\in\N} \{X_i >t_i\}\ \Bigg\vert \ H \Bigg) 
 = \prod_{i\in\N} \exp\lc-H_{t_i}\rc, \ (t_1,t_2,\ldots) \in [-\infty,\infty)^{\N}.\label{eqnconstructionminidsequence}
\end{align}
The sequence $\bmX=(X_1,X_2,\ldots)$ is a conditionally (on H) i.i.d.\ sequence with conditional marginal survival function $t\mapsto \exp(-H_t)$\footnote{This follows from (\ref{eqnconstructionminidsequence}) by choosing $(t_1,t_2,\ldots)$ such that all but one $t_i$ are equal to $-\infty$. }, where $H_{-\infty}$ is defined as the almost sure limit $\lim_{t\to-\infty}H_t$.
While we may plug in arbitrary nnnd c\`adl\`ag id-processes $H$ on the right-hand side of Equation (\ref{eqnconstructionminidsequence}), we prove that one actually obtains each exchangeable min-id sequence $\bmX$ on the left-hand side of Equation (\ref{eqnconstructionminidsequence}) via this construction, i.e.\ we prove a one-to-one correspondence of exchangeable min-id sequences and nnnd c\`adl\`ag id-processes.

The assumption of (infinite) exchangeability often appears in practical applications when the overall complexity of the model needs to be limited. For example, if a multivariate phenomenon cannot be modeled by independent components, but further information about the dependence structure of the margins is missing, it is often reasonable to assume exchangeability of the components. 
The assumption of exchangeability is convenient in such situations since it preserves complete flexibility between independence and full dependence.
Moreover, even when there are legitimate reasons to believe that the modeled margins are not exchangeable, it is often reasonable to divide the modeled components into homogeneous subgroups which are intrinsically exchangeable. 
Building a model which has exchangeable subgroups allows to maintain a quite simple dependence structure inside these subgroup, while the non-exchangeable dependence structure between subgroups remains tractable.

A particular advantage of min-(or max-)id distributions in the modeling of extreme events is their flexible dependence structure. For example, the so-called extremal coefficient, which measures the distance of the model to the model of independent components at a fixed threshold, can actually be chosen to be threshold-dependent, a feature that is often discovered in real world data sets, see e.g.\ \cite{Huser2018MaxinfinitelyDM}. On the other hand, if one would resort to classical min-stable models, such features cannot be modeled, since the extremal coefficient of min-stable models with identical margins is known to be independent of the chosen threshold. Moreover, going from min-stability to min-infinite divisibility comes essentially without additional mathematical technicalities, while min-id models easily incorporate arbitrary marginal distributions in contrast to the restricted flexibility of marginal distributions in min-stable models.
%Furthermore, applications of max-id distributions are not limited to the modeling of extreme events. For example, they have been applied to credit default modeling \cite{beer2019pricing,mai2009tractable}, modeling of waiting times \cite{dempsey2018} but also have connections to theoretical topics such as random partitions of the natural numbers \cite{gnedin2005regenerative}.

Regarding the analytical treatments of $H$ and $\bmX$, we characterize the L\'evy measures of nnnd c\`adl\`ag id-processes as precisely those which are concentrated on nnnd paths and characterize the exponent measures of exchangeable min-id sequences $\bmX$ as precisely those that are the sum of some simple ``drift measure'' and a (possibly infinite) mixture of product probability measures. The following diagram summarizes our findings in a nutshell: 
\begin{figure}[H]
    \centering
\begin{tikzpicture}[scale=0.75]
\draw (1.75,8.4) node[right=1pt] {$\bmX=(X_1,X_2,\ldots )$};
\draw (10,8.4) node[right=1pt] {$(H_t)_{t\in\RR}$};
\path[draw=black,dashed](0.0,-0.1) rectangle (6.1,3.1);
\path[draw=black](3.0,0.0) rectangle (6.0,3.0);
\path[draw=black](8.0,0.0) rectangle (11.0,3.0);
\path[draw=black,dashed](7.9,-0.1) rectangle (14,3.1);
\path[draw=black](8.0,5.0) rectangle (11.0,8.0);
\path[draw=black,dashed](7.9,4.9) rectangle (14,8.1);
\path[draw=black,dashed](0.0,4.9) rectangle (6.1,8.1);
\path[draw=black](3.0,5.0) rectangle (6.0,8.0);

\draw[<->,dashed] (1.5,3.1) -- (1.5,4.9);
\draw[<->,dashed] (12.5,3.1) -- (12.5,4.9);
\draw[<->] (4.5,3) -- (4.5,5);
\draw[<->] (9.5,3) -- (9.5,5);
\draw[<->] (6,1.5) -- (8,1.5);
\draw[<->] (6,6.5) -- (8,6.5);

\draw (0.7,1.6) node[right=1pt] {\footnotesize Exponent};
\draw (0.7,1.2) node[right=1pt] {\footnotesize measure};
\draw (3.25,2.0) node[right=1pt] {\footnotesize Mixture of};
\draw (3.25,1.6) node[right=1pt] {\footnotesize product };
\draw (3.25,1.2) node[right=1pt] {\footnotesize probability};
\draw (3.25,0.8) node[right=1pt] {\footnotesize measures};
\draw (8.15,1.6) node[right=1pt] {\footnotesize Supported on};
\draw (8.15,1.2) node[right=1pt] {\footnotesize nnnd paths};
\draw (11.6,1.6) node[right=1pt] {\footnotesize L\' evy};
\draw (11.6,1.2) node[right=1pt] {\footnotesize measure};
\draw (0.7,6.6) node[right=1pt] {\footnotesize Min-id};
\draw (0.7,6.2) node[right=1pt] {\footnotesize sequence};
\draw (3.35,6.6) node[right=1pt] {\footnotesize Exchangeable};
\draw (9.0,6.6) node[right=1pt] {\footnotesize nnnd };
\draw (11.6,6.6) node[right=1pt] {\footnotesize c\`adl\`ag};
\draw (11.6,6.2) node[right=1pt] {\footnotesize id-process};

\draw (6.1,1.7) node[right=1pt] {\footnotesize Thm \ref{theoremexponentmeasureisiidmixture}};
\draw (6.1,6.7) node[right=1pt] {\footnotesize Thm \ref{theoremchroncorrespondstominiddist}};
\draw (1.2,3.0) node[right=1pt,rotate=90] {\tiny \cite{vatanmaxidinfinitedim}};
\draw (4.2,3.15) node[right=1pt,rotate=90] {\footnotesize Thm \ref{theoremexponentmeasureisiidmixture}};
\draw (9.2,3.05) node[right=1pt,rotate=90] {\footnotesize Prop \ref{propositionltchronometer}};
\draw (12.2,2.8) node[right=1pt,rotate=90] {\tiny \cite{rosinskiinfdivproc}};
\end{tikzpicture}
\caption{Correspondences of exchangeable min-id sequences (top, left), exponent measures (bottom, left), non-decreasing and non-negative c\`adl\`ag id-processes (top, right) and L\'evy measures (bottom, right).}
\label{Venndiagcorrespondencesminididexpmeasurelevymeasure}
\end{figure}
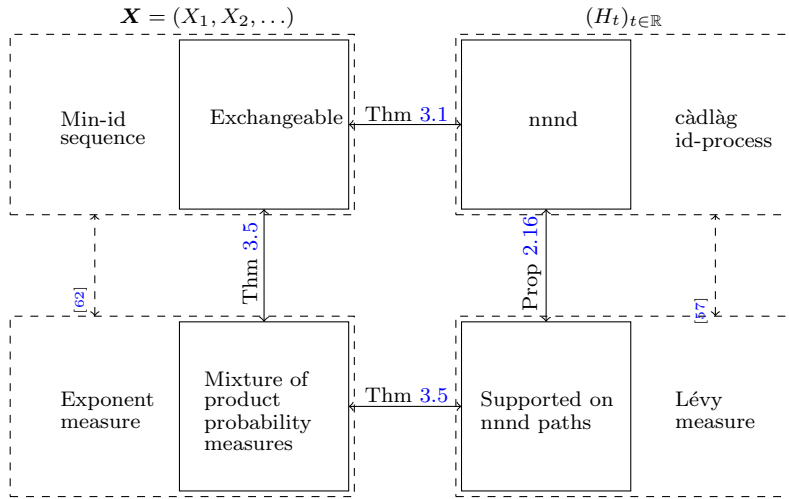

For many subfamilies of id-processes, there exist well-established theories and applications on the stochastic process level. The relation between the nnnd instances of these subfamilies with the multivariate probability laws of $\bmX$ via de Finetti's Theorem has been explored in several previous articles, which are unified and extended by the present work. Firstly, if $H$ is a non-decreasing L\'evy process (aka L\'evy subordinator), the finite-dimensional distributions of the corresponding sequence $\bmX$ are so-called Marshall--Olkin distributions, a result first found in \cite{maischererreparammo}, re-discovered and further explored in \cite{sun2017marshall}. Secondly and slightly more general, if $H$ is a nnnd additive process (aka additive subordinator), the survival function defined by $t \mapsto \exp(-H_t)$ is called a neutral-to-the-right prior in non-parametric Bayesian statistics. The resulting non-parametric Bayesian estimation techniques are explored, e.g., in \cite{kalbfleisch1978,hjort1990,james2009,james2005,epifani2003,regazzini2003}, with the most prominent representative being the Dirichlet process developed by \cite{ferguson1973}. The finite-dimensional distributions of the associated exchangeable sequence $\bmX$ are shown to correspond to exogenous shock models in \cite{MaiSchenkSchererexchangeableexshock,Sloot2020b}. A special case of particular interest is obtained in case $H$ is a Sato subordinator, leading to a characterization of self-decomposability on the half-line in terms multivariate distribution functions in \cite{Maiselfdecomp}. Thirdly, if $H$ is non-decreasing and strongly infinitely divisible with respect to time (aka time stable/strong-idt), see \cite{koppmolchanov,mansuy2005processes,hakassou2013idt}, the finite-dimensional distributions of $\bmX$ are shown to be min-stable in \cite{MaiSchererexMSMVE,maicanonicalspecrepofstabtail}. The following diagram summarizes the correspondences.
\begin{figure}[H]
    \centering
\begin{tikzpicture}[scale=0.74]
\path[draw=black](9.0,0.0) rectangle (16.0,7.0);
\path[draw=black](0.0,0.0) rectangle (7.0,7.0);
\draw (7.15,6.5) node[right=1pt] {\tiny Thm \ref{theoremchroncorrespondstominiddist}};
\draw[<->] (5.5,6.3) -- (10.8,6.3);
\path[draw=green,fill=green!80!black,opacity=0.5](11,2) rectangle (14.3,3.0);
\path[draw=green,fill=green!80!black,opacity=0.5](2,2) rectangle (5.3,3.0);
\path[thick,draw=yellow,fill=yellow!80!black,opacity=0.5](9.8,1.0) rectangle (14.3,4.0);
\path[thick,draw=yellow,fill=yellow!80!black,opacity=0.5](0.8,1.0) rectangle (5.3,4.0);
\path[draw=red,thick,fill=red!80!black,opacity=0.5](11.0,3.0) rectangle (15.0,6.0);
\path[draw=red,thick,fill=red!80!black,opacity=0.5](2.0,3.0) rectangle (6.0,6.0);
\draw (11.5,4.8) node[right=1pt] {\footnotesize subordinator};
\draw (11.5,5.2) node[right=1pt] {\footnotesize Strong-idt};
\draw (3.0,5.0) node[right=1pt] {\footnotesize Min-stable};
\draw (7.0,5.2) node[right=1pt] {\tiny \cite{maicanonicalspecrepofstabtail}};
\draw (6.2,4.8) node[right=1pt] {\tiny  \cite{MaiSchererexMSMVE}};
\draw[<->] (5.5,5.0) -- (11.2,5.0);
\draw (0.63,1.5) node[right=1pt] {\footnotesize Exogenous shock model};
\draw (10.8,1.6) node[right=1pt] {\footnotesize Additive};
\draw (10.8,1.2) node[right=1pt] {\footnotesize subordinator};
\draw (7.0,1.7) node[right=1pt] {\tiny  \cite{Sloot2020b}};
\draw (6.7,1.3) node[right=1pt] {\tiny  \cite{MaiSchenkSchererexchangeableexshock}};
\draw[<->] (5.0,1.5) -- (10.8,1.5);
\draw (11,2.5) node[right=1pt] {\footnotesize Sato subordinator };
\draw (2.5,2.5) node[right=1pt] {\footnotesize  H-Sato};
\draw (6.7,2.7) node[right=1pt] {\tiny  \cite{Maiselfdecomp}};
\draw[<->] (5.0,2.5) -- (11.2,2.5);
\draw (9,0.5) node[right=1pt] {\footnotesize Extended chronometer};
\draw (0,0.5) node[right=1pt] {\footnotesize Min-id};
\draw (11,3.5) node[right=1pt] {\footnotesize L\'evy subordinator};
\draw (2.0,3.5) node[right=1pt] {\footnotesize  Marshall--Olkin};
\draw (7.0,3.7) node[right=1pt] {\tiny \cite{janphd}};
\draw (6.4,3.3) node[right=1pt] {\tiny \cite{MaiSchererextendibilityMO}};
\draw[<->] (5.0,3.5) -- (11.2,3.5);
\draw (9.7,7.3) node[right=1pt] {$(H_t)_{t\in\RR}$ nnnd and c\`adl\`ag};
\draw (0.0,7.3) node[right=1pt] {$\mathbf{X}=(X_1,X_2,\ldots )$ exchangeable};
\end{tikzpicture}
\caption{Embedding of established correspondences of nnnd c\`adl\`ag processes and exchangeable sequences into the present framework.}
\label{Venndiagprocessvssequence}
\end{figure}
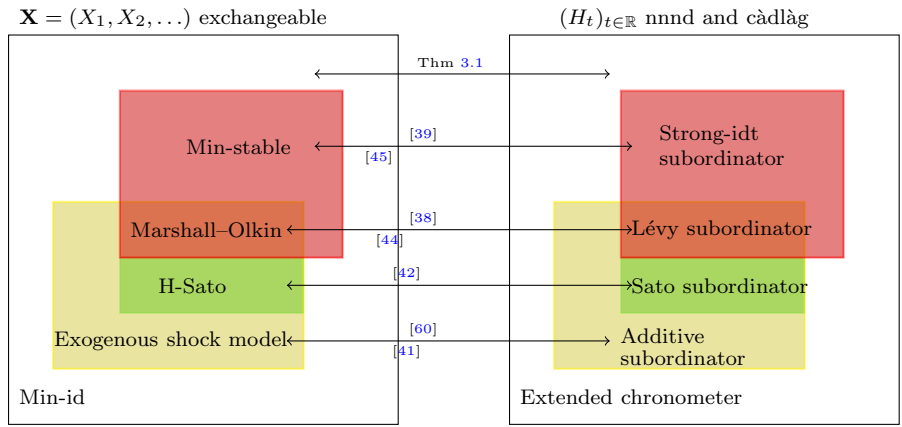

It is worth noting that our framework is general enough to include non-continuous min-id distributions, which correspond to id-processes $H$ that are not stochastically continuous. While this stands in glaring contrast to most of the aforementioned references and might on first glimpse be accompanied by technical problems, our derivations show that a distinction between stochastically continuous and non-continuous processes is not a crucial technical obstacle, but rather a distinctive feature to study after a general theory is established.

\subsection*{Structure of the manuscript}

In Section~\ref{sectionmathprelim}, we recall the most important results about min-id distributions and characterize their exchangeable subclass. Furthermore, we characterize nnnd infinitely divisible stochastic processes by their L\'evy measure and drift. As a byproduct, we show that every c\`adl\`ag id-process can be represented as the sum of i.i.d.\ c\`adl\`ag processes, which solves an open problem posed by \cite{BarndorffNielsen2006InfiniteDF}. Section~\ref{sectionmainresults} provides the main contributions of this work. First, we show that each exchangeable min-id sequence uniquely corresponds to an nnnd infinitely divisible c\`adl\`ag process. Second, we show that the exponent measure of an exchangeable min-id sequence is given by the sum of a simple drift measure and a mixture of product probability measures. In Section~\ref{sectionexamples} we present some important examples of exchangeable min-id sequences and embed them into our framework. A summary of the key findings is given in Section~\ref{sectionconclusion}. All non-trivial proofs are deferred to Appendix \ref{appendixproofs}.

\section{Preliminaries}
\label{sectionmathprelim}

\subsection{Notation}
The following notation is used throughout this paper. A topological space $(S,\tau)$ is always equipped with its Borel $\sigma$-algebra, denoted as $\bc(S)$. Upon existence, we frequently use the notation $0_S$ to refer to the neutral element w.r.t.\ addition in $S$. The letters $\RR$, $\bR$, $\QQ$, $\N$ and $\N_0$ denote the real numbers, (extended) real numbers including $\pm \infty$, rational numbers, natural numbers and natural numbers including $0$, all equipped with their standard topologies. For $A\subset\bR$, $A^\N$ denotes the space of $A$-valued sequences equipped with the product (subspace-)topology. Vectors and sequences are written in bold letters to distinguish them from scalars. The symbols $>, \geq , <$ and $\leq$ are understood componentwise, e.g.\ for $\bm x, \bm y \in\bR^d$, $\bm x< \bm y$ means that $x_i<y_i$ for all $1\leq i\leq d$. Similarly, for every $\bmx,\bm y\in \bR^d$, the operators $\max$ (resp.\ $\min$)$\{\bm x,\bm y\}$ are applied componentwise. For every set $A\subset \bR^d$ we define $\max$ (resp.\ $\min$, $\sup$, $\inf$) $A$ as the componentwise maximum (resp. minimum, supremum, infimum) of elements in $A$, where $\min \emptyset:=\inf \emptyset:=\binfty$ and $\max\emptyset:=\sup \emptyset:=-\binfty$. Moreover, for any $\bm a,\bm b\in\bR^d$ we define $[\bm a,\bm b]:=\times_{i=1}^d [a_i,b_i]$, where $[a_i,b_i]$ denotes a closed interval. The obvious modifications apply to $(\bm a,\bm b)$, $[\bm a,\bm b)$ and $[\bm a,\bm b)$. The function $\id_A (\bmx)$ denotes the indicator function of a set $A$.

$D(T)$ (resp.\ $D^\infty(T)$) denotes the space of real-valued (resp.\  extended real-valued) c\`adl\`ag functions, i.e.\ right-continuous functions with left limits, which are indexed by a set $T\subset \RR$.  $D^\infty(T)$ is always equipped with the (Borel) $\sigma$-algebra generated by the finite dimensional projections, i.e.\ $\bc(D^\infty (T)):=\sigma \big( \big\{ \{ x\in D^\infty (T)\mid (x(t_1),\ldots, x(t_d))\in A\},(t_i)_{1\leq i\leq d}\subset T,$ $ A\in\bc\big(\bR^d\big), d\in\N \big\}\big)$. Note that this $\sigma$-algebra is a Borel $\sigma$-algebra, since it can be generated as the Borel $\sigma$-algebra of a topology on $D^\infty(T)$. Similarly, $D(T)$ is equipped with the subspace (Borel) $\sigma$-algebra $\bc(D(T)):=\bc(D^\infty(T))\cap D(T)$. The function $0_{D^\infty(T)}:=(0)_{t\in T}$ denotes the function which vanishes everywhere. A c\`adl\`ag process indexed by $T$ denotes a random element $H\in D^\infty(T)$. Sometimes $H$ is also referred to as $(H_t)_{t\in T}$ to emphasize the stochastic process character. $H_t$ refers to the extended real-valued random variable obtained by projecting $H$ at ``time'' $t$.

We write $X\sim Y$ to denote that two random elements $X$ and $Y$ are identical in distribution, even though $X$ and $Y$ do not need to be defined on the same probability space. If the probability space is not explicitly specified, we adopt the usual notation and denote the probability measure as $\pp$. The distribution function $F$ of a random vector $\bmX_d\in\bR^d$ is defined as $F(\bmx):=\pp\lc \bmX_d\leq \bmx \rc $. The survival function $\fb$ of $\bmX_d$ is defined as $\fb(\bmx):=\pp\lc \bmX_d\in \times_{i=1}^d \{x_i,\infty]\rc$, where ``$\{$'' is interpreted as ``$($'' if $x_i>-\infty$ and ``$\{$'' is interpreted as ``$[$'' if $x_i=-\infty$. We frequently write $\bmX_d\sim F$ (resp.\ $\bmX_d\sim\fb$) to denote that the random vector $\bmX_d$ has distribution (resp.\ survival) function $F$ (resp.\ $\fb$). A random variable has exponential distribution with mean $1/\lambda\geq 0$ if it has survival function $\fb(x)=\exp\lc-\lambda x\id_{[0,\infty)}(x)\rc$.
The terms min- (resp.\ max-)id random vector and min- (resp.\ max-)id distribution will be used synonimously, depending on whether we refer to a random vector or its associated distribution. 

In slight abuse of the common terminology we say that a measure $\mu$ is supported on a set $A\subset \bc\lc \bR^d\rc$ if $\mu\lc \bR^d\setminus A\rc=0$.

\subsection{Exchangeable min- and max-id distributions}
\label{subsectionminidmaxiddistributions}
First, let us recall the definition of exchangeable and extendible random vectors and exchangeable sequences.

\begin{defn}[Exchangeable and extendible random vector/sequence] 
\label{definitionexchangeableextendible}
\begin{enumerate}
    \item A random vector $\bmX_d=(X_{d,1},\ldots,X_{d,d})\in\bR^d$ is exchangeable if $\bmX_d\sim \big(X_{d,\pi(1)},\ldots$ $,X_{d,\pi(d)}\big)$ for all permutations $\pi$ on $\{1,\ldots,d\}$. A random sequence $\bmX=\lc X_i\rc_{i\in\N}\in\bR^\N$ is exchangeable if all its finite dimensional marginal distributions are exchangeable.
    \item An exchangeable random vector $\bmX_d=(X_{d,1},\ldots,X_{d,d})\in\bR^d$ is extendible if there exists an exchangeable sequence $\bmX=\lc X_i\rc_{i\in\N}\in\bR^\N$ (possibly defined on a different probability space) satisfying $\bmX_d\sim ( X_1,\ldots, X_d)$. 
\end{enumerate}
\end{defn}
Extendible random vectors form a proper subclass of exchangeable random vectors, since there exist many exchangeable random vectors which are not extendible. For example, consider a random vector $\bmX_2\in\RR^2$ which follows a bivariate normal distribution with negative correlation.\footnote{Another example being an Archimedean copula with a $d$-monotone generator that is not completely monotone.} \cite[p.\ 7]{aldousexchangeable} shows that exchangeable sequences of random variables necessarily have non-negative correlation. Therefore, $\bmX_2$ is exchangeable but not extendible to an exchangeable sequence. It is also worth noting that an extendible random vector $\bmX$ may be extendible to more than one exchangeable sequence. On the other side, de Finetti's Theorem provides a unique stochastic representation of exchangeable sequences and thus also characterizes extendible random vectors.

\begin{thm}[De Finetti {\cite[Chapter 1.3]{aldousexchangeable}}]
A sequence $\bm X=(X_i)_{i\in\N} \in\bR^\N$ is exchangeable if and only if there exists an nnnd stochastic process $H\in D^\infty(\RR)$ such that 
\begin{align}
    (X_i)_{i\in\N}\sim \big( \inf\{t\in\RR\mid H_t\geq E_i\}\big)_{i\in\N},  \label{constructionexchangeseq}
\end{align}
where $(E_i)_{i\in\N}$ is a sequence of i.i.d.\ unit exponential random variables independent of $H$. 
$\bmX$ has survival function
$$  \pp\lc X_1>t_1,X_2>t_2,\ldots \rc=\ee\lk \prod_{i\in\N} e^{-H_{t_i}} \rk \ , \ \bmt=(t_1,t_2,\ldots)\in[-\infty,\infty)^\N .$$
Moreover, the distribution of the process $H$ is uniquely determined by the distribution of $\bmX$ and vice versa.
\end{thm}

De Finetti's Theorem can be refined for exchangeable sequences with continuous marginal distribution.

\begin{cor}[Exchangeable sequences with continuous marginal distribution correspond to stochastically continuous c\`adl\`ag processes]
\label{corcontdistrcorresstochcontchron}
The sequence $\bmX\in \bR^\N$ from Equation (\ref{constructionexchangeseq}) has continuous marginal distributions if and only if $H$ is stochastically continuous.
\end{cor}

\begin{proof}
The proof can be found in Appendix \ref{appendixproofs}.
\end{proof}

Next, we recall the most important results about min- and max-id distributions. We mainly follow \cite[Chapter 5]{resnickextreme} and translate the results from max-id random vectors on $\RR^d$ to min-id random vectors on $\minf^d$. Many of these translations are straightforward, but some lurking technical subtleties need to be considered and are emphasized in the upcoming paragraphs. We start with the formal definition of min- (resp.\ max-)id random vectors/sequences. There are at least two equivalent definitions that will be used frequently throughout the paper. Therefore, they are presented jointly.

\begin{defn}[Min-id distribution / Max-id distribution]
\label{definitionmaxidminid}

\begin{enumerate}
    \item A random vector $\bmX_d\in\bR^d$ is min-infinitely divisible (min-id) if for every $n\in\N$ there exist i.i.d.\ random vectors $(\bmX_d^{(i,1/n)})_{\leqn}$\footnote{The notation $\bmX_d^{(i,1/n)}$ shall emphasize that every $\bmX_d^{(i,1/n)}$ can be interpreted as a $1/n$ contribution to $\bmX$, since the $\bmX_d^{(i,1/n)}$ are all equally likely to contribute $\min_{\leqn} \bmX_d^{(i,1/n)}\sim\bmX_d$} such that $\bmX_d\sim \min_{\leqn} \bmX_d^{(i,1/n)}$. Equivalently, the survival function $\fb$ of a random vector $\bmX_d\in\bR^d$ is min-id if $\fb^{t}$ is a survival function for every $t>0$.
    \item A random vector $\bmX_d\in\bR^d$ is max-infinitely divisible (max-id) if for every $n\in\N$ there exist i.i.d.\ random vectors $(\bmX_d^{(i,1/n)})_{\leqn}$ such that $\bmX_d\sim \max_{\leqn} \bmX_d^{(i,1/n)}$. Equivalently, the distribution function $F$ of a random vector $\bmX_d\in\bR^d$ is max-id if $F^{t}$ is a distribution function for every $t>0$.
\end{enumerate}
Similarly, a sequence $\bmX\in \bR^\N$ is called min- (resp.\ max-)id if $(X_{i_1},\ldots,X_{i_d})$ is a min- (resp.\ max-)id random vector for every $(i_1,\ldots,i_d)\in \N^d$.

\end{defn}
Every univariate random variable $X\in\bR$ is min- and max-id. However, already for $d\geq 2$ it is certainly not trivial to decide whether a given random vector or survival (resp.\ distribution) function is min- (resp.\ max-)id.  For example, a bivariate normal distribution with negative correlation is not min- (resp.\ max-)id, since min- (resp.\ max-)id random vectors necessarily exhibit non-negative correlation \cite[Proposition 5.29]{resnickextreme}.

An important property of the class of min- and max-id distributions is their distributional closure under monotone marginal transformations. This fact is summarized in the following lemma, which is a slight extension of \cite[Proposition 5.2 iii)]{resnickextreme}.

\begin{lem}[Class of min- and max-id distributions is closed under monotone transformations]
\label{lemmatrafomaxtominid}
Let $\bmX_d\ (resp.\ \bm Y_d)\in\bR^d$ denote a min- (resp.\ max-)id random vector. The following statements are valid.
\begin{enumerate}
\item Let $(f_i)_{1\leq i \leq d}: \bR\to \bR$ be non-decreasing. Then $f(\bmX_d):=\big(f_1(X_1),\ldots,f_d(X_d)\big)$ $\big(resp.\ $ $ f(\bm Y_d):=\big(f_1(Y_1),\ldots,f_d(Y_d)\big)\big)$ is min- (resp.\ max-)id.
\item Let $(f_i)_{1\leq i \leq d}: \bR\to \bR$ be non-increasing. Then $f(\bmX_d)$ $(resp.\ f(\bm Y_d))$ is max- (resp.\ min-)id.
\end{enumerate}
\end{lem}
%\begin{proof}
%The proof is similar to the proof of \cite[Proposition 5.2 iii)]{resnickextreme}.
%\begin{enumerate}
%\item Follows from \cite[Proposition 5.2 iii)]{resnickextreme}.
%\item It suffices to show that there exist some i.i.d.\ random vectors $(Z_i^{(1/n)})_{1\leq i\leq n}$ such that $f(\bmX)\sim \max_{1\leq i\leq n} Z_i$. Since $\bmX$ is min-id, there exist $(\tilde{Z}_i^{(1/n)})_{1\leq i\leq n}$ such that $\bmX\sim \min_{1\leq i\leq n}\tilde{Z}_i^{(1/n)}$. This implies that $f(\bmX)\sim f\lc \min_{1\leq i\leq n}\tilde{Z}_i^{(1/n)}\rc=\max_{1\leq i\leq n}f\lc\tilde{Z}_i^{(1/n)}\rc$. Defining $Z_i:=f\lc\tilde{Z}_i^{(1/n)}\rc$ yields the claim. The fact that $f(\bm Y)$ is min-id follows analogously.
%\end{enumerate}
%\end{proof}

Lemma~\ref{lemmatrafomaxtominid} shows that studying the entire class of min- and max-id distributions is equivalent to studying min-id distributions supported on an arbitrary non-empty subset of $A\subset\bR^d$, since for every min- or max-id random vector $\bmX_d$ there exist some strictly monotone functions $\lc f_i\rc_{1\leq i\leq d}$ such that $f(\bmX_d)\in A$ is min-id. A convenient choice for our analysis is the set $A=\minf^d$. This choice allows us to assume that a min-id distribution does not have mass on $\{\bmx\in\bR^d\mid x_i=-\infty  \text{ for some }1\leq i\leq d\}$. Such min-id distributions are uniquely determined by their survival function restricted to $\RR^d$. Thus, we can avoid the technical subtleties involving survival functions on $\bR^d$. The convenience of this seemingly artificial fact will become clear in Section~\ref{sectionextminidcorrespondstoextchron}, since the choice $\bmX\in\minf^\N$ allows us to restrict our attention to c\`adl\`ag processes vanishing at $-\infty$. Therefore, without loss of generality, we only consider min-id random vectors on $\minf^d$ in the remainder of the paper, if not explicitly mentioned otherwise.

Similar to \cite[Proposition 5.8]{resnickextreme}, we can characterize min-id distributions on $\minf^d$ by a so-called exponent measure. To define the exponent measure of a min-id distribution we need to introduce some notation. For $\bell  \in(-\infty,\infty]^d$ define the set $E^d_\bell :=[-\binfty,\bell]\setminus \{\bell\}$, which is equipped with the subspace topology inherited from $\bR^d$. Furthermore, for $\bmx<\bell$, define $(\bmx,\binfty]^\complement:=E^d_\bell \setminus (\bmx,\binfty]$. 

\begin{defn}[Exponent measure of min-id distributions]
\label{definitionexponentmeasure}
A Radon measure $\mu_d$ on $E^d_\bell$ is called exponent measure if it satisfies
$$\mu_d\lc \bigcup_{i=1}^d \big\{\bmx\in E^d_\bell \mid x_i=-\infty\big\}\rc=0.$$
\end{defn}

This definition of an exponent measure ensures that $\bmX_d\sim \exp\lc-\mu_d\lc (\cdot,\binfty]^\complement\rc\rc$ is min-id on $(-\infty,\infty]^d$, since \begin{align*}
    &\pp(X_i=-\infty \text{ for some }1\leq i\leq d)\\
    &=1-\exp\lc- \mu_d\lc \bigcup_{i=1}^d \{\bmx\in E^d_\bell \mid x_i=-\infty\}\rc\rc= 0.
\end{align*} 
The next proposition, which is similar to \cite[Proposition 5.8]{resnickextreme}, shows that every min-id random vector has a survival function of the form $\exp\lc -\mu_d(\cdot,\binfty]^\complement\rc$.

\begin{prop}[Characterization of min-id distributions]
\label{propositionexponentmeasureofminid}
The following are equivalent:
\begin{enumerate}
\item $\bmX_d\in \minf^d$ is min-id.
\item There exist $\bell\in \minf^d$ and an exponent measure $\mu_d$ defined on $E^d_\bell$ such that 
$$\bmX_d\sim \fb(\bmx)=\begin{cases}
\exp\bigg(-\mu_d\bigg( (\bmx,\binfty]^\complement\bigg)\bigg) & \bmx<\bell, \\
0 & else.
\end{cases} $$
\end{enumerate}
Moreover, the exponent measure $\mu_d$ associated with $\bmX_d$ is unique.
\end{prop}

\begin{proof}
The proof is a translation of the proof of \cite[Proposition 5.8]{resnickextreme} to the min-id case.
\end{proof}

We have already mentioned that it is sufficient to restrict our study of min-id $\bmX_d\in\bR^d$ to that of min-id $\bmX_d\in(-\infty,\infty]^d$. Since $E^d_\bell$ is almost rectangular, we can restrict the study of min-id distributions on $\minf^d$ to an even smaller and more convenient class of min-id distributions, because Lemma \ref{lemmatrafomaxtominid} allows to transform every min-id random vector $\bmX_d$ to a min-id random vector $\tilde{\bmX}_d$ which satisfies $\bell=\times_{i=1}^d \{ \infty\}$. Thus, from now on, we always assume that a min-id random vector $\bmX_d$ satisfies the following condition, if not explicitly mentioned otherwise.

\begin{cond}
\label{conditionrestrictionminiddistr}
A min-id random vector $\bmX_d\in\bR^d$ satisfies Condition $(\Diamond)$ if the following two assumptions are satisfied
\begin{enumerate}
    \item[$(\Diamond 1)$] $\bmX_d\in\minf^d$ and 
    \item[$(\Diamond 2)$] $\bell=\binfty$.
\end{enumerate}
\end{cond}

% \begin{proof}
% As already mentioned Lemma~\ref{lemmatrafomaxtominid} ensures that every min-id random vector $\bmX_d\in\bR^d$ can be transformed to a min-id random vector $f(\bmX_d):=\lc f_1(X_1),\ldots,f_d(X_d)\rc\in(-\infty,\infty]^d$ via strictly non-decreasing transformations $(f_i)_{1\leq i\leq d}$. Proposition~\ref{propositionexponentmeasureofminid} shows that we can w.l.o.g.\ choose $f_i$ such that $f_i(\bell_i)=\infty$ and that $f_i(-\infty)=a>-\infty$. In this case we obtain a min-id random vector $f(\bmX_d)$ with exponent measure $\mu_d$ supported on $E_\binfty^d$. Therefore, $f(\bmX_d)$ satisfies Conditions $(\Diamond 1)$ and $(\Diamond 2)$. Now, results about $\bmX_d$ are obtained via the image measure of $\bmX_d$ under the transformation $f$. 
% \end{proof}

The purpose of Condition $(\Diamond)$ is to simplify our exposition in the remainder of the paper. It ensures that we can always assume that a min-id random vector $\bmX_d$ satisfies $\pp(\bmX_d>\bmx)>0$ for all $\bmx\in\RR^d$ and $\pp(X_i>-\infty )=1$ for all $1\leq i\leq d$. These constraints allow to avoid splitting the proofs of our main theorems into several cases.
%The following example illustrates the use of Condition $(\Diamond2)$.

% \begin{ex}[A constant vector]
% Consider the min-id random vector $\bmX^{(a)}=(a,\ldots,a )$ $\in \bR^d$ for some $a\in\bR$. If $a<\infty$, $\bmX^{(a)}$ does not satisfy Condition $(\Diamond 2)$. The exponent measure $\mu$ of $\bmX^{(a)}$ is given by the zero measure on $E_{\bm a}$. Obviously, $\bmX^{(a)}$ is extendible and its de Finetti representation is given by $\bmX^{(a)}\sim  \lc\inf\{t\in\RR\mid H^{(a)}_t>E_i\}\rc_{1\leq i\leq d}$, where $(E_i)_{1\leq i\leq d}$ are i.i.d.\ unit exponential random variables and $H^{(a)}_t=\infty\id_{\{t\geq a\}}$. Theoretically, it suffices to consider only one representative of the family $(\bmX^{(a)})_{a\in\bR}$. Condition $(\Diamond)$ ensures that we consider exactly one representative of the family $(\bmX^{(a)})_{a\in\bR}$, namely $\bmX^{(\infty)}$. Moreover, explicitly picking representative $\bmX^{(\infty)}$ will allow us to avoid some technical subtleties in Sections~\ref{subsectionextendedchron} and~\ref{sectionmainresults} involving processes which are almost surely infinite, since $H^{(\infty)}=0$ is the only process which satisfies this constraint.
% \end{ex}

Up to this point we have mainly introduced and reformulated existing results. In the following, we characterize exchangeable and extendible min-id random vectors, which have not yet been discussed in the literature. 

A natural question is whether the exchangeability (resp.\ extendibility) of the min-id random vector $\bmX_d$ is equivalent to the exchangeability (resp.\ extendibility) of its associated exponent measure $\mu_d$. To this purpose, let us properly define exchangeability of an exponent measure. For any $\bmx\in\bR^d$ and permutation $\pi$ on $\{1,\ldots,d\}$ define $\pi(\bmx)=(x_{\pi(1)},\ldots,x_{\pi(d)})$. Similarly, for any set $A\subset \bR^d$, define $\pi(A):=\{\pi(\bmx)\mid \bmx\in A\}$. An exponent measure $\mu_d$ on $E^d_\binfty$ is called exchangeable if $\mu_d(A)=\mu_d(\pi(A))$ for every $A\in \bc\lc E^d_\binfty\rc$ and every permutation $\pi$ on $\{1,\ldots,d\}$.

\begin{prop}[Exchangeable exponent measure]
\label{propositionexchangeabilityexponentmeasure}
The following are equivalent for every min-id random vector $\bmX_d\in\minf^d$:
\begin{enumerate}
\item $\bmX_d$ is exchangeable.
\item $\mu_d$ is exchangeable. 
\end{enumerate}
\end{prop}

It may seem obvious that $\mu_d$ is exchangeable if $\bmX_d$ is exchangeable, since Proposition~\ref{propositionexponentmeasureofminid} implies exchangeability of $\mu_d$ on sets of the type $(\bmx,\binfty]^\complement$. However, $\mu_d$ may have infinite mass and exchangeability on sets of the type $(\bmx,\binfty]^\complement$ is not a sufficient criterion for exchangeability of infinite measures in general.

\begin{proof}
The proof can be found in Appendix \ref{appendixproofs}.
\end{proof}

Proposition~\ref{propositionexchangeabilityexponentmeasure} shows that the exchangeability of $\bmX_d$ is in one-to-one correspondence with the exchangeability of $\mu_d$. A similar statement holds for the extendibility of $\bmX_d$, but the precise formulation of this result is tedious. The problem arises from the definition of $E^d_\binfty$, since $E^d_\binfty$ is not a product space and this may lead to projections to the point $\binfty\not\in E^{d^\prime}_\binfty$, $d^\prime< d$. 
Vatan has investigated this problem in the context of exponent measures of max-id sequences in \cite{vatanmaxidinfinitedim}. Unfortunately, his definition of an exponent measure slightly differs from the original definition in \cite{resnickextreme} and does not directly translate to Definition~\ref{definitionexponentmeasure}. Vatan puts infinite mass on the lower boundary $-\bell$ of the support of the exponent measure of a max-id sequence $-\bmX$ to ensure that exponent measures are projective. Under this constraint he proved the existence of a unique projective exponent measure on $[-\bell,\binfty]:=\times_{i\in\N}[-\ell_i,\infty]$ associated with $-\bmX$. Removing the point $-\bell$ from the support of Vatan's exponent measure allows us to translate his results to our setting and we obtain the exponent measure of a min-id sequence $\bmX$. Note that this modified exponent measure is generally not projective, which is due to the removal of $-\bell$. We denote this global exponent measure on $E^\N_\bell:=[-\binfty,\bell]\setminus \{\bell\}$ by $\mu$. 
Similar to the study of min-id random vectors we can restrict the study of min-id sequences to min-id sequences whose $d$-dimensional margins satisfy Condition $(\Diamond)$. Thus, from now on we will assume that the $d$-dimensional margins of a min-id sequence satisfy Condition $(\Diamond)$, i.e.\ $\bmX\in(-\infty,\infty]^\N$ and $\bell=\binfty$.\\
Define $\mu_{i_1,\ldots,i_d}$ as the unique exponent measure of the $d$-dimensional margin $(X_{i_1},\ldots,X_{i_d})$ of the min-id sequence $\bmX$. The results of \cite{vatanmaxidinfinitedim} are summarized in the following proposition and clarify the projective properties of extendible exponent measures $\mu_d$. 
%Basically, we have to restrict the set $A$ from Example~\ref{exampleprojtoinf} to be in the ``right'' $\sigma$-algebra $\bc\lc E^{d^\prime}_\binfty\rc$ instead of $\bc\lc \minf^{d^\prime}\rc$.justifies to write ``the'' exponent measure of an extendible min-id sequence. 
%In line with the notation in the proof of Proposition~\ref{propositionextendibleexponentmeasure} we 

%The following example illustrates the technical difficulties inherent in the definition of extendible exponent measures.

%To conclude this subsection we extend the notion of exponent measures of min-id random vectors to exponent measures of min-id sequences. 

\begin{prop}[\cite{vatanmaxidinfinitedim}, Exponent measure of a sequence]
\label{propositionextendibleexponentmeasure}
The following are equivalent for every min-id random vector $\bmX_d\in\minf^d$ satisfying Condition $(\Diamond)$:
\begin{enumerate}
\item $\bmX_d$ is extendible to a min-id sequence $\bmX$.
\item There exists a unique exchangeable global exponent measure $\mu$ on $E^\N_\binfty$ such that
$$ \mu\lc \{\bmx\in E^\N_\binfty\mid (x_{i_1},\ldots,x_{i_d})\in A\}\rc= \mu_{i_1,\ldots,i_d}(A)=\mu_d(A), $$
for all distinct $(i_1,\ldots ,i_d)\in \N^d$ and $A\in \bc\lc E^d_\binfty\rc$. 
%$\mu_d$ is extendible, i.e.\ for every $n\geq d$ there exists an exchangeable exponent measure $\mu_n$ defined on $E_\binfty^n$ satisfying $\mu_n(A\times [-\infty,\infty]^{n-d} )=\mu_d(A)$ for every $A\in \bc(E_\binfty^d)$. Moreover, 
\end{enumerate}
\end{prop}

\begin{proof}
The proof follows from an application of Proposition \ref{propositionexchangeabilityexponentmeasure} and a translation of the results of \cite{vatanmaxidinfinitedim} to the framework of \cite{resnickextreme} and the min-id case.
\end{proof}

% \begin{ex}[Projection to $\binfty$]
% \label{exampleprojtoinf}
% Recall that a proper projective exponent measure should satisfy $\mu_d\lc \{x\in E^d_\binfty \mid (x_1,\ldots,x_{d-d^\prime}) \in A\}\rc=\mu_{d-d^\prime}(A\cap E_\binfty^{d^\prime})$ for all $A\in\bc\lc \minf^{d^\prime}\rc$ and $d^\prime< d$.

% Let $d=3$ and consider $\bmX\in\RR^3$ to be i.i.d.\ with univariate survival function $\bar{G}$. Choose $a\in\RR$ and define the (two dimensional) set $A=[-\infty,a]^2\cup \{(\infty,\infty)\}$, which is not a subset of $E_\binfty^2$, because it contains the point $(\infty,\infty)$. Note that $\overline{A}:=\big( [-\infty,a]^2\times [-\infty,\infty]\big)\cup \big( \{(\infty,\infty)\}\times [-\infty,\infty)\big)$ is a subset of $E^3_\infty$ and $A$ is the projection of $\overline{A}$ to $E^2_\binfty$, i.e.\ $A=\{(x_1,x_2)\in E_\binfty^2\mid (x_1,x_2,x_3) \in \overline{A}\}$. We observe that 
% $$\mu_3(\overline{A})=\mu_3(\{\bmx\in E^3_\binfty \mid (x_1,x_2)\in A\})=\infty\not=\mu_2(A\cap E_\binfty^2)=-\log \lc \bar{G}(a)^2\rc,$$
% because the mass of $\{(\infty,\infty)\}\times [-\infty,\infty)$ is ``lost'' by the projection of $\overline{A}$ to $E^{2}_\binfty$. Thus, $\mu_3(\{x\in E^3_\binfty\mid (x_1,x_2)\in A\})$ may not be equal to $\mu_2(A\cap E^2_\binfty)$ for arbitrary sets $A\in\bc(\minf^2)$.
% \end{ex}

Propositions~\ref{propositionexchangeabilityexponentmeasure} and~\ref{propositionextendibleexponentmeasure} allow to characterize the upper tail dependence coefficients of an exchangeable min-id random vector via its exponent measure. To this purpose, for $2\leq d^\prime\leq d$, define the $d^\prime$-variate upper tail dependence coefficient of an exchangeable min-id random vector $\bmX_d\in\minf^d$ as $\rho^u_{d^\prime}:=\lim_{t\to\infty} \pp(X_1>t\mid X_2>t,\ldots ,X_{d^\prime}>t)$. 
%Analogously, we define the $d$-variate lower tail dependence coefficient of an exchangeable sequence $\bmX$ as
%$\rho^l_d:=\lim_{t\to-\inf_{t\in\RR} \pp(X_1>t)>0} \pp(X_1\leq t\mid X_2\leq t,\ldots ,X_d\leq t)$.

\begin{cor}[Upper tail dependence of exchangeable min-id distribution]
\label{cortaildependence}
The upper tail dependence coefficient of an exchangeable min-id random vector $\bmX_d\in \minf^d$ satisfying Condition $(\Diamond)$ is given by
%\begin{enumerate}
%    \item 
\begin{align*}
    \rho^u_{d^\prime}=\exp \lc -\lim_{t\to\infty} \mu_{d^\prime}\lc [-\infty,t]\times (t,\infty]^{d^\prime-1}\rc\rc.
\end{align*}
In particular $\rho^u_{d^\prime}\leq \rho^u_{d^\prime+1} $ for all $2\leq d^\prime<d$.
\end{cor}
\begin{proof}
The proof can be found in Appendix \ref{appendixproofs}.
\end{proof}
As a caveat we want to remark that in general
$\lim_{t\to\infty} \mu_{d^\prime}\lc [-\infty,t]\times (t,\infty]^{d^\prime-1}\rc$ may not be equal to $\mu_{d^\prime}\lc [-\infty,\infty)\times \{\infty\}^{d^\prime-1}\rc$, which is due to the (possibly) infinite mass of $\mu_{d^\prime}$.

% Combining the previous results we can formulate Condition $(\Diamond)$ for min-id sequences.

% \begin{cond}[Condition $(\Diamond)$ for min-id sequences]
% \label{corconditionminidsequence}
% Let $\mu$ denote the unique exponent measure of a min-id sequence $\bmX$. Without loss of generality we can assume that $\bmX\in \minf^\N$ and $\bell=\binfty$, i.e.\ $\mu$ is supported on $E^\N_\binfty$.
% \end{cond}
% \begin{proof}
% The proof is identical to the proof of Condition~\ref{conditionrestrictionminiddistr}, simply replace $d$ by $\N$.
% \end{proof}
% Note that Condition $(\Diamond)$ is equivalent to the assumption that every $d$-dimensional margin $(X_{i_1},\ldots,X_{i_d})$ of the min-id sequence $\bmX$ satisfies Condition $(\Diamond)$. 

\subsection{Extended chronometers}
\label{subsectionextendedchron}

In this subsection we characterize the class of nnnd infinitely divisible c\`adl\`ag processes. In particular, we show that the L\'evy measure of such processes is concentrated on nnnd paths. \\

First, we recall the definition of infinitely divisible random vectors and infinitely divisible stochastic processes. For the sake of well-definedness of sums of (possibly) infinite quantities we only allow for random vectors and stochastic processes which cannot assume the values infinity and negative infinity with positive probability at the same time. Our specifications are formalized in the following definition.

\begin{defn}[Infinitely divisible]
 \label{defidrvidprocess}
\begin{enumerate}
\item A random vector $\bm H\in\bR^d$ such that for all $1\leq i\leq d$ we either have $P(H_i=\infty)=0$ or $\pp(H_i=-\infty)=0$ is infinitely divisible (id) if for all $n\in \N$ there exist i.i.d.\ random vectors $(\bm H^{(i,1/n)})_{1\leq i\leq n}$ such that $\bm H \sim \sum_{i=1}^n \bm H^{(i,1/n)}$.
\item A stochastic process $(H_t)_{t\in \RR}\in\bR^\RR$ such that for all $t\in\RR$ we either have $P(H_t=\infty)=0$ or $\pp(H_t=-\infty)=0$ is infinitely divisible (id) if for all $n\in \N$ there exist i.i.d.\ stochastic processes $\lc \lc H^{(i,1/n)}_t\rc_{t\in \RR}\rc_{1\leq i \leq n} \in \lc \RR^\RR\rc^n$ such that $H \sim \sum_{i=1}^n H^{(i,1/n)}$.
\end{enumerate}
\end{defn}

An excellent textbook treatment of infinitely divisible distributions is \cite{satolevyinfinitely}. Infinitely divisible stochastic processes have been investigated, among others, by \cite{leeinfinitely,maruyamainfinitely,BarndorffNielsen2006InfiniteDF,rosinskiinfdivproc}. We mainly follow the pathwise approach of \cite{rosinskiinfdivproc}. However, in contrast to \cite{rosinskiinfdivproc}, we only consider c\`adl\`ag id-processes and allow jumps to $\infty$. One can show that all relevant results of \cite[Sections 1-3]{rosinskiinfdivproc} remain valid under this slight change of the general framework.

In this paper we focus on nnnd c\`adl\`ag id-processes. For us, this is not a loss of generality, since de Finetti's Theorem implies that $H$ in the construction method of Equation $(\ref{constructionexchangeseq})$ can be chosen as a nnnd c\`adl\`ag process. These processes can be viewed as an extension of chronometers, which were introduced in \cite{BarndorffNielsen2006InfiniteDF}.

\begin{defn}[Extended chronometer]
\label{defextchronometer}
A nnnd infinitely divisible process $(H_t)_{t\in\RR}$ $\in D^\infty(\RR)$ is called an extended chronometer.
\end{defn}

In contrast to chronometers in \cite{BarndorffNielsen2006InfiniteDF}, we do not require the stochastic continuity of extended chronometers. This generalization is necessary to account for non-continuous min-id distributions in Section~\ref{sectionextminidcorrespondstoextchron}. To simplify our developments we focus on extended chronometers that are compatible with Condition $(\Diamond)$. To this purpose we need to impose two additional constraints on extended chronometers, which are specified in the following condition.
%referred to as Condition $(\Diamond^\prime)$.
\begin{condprime}
An extended chronometer $(H_t)_{t\in\RR}$ satisfies Condition $(\Diamond^\prime)$ if the following two conditions are satisfied
\begin{enumerate}
    \item[$(\Diamond^\prime 1)$] $\pp\lc \lim_{t\to -\infty}H_t=0\rc=1$ and
    \item[$(\Diamond^\prime 2)$] $\pp(H_t=\infty)<1$ for all $t\in\RR$.
\end{enumerate} 
\end{condprime}

\begin{rem}[Equivalence of Conditions $(\Diamond)$ and $(\Diamond^\prime)$]
Consider an extended chronometer $(H_t)_{t\in\RR}$ and an i.i.d.\ unit exponential sequence $(E_i)_{i\in\N}$. By virtue of Equation (\ref{constructionexchangeseq}) we construct an exchangeable sequence $\bmX:=\lc \inf \{t\in\RR\mid H_t\geq E_i\}\rc_{i\in\N}$. It is easy to see that $\bmX>-\binfty$ almost surely if and only if Condition $(\Diamond^\prime 1)$ is satisfied. Moreover, assuming that $\bmX$ is min-id, one can check that Condition $(\Diamond^\prime 2)$ is equivalent to the constraint $\bell=\binfty$. Therefore, under the assumption that $\bmX$ is min-id, $\bmX$ satisfies Condition $(\Diamond)$ if and only if $(H_t)_{t\in\RR}$ satisfies Condition $(\Diamond^\prime)$. 
\end{rem}

Recall that the distribution of a non-negative stochastic process in $D^\infty(\RR)$ is uniquely determined by its Laplace transform. Translating \cite[Theorems 2.8 and 3.4]{rosinskiinfdivproc} to non-negative c\`adl\`ag id-processes shows that the Laplace transform of a non-negative c\`adl\`ag id-process can be uniquely characterized by a function $b:\RR\to\RR$ and a measure $\nu$ on $D^\infty(\RR)$. To be precise, for all $d\in\N$, the Laplace transform of the $d$-dimensional margins of a non-negative id-process $(H_t)_{t\in\RR}\in D^\infty(\RR)$ satisfying Condition $(\Diamond^\prime 2)$ can be written as
\begin{align}
&L: [0,\infty)^d \times \RR^d \to [0,1];\  (\bm z,\bm t)\mapsto\  \ee\lk \exp\lc -\sum_{i=1}^d  z_i H_{t_i}\rc\rk\nonumber \\
&=\exp\Bigg(-\sum_{i=1}^d z_i b(t_i) \nonumber\\
& + \bigintsss_{D^\infty(\RR)}\lc \exp\lc -\sum_{i=1}^d z_i x(t_i)\rc-1-\sum_{i=1}^d z_i x(t_i)\id_{\{\vert x(t_i)\vert <\epsilon\}} \rc \nu(\mathrm{d}x)\Bigg),\label{laplaceexponentidprocess} 
\end{align}
where $\epsilon>0$ is arbitrary, $b:\RR \to \RR$ is a unique function and $\nu$ is a unique measure on $D^\infty(\RR)$ satisfying
\begin{enumerate}
    \item $\nu(0_{D^\infty(\RR)})=0$, and
    \item $\int_{D^\infty(\RR)} \min \{x(t),1\}\nu(\mathrm{d}x)<\infty$ for all $t\in\RR$.
\end{enumerate}
The function $b$ is called drift of $(H_t)_{t\in\RR}$ and the measure $\nu$ is called L\'evy measure of $(H_t)_{t\in\RR}$. More generally, every measure $\nu$ on $D^\infty(\RR)$ satisfying Conditions $1.$ and $2.$ is called a L\'evy measure.

The literature usually treats id-processes as stochastic processes in $\RR^\RR$. Therefore, if $H\in D^\infty(\RR)$ is id, it is not clear whether the $H^{(i,1/n)}$ appearing in Definition \ref{defidrvidprocess} can also be chosen as elements of $D^\infty(\RR)$, as was already pointed out in \cite[Remark 4.5]{BarndorffNielsen2006InfiniteDF}. If $H$ is stochastically continuous, non-negative and almost surely non-decreasing the answer is positive, as can be deduced from \cite[Proposition 6.2]{BarndorffNielsen2006InfiniteDF}.

In our framework it is important that all $H^{(i,1/n)}$ can be chosen as c\`adl\`ag processes, since we need to interpret $\exp\lc -H^{(i,1/n)}\rc$ as a random survival function in Section \ref{sectionmainresults}. The next lemma generalizes the results of \cite[Proposition 6.2]{BarndorffNielsen2006InfiniteDF} and shows that all c\`adl\`ag id-processes have a representation as a sum of c\`adl\`ag id-processes.
 
\begin{lem}[C\`adl\`ag id-processes are distributed as i.i.d.\ sum of c\`adl\`ag id-processes]
\label{lemcadlagsums}
Consider an id-process $H\in D^\infty(\RR)$. Then, for every $n\in\N$, we can find i.i.d.\ id-processes $\lc H^{(i,1/n)}\rc_{1\leq i\leq n}\in\lc D^\infty (\RR)\rc^n$ such that $H\sim \sumn H^{(i,1/n)}$.
\end{lem}

\begin{proof}
The proof can be found in Appendix \ref{appendixproofs}.
\end{proof}

Lemma~\ref{lemcadlagsums} verifies that Definition~\ref{defidrvidprocess} could also be formulated solely for process which are id in the space $D^\infty(\RR)$, since it excludes the possibility of the existence of a stochastic process $H\in D^\infty(\RR)$ which is id in $\RR^\RR$, but not id in $D^\infty(\RR)$. 

To interpret $\exp\lc -H^{(i,1/n)}\rc$ as a random distribution function when $H$ is an extended chronometer it remains to ensure that each $H^{(i,1/n)}$ is nnnd, i.e.\ an extended chronometer. 
The following corollary shows that extended chronometers are infinitely divisible in the space of extended chronometers, i.e.\ that every extended chronometer can be represented as the sum of arbitrarily many extended chronometers.

\begin{cor}[Extended chronometers are distributed as i.i.d.\ sum of extended chronometers]
\label{corcadlagsums}
Consider an extended chronometer $H$. Then, for every $n\in\N$, we can find i.i.d.\ extended chronometers $\lc H^{(i,1/n)}\rc_{1\leq i\leq n}$ such that $H\sim \sumn H^{(i,1/n)}$.
\end{cor}

\begin{proof}
The proof can be found in Appendix \ref{appendixproofs}.
\end{proof}

Our next goal is to connect the path properties of an extended chronometer with the support of its L\'evy measure. To get an intuition about the correspondences of path properties of an id-process and properties of its L\'evy measure we recall the following example from \cite{rosinskiinfdivproc}.

\begin{ex}[Finite L\'evy measure {\cite[Example 2.26]{rosinskiinfdivproc}}]
\label{examplefinitelevymeasure}
Consider a sequence of i.i.d.\ c\`adl\`ag processes $(h^{(i)})_{i\in\N}$ with marginal distribution $\pp_h$, a Poisson random variable $N$ with mean $\lambda$ and define a stochastic process $(H_t)_{t\in\RR}=\lc\sum_{i=1}^N h^{(i)}_t\rc_{t\in\RR}$. \cite[Example 2.26]{rosinskiinfdivproc} proves that $H$ is infinitely divisible with drift $0$ and L\'evy measure $\lambda \pp_h\lc\cdot \cap \{0_{D^\infty(\RR)}\}^\complement\rc$. Obviously, $H$ is nnnd if and only if the L\'evy measure $\lambda \pp_h\lc\cdot \cap \{0_{D^\infty(\RR)}\}^\complement\rc$ is supported on nnnd functions. 
\end{ex}

Unfortunately, the L\'evy measure of an id-process is infinite in most cases of interest. Thus, the construction method in Example~\ref{examplefinitelevymeasure} is rather limited and we cannot immediately draw the same conclusions as in Example~\ref{examplefinitelevymeasure} for general extended chronometers. 

In the following we show that the observations of Example~\ref{examplefinitelevymeasure} remain valid for extended chronometers, i.e.\ if a c\`adl\`ag id-process is nnnd then its L\'evy measure is concentrated on nnnd c\`adl\`ag functions. More specifically, we show that the L\'evy measure of an extended chronometer satisfying Condition $(\Diamond^\prime)$ is supported on the nnnd functions in $D^\infty(\RR)$ satisfying $\lim_{t\to-\infty}x(t)=0$. A weaker version of this statement was claimed but not proven in \cite[Section 4]{leeinfinitely}, who followed a technically different approach in comparison to the pathwise approach of \cite{rosinskiinfdivproc}. The author claimed that a proof of his statement works similar to other proofs given in \cite{leeinfinitely}. However, all of the referred proofs are not very detailed and are not compatible with the pathwise approach of \cite{rosinskiinfdivproc}. An alternative approach to prove our claim would use an application of \cite[Theorem 3.4]{rosinskiinfdivproc}, which provides a tool to restrict the L\'evy measure of an id-process to a smaller domain. Unfortunately, the theorem cannot be applied in our setting, since our favored domain, nnnd c\`adl\`ag functions, does not form an algebraic group under addition. Therefore, we provide a formal proof of our claims in the following proposition.

\begin{prop}[Laplace transform of an extended chronometer satisfying Condition $(\Diamond^\prime)$]
\label{propositionltchronometer}
Let $d\in\N$ and let $(H_t)_{t\in\RR}$ denote an extended chronometer satisfying Condition $(\Diamond^\prime)$. Then, for $\bm z\in[0,\infty)^d,\bm t\in\RR^d$, we have 
\begin{align*}
    L(\bm t,\bm z)&=\ee\lk e^{-\sumd z_i H_{t_i}} \rk\\
    &=\exp\lc  -\sumd z_i b(t_i) -  \int_{M^0_\infty} \lc 1-e^{-\sumd z_i x(t_i) } \rc \nu(\mathrm{d}x)\rc,
\end{align*} 
where $\nu$ is a L\'evy measure on 
$$M^0_\infty:=\{ x\in D^\infty(\RR) \mid x \text{ is non-decreasing}, \lim_{t\to -\infty} x(t)=0\},$$
and $b\in M^0_\infty\cap D(\RR)$.
\end{prop}

\begin{proof}
The proof can be found in Appendix \ref{appendixproofs}.
\end{proof}

\begin{rem}[Id-process with L\'evy measure on $M_\infty^0$ has extended chronometer version]
By a similar reasoning as in the proof of Proposition \ref{propositionltchronometer} it is also possible to prove that every driftless c\`adl\`ag id-process with L\'evy measure concentrated on nnnd c\`adl\`ag functions has a version that is non-decreasing and non-negative. We omit a proof of this statement, since this fact will not be used in our paper.
\end{rem}

\begin{rem}[Condition $(\Diamond^\prime 1)$ corresponds to vanishing functions in L\'evy measure]
The proof of Proposition \ref{propositionltchronometer} shows that the L\'evy measure of an id-process satisfying Condition $(\Diamond^\prime 1)$  is concentrated on c\`adl\`ag paths vanishing at $-\infty$. If we omit Condition $(\Diamond^\prime 1)$ this is no longer the case, which is the reason why we later need to omit this condition in Corollary \ref{corollarymainresult}. 
\end{rem}

Interestingly, we can infer $\pp(H_t=\infty)$ from the associated L\'evy measure, as the next example shows. 

\begin{ex}[Probability of a jump to $\infty$]
\label{examplekilledprocess}
Consider an extended chronometer $(H_t)_{t\in\RR}$ satisfying Condition $(\Diamond^\prime)$ with L\'evy measure $\nu$ and drift $b$. We want to investigate $\pp(H_t=\infty)$ for every $t\in\RR$, which is equivalent to investigating $\pp\lc X_i<t\ \forall\ i\in\N\rc$ for every $t\in\RR$, where $\bmX$ denotes the exchangeable sequence  constructed via Equation (\ref{constructionexchangeseq}).

The (one-dimensional) L\'evy--Khintchine representation of the infinitely divisible random variable $H_t$ yields $\pp(H_t=\infty)=1-\exp(-\upsilon_t(\infty))$, where $\upsilon_t$ denotes the (one-dimensional) L\'evy measure of $H_t$. Proposition~\ref{propositionltchronometer} shows that $\upsilon_t(\infty)=\nu\lc \{x\in M_\infty^0 \mid x(t)=\infty\}\rc$. Therefore, 
$$\pp(H_t=\infty)=1-\exp\lc -\nu\lc \{x\in M_\infty^0 \mid x(t)=\infty\}\rc\rc.$$

We emphasize that, e.g.\ in contrast to additive processes, $H$ cannot be decomposed into $H=H^{(1)}+H^{(2)}$, where $H^{(1)}$ is always finite and independent of $H^{(2)}\in\{0,\infty\}^\RR$. Therefore, jumps to $\infty$ do not occur independently of the path behavior of the process in general. Let us verify this claim by an application of Example~\ref{examplefinitelevymeasure}. Decompose $\nu$ into $\nu=\nu_\infty+\nu_f$, where $\nu_f:=\nu\lc \cdot\cap \{x\in M_\infty^0 \mid x(t)<\infty \text{ for all } t\in\RR\}\rc$ is concentrated on finite paths and $v_\infty:=\nu\big(\cdot \cap \{x\in M_\infty^0\mid$ $x(t)=\infty \text{ for some } t\in\RR\}\big) $ is concentrated on paths that jump to $\infty$. Now, assuming that $\nu_\infty$ is a finite measure with total mass $c$, we define $H^{(2)}$ as the id-process with L\'evy measure $\nu_\infty$ and $H^{(1)}$ as an independent id-process with L\'evy measure $\nu_f$ and drift $b$. Obviously, $H\sim H^{(1)}+H^{(2)}$, where $H^{(1)}\in D(\RR)$ has finite sample paths and 
$$H^{(2)}\sim \sum_{i=1}^N h^{(2,i)},$$
where $(h^{(2,i)})_{i\in\N}$ denotes an i.i.d.\ sequence of c\`adl\`ag processes with distribution $\nu_\infty /c$ and $N$ denotes an independent Poisson random variable with mean $c$. Since the paths of $h^{(2,1)}$ can follow every increasing (c\`adl\`ag) path we observe that $H^{(2)}_t$ may take all finite values. Thus, in general, $H$ cannot be decomposed into a finite process $H^{(1)}$ and a ``killing'' process $H^{(2)}\in\{0,\infty\}^\RR$. The case of infinite $\nu_\infty$ follows from $\nu_{t,\infty}:=\nu\big( \cdot\cap \{x\in M_\infty^0 \mid x(t)=\infty\}\big)$, where $\nu_{t,\infty}$ is a finite measure for all $t\in\RR$ and $\nu_\infty=\lim_{t\to\infty}\nu_{t,\infty}$.
\end{ex}

\section{Main results: Linking exchangeable min-id sequences to extended chronometers}
\label{sectionmainresults}
After having collected all auxiliary results, we now formulate the main contributions of this paper.

\subsection{Extendible min-id distributions satisfying $(\Diamond)$ are in one-to-one correspondence with extended chronometers satisfying $(\Diamond^\prime)$}
\label{sectionextminidcorrespondstoextchron}

We start with the characterization of extendible min-id distributions satisfying Condition $(\Diamond)$. Recall that de Finetti's Theorem implies that every exchangeable sequence $\bmX\in\bR^\N$ is in one-to-one correspondence with a nnnd c\`adl\`ag process. We show that the class of stochastic processes corresponding to min-id sequences satisfying Condition $(\Diamond)$ is precisely the class of extended chronometers satisfying Condition $(\Diamond^\prime)$.

\begin{thm}[Extendible min-id distributions correspond to extended chronometers]
\label{theoremchroncorrespondstominiddist}
The following are equivalent:
\begin{enumerate}
\item $\bmX$ is an exchangeable min-id sequence satisfying Condition $(\Diamond)$.
\item There exists an extended chronometer $(H_t)_{t\in\RR}\in D^\infty(\RR)$
satisfying Condition $(\Diamond^\prime)$ such that $\bmX\sim \big(\inf \{t\in\RR\mid H_t\geq E_i\}\big)_{i\in\N}$, where $(E_i)_{i\in\N}$ are i.i.d.\ unit exponential and independent of $H$.  
\end{enumerate}
Moreover, the law of $H$ is uniquely associated to the law of $\bmX$.
\end{thm}

\begin{proof}
The proof can be found in Appendix \ref{appendixproofs}.
\end{proof}

It is worth noting that Lemma~\ref{lemmatrafomaxtominid} can be translated into a time-change of the extended chronometer.
\begin{cor}[Marginal transformation of exchangeable min-id sequence is time-change of the chronometer]
\label{cormarginaltrafoistimechange}
Consider an exchangeable min-id sequence $\bmX\in (\infty,\infty]^\N$ and a left-continuous non-decreasing transformation $f$. Let $H^\bmX$ denote the extended chronometer corresponding to $\bmX$ and $H^{f(\bmX)}$ denote the chronometer corresponding to $f(\bmX)$. Then $H^{f(\bmX)} \sim H^\bmX \circ f^{\leftharpoonup}$, where $f^{\leftharpoonup}(t):=\inf\{s\in\RR\mid f(s)> t\}$.
\end{cor}

\begin{proof}
The claim follows from the identity
$$\pp\big( f(X_1)>t_1,\ldots,f(X_d)>t_d\big)=\pp\big( X_1>f^{\leftharpoonup}(t_1),\ldots,X_d>f^{\leftharpoonup}(t_d)\big).$$
\end{proof}

The following examples present two interesting applications of Corollary \ref{cormarginaltrafoistimechange}.

\begin{ex}[$\N_0$-valued exchangeable min-id sequences]
Consider the non-decreasing left-continuous transformation $x\mapsto \lceil x\rceil:= \min\{n\in\N_0\mid x\leq n\}$. Lemma~\ref{lemmatrafomaxtominid} implies that each exchangeable min-id sequence $\bmX$ can be transformed into an $\N_0$-valued exchangeable min-id sequence $\lceil \bmX \rceil$. Corollary~\ref{cormarginaltrafoistimechange} shows that the extended chronometer $H^{\lceil\bmX\rceil}$ associated with $\lceil \bmX \rceil$ can be obtained via a time-change of the extended chronometer $H^\bmX$ associated with $\bmX$. Thus, $H^{\lceil\bmX\rceil}\sim H^\bmX \circ \lfloor \cdot \rfloor$, where $\lfloor x\rfloor:=\lceil x \rceil^\leftharpoonup =\max \{n\in\N_0 \mid x\geq n\}$. Defining $\lc J_i\rc_{i\in\N}:=\lc H^\bmX_i-H^\bmX_{i-1}\rc_{i\in\N}$ we observe that $$H^{\lceil\bmX\rceil}_t=H^\bmX_{\lfloor t\rfloor}=\lc H^\bmX_0+\sum_{i=1}^{\lfloor t \rfloor} J_i \rc \id_{\{t\geq 0\}}$$
can be represented as a pure jump process. If $H^\bmX_0=0$ and $H^\bmX$ has stationary and independent increments $H^{\lceil\bmX\rceil}$ is known as a random walk and the sequence $\lceil \bmX \rceil$ follows a multivariate narrow-sense geometric distribution, see \cite{Maimultivgeom}. In this case $\bmX$ has $d$-dimensional marginal distributions 
$$ (X_1,\ldots,X_d)\sim \big( \inf\{E_I \mid i\in I\}\big)_{1\leq i\leq d},$$  
where $(E_I)_{I\subset\{1,\ldots,d\}}$  is a collection of independent geometrically distributed random variables with parameters $(1-p_I)$ such that $p_I$ only depends on $\vert I\vert$.
Moreover, the associated extended chronometer $H^{\lceil \bmX \rceil}$ is a random walk with infinitely divisible i.i.d.\ jumps $J_i\sim H^\bmX(1)$.
\end{ex}

\begin{ex}[From min- to max-id]
\label{examplemintomaxid}
Consider an exchangeable min-id sequence $\bmX$ with associated chronometer $H^\bmX$ and a continuous strictly decreasing transformation $f$ with its corresponding inverse $f^{-1}$. Lemma~\ref{lemmatrafomaxtominid} shows that $\bm Y^{(f)}:=f(\bmX)$ is an exchangeable max-id sequence. According to de Finetti's Theorem there exists a random distribution function $F^{(f)}$ such that
$$\bm Y^{(f)}\sim \big( \inf\{ t\in\RR \mid F^{(f)}_t\geq U_i\}\big)_{i\in\N} $$
for an i.i.d.\ sequence of Uniform$(0,1)$ distributed random variables $(U_i)_{i\in\N}$. Noting that 
\begin{align*}
&\pp(Y^{(f)}_1\leq t_1,\ldots,Y^{(f)}_d\leq t_d)=\pp\lc X_1 \geq f^{-1}(t_1),\ldots,X_d\geq f^{-1}(t_d)\rc\\
&=\lim_{\bm z\searrow \bm t} \pp\lc X_1>f^{-1}( z_1),\ldots,X_d> f^{-1}(z_d)\rc 
= \lim_{\bm z\searrow \bm t}\ee\lk \exp\lc -\sum_{i=1}^d H^{\bmX}_{f^{-1}(z_i)}\rc\rk    \\
&=\ee\lk \exp\lc -\sum_{i=1}^d \lim_{ z_i\searrow t_i}H^{\bmX}_{f^{-1}(z_i)}\rc\rk 
\end{align*}
yields that $\lc F^{(f)}_t\rc_{t\in\RR}\sim \lc \exp\lc -\lim_{z\searrow t} H^{\bmX}_{f^{-1}(z)}\rc\rc_{t\in\RR}.$ Therefore,
$$\bm Y^{(f)}\sim \lc \inf\bigg\{ t\in\RR \ \bigg\vert -\log\lc 1-\exp\lc -\lim_{z\searrow t} H^{\bmX}_{f^{-1}(z_i)}\rc\rc\geq E_i\bigg\}\rc_{i\in\N}, $$
where $(E_i)_{i\in\N}$ is an i.i.d.\ sequence of unit exponential random variables.
\end{ex}

\subsection{The exponent measure of an exchangeable min-id sequence satisfying $(\Diamond)$ is a mixture of product probability measures}

We present an analogue of de Finetti's Theorem for exponent measures of exchangeable min-id sequences satisfying $(\Diamond)$.

First, we need to introduce some notation. For any distribution function $G$ of a random variable on $(-\infty,\infty]$, define $\pp_G$ as the probability measure associated with the distribution function $G$. Furthermore, $\otimes_{i=1}^d \pp_G$ denotes the probability measure on $(-\infty,\infty]^d$ associated with $d\in\N\cup\{\infty\}$ i.i.d.\ copies of random variables with distribution $\pp_G$. Define the space of distribution functions of random variables on $(-\infty,\infty]$ as
$\overline{M}^0_\infty:=\big\{G:\RR\to [0,1]\mid G\text{ is distribution function of a random variable on } (-\infty,\infty] \big\}$. Let $\gamma$ denote a measure on $\overline{M}^0_\infty$ and define a mixture of product probability measures by 
\begin{align}
 \mu_{\gamma,d}(\cdot):=\int_{\overline{M}^0_\infty} \otimes_{i=1}^d \pp_G(\cdot)\ \gamma(\mathrm{d}G) .\label{eqnconstmugamma}
\end{align}
Additionally, for any non-decreasing function $b\in M^0_\infty$, define $\mu_{b,d}$ as the exponent measure of $d\in\N\cup\{\infty\}$ i.i.d.\ copies of random variables on $(-\infty,\infty]$ with survival function $\exp(-b(\cdot))$ whose existence is ensured by Proposition \ref{propositionextendibleexponentmeasure}.

\begin{thm}[Exponent measure of exchangeable min-id sequence]
\label{theoremexponentmeasureisiidmixture}
The following are equivalent:
\begin{enumerate}
\item $\bmX$ is an exchangeable min-id sequence satisfying $(\Diamond)$.
% \item 
% The exchangeable exponent measure $\mu$ is supported on $E^\N_\binfty$ and satisfies 
% \begin{align*}
% \mu\lc \bigg\{\bmx\in\bR^\N \ \bigg\vert \ (x_{i_1},\ldots,x_{i_d})\in A \bigg\} \rc=\mu_{b,d} (A)+ \bigintssss_{\overline{M}^0_\infty} \otimes_{i=1}^d\pp_G\lc A\rc \gamma(\mathrm{d}G)
% \end{align*}
% for every $A\in\bc(E^d_\binfty)$, $(i_1,\ldots,i_d)\in \N^d$ and $d\in\N$, where $\gamma$ is a measure on 
% $$\overline{M}^0_\infty:=\big\{G:\RR\to [0,1]\mid G\text{ is a distribution function of a random variable on } (-\infty,\infty] \big\}$$
% satisfying $\gamma\lc 0_{\overline{M}^\infty_0}\rc=0 \text{ and } \int_{\overline{M}^0_\infty} G(t)\gamma(\mathrm{d}G)<\infty \text{ for all } t\in\RR$.
\item There exist a function $b\in M^0_\infty$ and a measure $\gamma$  on $\overline{M}^0_\infty$ satisfying $\gamma\lc 0_{D^\infty(\RR)}\rc=0$ and $\int_{\overline{M}^0_\infty} G(t)\gamma(\mathrm{d}G)<\infty$ for all $t\in\RR$ such that the exponent measure of $\bmX$ is given by
\begin{align*}
\mu\lc A  \rc=\mu_{b,\infty} (A)+ \mu_{\gamma,\infty}(A)
\end{align*}
for every $A\in\bc(E^\N_\binfty)$.
\end{enumerate} 
Moreover, the L\'evy measure and drift of the extended chronometer associated with $\bmX$ are given by $\nu(A)=\gamma(\{G\in \overline{M}^0_\infty\mid G=1-\exp(-x(\cdot))\text{ for some } x\in A\})$ for every $A\in\bc\lc D^\infty(\RR)\rc$ and $b(t)=\mu_{b,1} \lc [-\infty,t]\rc$.
\end{thm}

\begin{proof}
The proof can be found in Appendix \ref{appendixproofs}.
\end{proof}

Recall that a min-id sequence has exponent measure supported on $A^{\perp}:=\big\{\bmx\in\bR^\N\mid x_i=\infty \text{ for all but one } i\in\N\big\}$ if and only if it is an i.i.d.\ sequence. Theorem~\ref{theoremexponentmeasureisiidmixture} yields a decomposition of the global exponent measure $\mu$ into $\mu_{b}+\mu_\gamma$, where $\mu_{b}:=\mu_{b,\infty}$ is supported on the set $A^{\perp}$ and $\mu_\gamma:=\mu_{\gamma,\infty}$. Therefore, $\bmX\sim\min \{\bmX^{(1)},\bmX^{(2)}\}$, where $\bmX^{(1)}$ is an i.i.d.\ sequence with exponent measure $\mu_{b}$ and $\bmX^{(2)}$ is an exchangeable min-id sequence with exponent measure $\mu_\gamma$. This raises the question whether the decomposition from Theorem~\ref{theoremexponentmeasureisiidmixture} separates $\mu$ into an independence part $\mu_{b}$ and a dependence part $\mu_\gamma$, which would be a desired feature for modeling purposes. Mathematically this translates to $\mu_{b}$ and $\mu_\gamma$ being singular.

% \textcolor{red}{OLD: If $\mu_\gamma$ does not have mass on $A^{\perp}$, then $\bmX^{(2)}$ does not contain an independent sequence, i.e.\ $\bmX^{(2)}$ cannot be further decomposed into the minimum of a non-trivial i.i.d.\ sequence and an exchangeable min-id sequence. On the level of the associated extended chronometer this would correspond to the fact that driftless extended chronometers do not put mass on $A^{\perp}$. However, the next corollary shows that $\mu_\gamma$ usually puts mass on $A^{\perp}$. Thus, the decomposition of $\mu$ into $\mu_{b}$ and $\mu_\gamma$ usually does not completely separate dependence from independence.  }

If $\mu_\gamma$ does not have mass on $A^{\perp}$, then $\bmX^{(2)}$ does not contain an independent sequence, i.e.\ $\bmX^{(2)}$ cannot be further decomposed into the minimum of a non-trivial i.i.d.\ sequence and an exchangeable min-id sequence. On the level of the associated extended chronometer this would correspond to the fact that exponent measures $\mu_\gamma$ associated to driftless extended chronometers do not put mass on $A^{\perp}$. Indeed, the next corollary shows that $\mu_\gamma$ never puts mass on $A^{\perp}$, which implies that the decomposition of $\mu$ into $\mu_{b}$ and $\mu_\gamma$ separates dependence from independence.  

\begin{cor}[Decomposition of an exponent measure into dependence and independence]
\label{corollarydecompositionexponentmeasure}
% Let $b\in M^0_\infty$ denote a drift and let $\tilde{H}$ denote a driftless extended chronometer satisfying Condition $(\Diamond^\prime)$ with associated exponent measure $\mu^{\tilde{H}}$. Let $\bmX$ denote the exchangeable min-id sequence associated to $H:=b+\tilde{H}$. Then $\mu^{\tilde{H}}=\mu_\gamma$ and the following are sufficient for $\mu_{b}$ and $\mu_\gamma$ being singular:
% \begin{enumerate}
% \item $\bmX=\min\{\bmX^{(\tilde{H})},\bmX^{(b)}\}$, where $\bmX^{(b)}$ is i.i.d.\ $\exp(-b(\cdot))$ and $\bmX^{(\tilde{H})}$ associated with $\tilde{H}$ satisfies $\lim_{d\to\infty}\rho^u_d=1$. 
% \end{enumerate}

%Let $b\in M^0_\infty$ denote a drift and $\tilde{H}$ denote a driftless extended chronometer satisfying Condition $(\Diamond^\prime)$ with associated exponent measure $\mu^{\tilde{H}}$. 

Let $\bmX$ denote the exchangeable min-id sequence associated to the nnnd id-process $H=b+\tilde{H}$, where $b$ denotes the drift of $H$ and $\tilde{H}$ denotes the driftless random component of $H$. Let $\mu^{\tilde{H}}$ denote the exponent measure associated to $\tilde{H}$. Then $\mu_{\gamma}=\mu^{\tilde{H}}$ and the exponent measure $\mu$ of $\bmX$ is given by $\mu=\mu_{b}+\mu_{\gamma}$, where $\mu_{b}$ and $\mu_{\gamma}$ are singular.
\end{cor}

\begin{proof}
% \begin{enumerate}
%     \item Let $2 \leq d\in\N$. Corollary \ref{cortaildependence} shows that $-\log\lc \rho^u_d\rc = \mu_d^{\tilde{H}}\lc [-\infty,\infty)\times \{\infty\}^{d-1}\rc$. Thus, if $\lim_{d\to\infty}\rho^u_d=1$, Theorem \ref{theoremexponentmeasureisiidmixture} implies 
% \begin{align*}
%   \mu^{\tilde{H}}\lc A^{\perp} \rc&\leq \sum_{i\in\N} \mu^{\tilde{H}}\lc \Big\{ \bmx\in\bR^\N\ \big\vert\  x_i<\infty \text{ and } x_j=\infty \text{ for all } i\not=j\in\N  \Big\}  \rc\\
%   &=\sum_{i\in\N} \lim_{d\to\infty} \mu^{\tilde{H}}\lc \Big\{ \bmx\in\bR^\N \mid (x_1,\ldots,x_d)\in [-\infty,\infty)\times \{\infty\}^{d-1} \Big\} \rc \\
%   &=\sum_{i\in\N} \lim_{d\to\infty} \mu_d^{\tilde{H}}\lc [-\infty,\infty)\times \{\infty\}^{d-1}\rc \\
%   &=0
% \end{align*}
% Therefore, $\mu_b$ abd $\mu^{\tilde{H}}$
% \item It is obvious that $\rho^u_d=1$
% \end{enumerate}
Note that $\pp_G\big( (-\infty,\infty)\big) \prod_{i\in\N} \pp_G (\{\infty\})=0$ for all $G\in\overline{M}_\infty^0$, since $\lim_{t\to\infty}G(t)>0$ for all $0_{D^\infty(\RR)}\not=G\in\overline{M}_\infty^0$. Thus,
\begin{align*}
    \mu_{\gamma}\lc A^{\perp}\rc& = \bigintssss_{\overline{M}^0_\infty}\sum_{i\in\N} \lc \otimes_{k\in\N}\pp_G\rc \big( \{x_i< \infty \text{ and }x_j=\infty \text{ for all } j\not=i\} \big) \gamma(\mathrm{d}G)\\
    &=\bigintssss_{\overline{M}^0_\infty}\sum_{i\in\N}\pp_G\big( (-\infty,\infty)\big) \bigg(\prod_{j\in\N} \pp_G (\{\infty\}) \bigg) \gamma(\mathrm{d}G)  =0
\end{align*}
Thus, $\mu_b$ and $\mu_\gamma$ are singular.
\end{proof}

\begin{rem}[Non-separability of dependence and independence for finite dimensional margins]
It is educational to observe that $\mu_{b,d}$ and $\mu_{\gamma,d}$ are singular for a fixed $d\in\N$ if and only if $\gamma$ is concentrated on the set $\{ G \in \overline{M}_\infty^0 \mid \lim_{t\to\infty}G(t)=1\}$, since 
\begin{align*}
    &\mu_{\gamma,d}\lc \cup_{1\leq i\leq d} \{x_i<\infty\text{ and } x_j=\infty \text{ for all } j\not=i\}\rc\\
    &=\sum_{1\leq i\leq d} \int_{\overline{M}^0_\infty}\pp_G\big( (-\infty,\infty)\big) \big( \pp_G (\{\infty\})\big)^{d-1} \gamma(\mathrm{d}G)  \\
    &=d \int_{\overline{M}_\infty^0} \lc \lim_{t\to\infty}G(t) \rc \lc 1-\lim_{t\to\infty}G(t)\rc^{d-1} \gamma(\mathrm{d} G) .
\end{align*}
Therefore, $\mu_{b,d}$ and $\mu_{\gamma,d}$ usually do not separate dependence from independence, meaning that $\bmX^{(2)}_d\sim \min\big\{ \bmX^{(2,\perp)}_d,\bmX^{(2,\not\perp)}_d\big\}\sim \exp\lc -\mu_{\gamma,d}\lc (\cdot,\binfty]^\complement \rc \rc$ where $\bmX^{(2,\perp)}_d$ has components which are $d$ i.i.d.\ copies of a non-trivial random variable and $\bmX^{(2,\not\perp)}_d$ denotes a non-trivial min-id random vector. Intuitively, this may be interpreted as follows: the finite dimensional exponent measure $\mu_{\gamma,d}$ smears around independence and this effect can only be distinguished from independence in the limit. 
\end{rem}

\begin{rem}[Dependence structure of exchangeable min-id sequences]
\label{remtaildependence}
\cite[Lemma 4.4]{MaiSchererexMSMVE} shows that exchangeable min-stable sequences admit positive $\rho^u_2$ if and only if the exchangeable min-stable sequence is given by the comonotonic sequence $\bmX=(\bar{X},\bar{X},\ldots)$, where $\bar{X}\in\bR$ is some univariate random variable. Therefore, exchangeable min-stable sequences satisfy $\rho^u_2\in\{0,1\}$ and $\rho^u_2=1$ implies $\bmX=(\bar{X},\bar{X},\ldots)$, which raises the question whether the same result holds for exchangeable min-id sequences.

The question can be answered by the following example: Define an exchangeable min-id sequence via the extended chronometer $(H_t)_{t\in\RR}=-\log\lc 1-\Gamma_t\rc$, where $(\Gamma_t)_{t\in\RR}$ denotes a stochastically continuous Dirichlet process, see \cite{ferguson1973}. The Dirichlet process naturally appears in Bayesian statistics when the distribution function $(\Gamma_t)_{t\in\RR}$ of an i.i.d.\ sequence is viewed as the random quantity of interest. The distribution of the Dirichlet process is then specified as the prior distribution on the space of distribution functions, which arises under the assumption that $\lc \int_{A_1}\mathrm{d} \Gamma,\ldots,\int_{A_d}\mathrm{d}\Gamma\rc $ follows a Dirichlet distribution for every $d\in\N$ and every measurable disjoint partition $\lc A_i\rc_{1\leq i \leq d}$ of $\RR$. The Dirichlet process is particularly convenient in Bayesian statistics, since its posterior distribution is again a (non-stochastically continuous) Dirichlet process.

The authors of \cite{ferguson1974,doksum1974} have shown that $H$ is infinitely divisible, which implies that the associated exchangeable sequence $\bmX$ is min-id. Moreover, \cite{mai2015dirichlet} show that the upper and lower bivariate tail dependence coefficients $\rho^u_2$ and $\rho^l_2$ of $\bmX$ can take any value in $(0,1)$. Thus, exchangeable min-id sequences can exhibit arbitrary positive bivariate upper and lower tail dependence, which shows that the dependence structure of exchangeable min-id sequences is much richer than the dependence structure of exchangeable min-stable sequences. 

To verify that $\rho_2^u=1$ does not imply $\bmX=(\bar{X},\bar{X},\ldots)$ when $\bmX$ is an exchangeable min-id sequence it is easy to see that every extended chronometer with finite L\'evy measure concentrated on the set $\{ x\in M_\infty^0\mid x(t)<\infty,\ \lim_{t\to\infty} x(t)=\infty \}$ yields an exchangeable min-id sequence $\bmX$ which satisfies $\rho^u_2=1$ but not $\bmX=(\bar{X},\bar{X},\ldots)$.
\end{rem}

\subsection{Characterization of general exchangeable min-id sequences}
\label{sectiongenextmin-id}
Even though we have mentioned that studying min-id distributions satisfying Condition $(\Diamond)$ is not a loss of generality, we feel the need to translate the results of Theorems~\ref{theoremchroncorrespondstominiddist} and~\ref{theoremexponentmeasureisiidmixture} to arbitrary exchangeable min-id sequences on $[-\infty,\infty]^\N$. The appearance of the following corollary is slightly more technical than that of Theorems~\ref{theoremchroncorrespondstominiddist} and~\ref{theoremexponentmeasureisiidmixture}, which is due to the subtleties in the definition of a survival function of random vectors in $[-\infty,\infty]^d$ and explains why we preferred to develop the preliminary results under Conditions $(\Diamond)$ and $(\Diamond^\prime)$.

\begin{cor}[Characterization of general exchangeable min-id sequences]
\label{corollarymainresult}
Assume that $\pp(\bmX=-\binfty)<1$. Then, the following are equivalent:

\begin{enumerate}
\item $\bmX\in\bR^\N$ is an exchangeable min-id sequence.

\item There exists $\bell=(\ell,\ell,\ldots )\in(-\infty,\infty]^\N$ such that $$(X_i)_{i\in\N} \sim \big( \inf\{t\in(-\infty,\ell) \mid H_t\geq E_i\}\big)_{i\in\N},$$ where $\inf \emptyset=: \ell$, $(E_i)_{i\in\N}$ is a sequence of i.i.d.\ unit exponential random variables independent of a unique extended chronometer $(H_t)_{t\in(-\infty,\ell)}\in D^\infty((-\infty,\ell))$ satisfying
\begin{enumerate}
\item $\pp(H_t=0)=1$ for all $t< \sup \big\{ x\in\RR\ \vert \ \pp(X_1\leq x)=0\big\}=:w$,
\item The L\'evy measure $\nu$ of $H$ is supported on
$$M_\ell:=\{x\in D^\infty((-\infty,\ell))\mid x\text{ is nnnd and }x(t)=0 \text{ for all } t<w\},$$
\item $H$ has real-valued drift $b\in M_\ell\cap D((-\infty,\ell))$.
\end{enumerate}

\item There exists $\bell=(\ell,\ell,\ldots )\in(-\infty,\infty]^\N$ and an exchangeable Radon measure $\mu$ on $E^\N_\bell$ such that
\begin{align*}
&\pp\lc (X_{i_1},\ldots X_{i_d})\in \times_{i=1}^d \{x_i,\ell]\rc\\
&=\begin{cases} \exp\Big( -\mu\big( \big\{\bm y \in E_\bell^\N \mid  (y_{i_1},\ldots,y_{i_d})\in \lc \times_{i=1}^d \{x_i,\infty]\rc^\complement \big\}\big)\Big) & \bmx<\bell\\
0 & otherwise
\end{cases},
\end{align*}
where $\{x_i,\infty]$ is interpreted as $(x_i,\infty]$ if $x_i>-\infty$ and $\{-\infty,\infty]$ is interpreted as $[-\infty,\infty]$. Moreover, for all $A\in \bc\lc E^\N_\bell\rc$, we have  
$$\mu\big\{\bm y \in E^\N_\bell\mid  (y_{i_1},\ldots,y_{i_d})\in A \big\}= \mu_{b}\lc A\rc  +\bigintssss_{\overline{M}_\ell} \otimes_{i\in\N}\pp_G\lc A\rc \gamma(\mathrm{d}G), $$
where $\gamma$ is a measure supported on 
$$\overline{M}_\ell:=\big\{G:(-\infty,\ell)\to [0,1]\mid G\text{ distr.\ fct.\ of random variable on } [w,\ell]\big\}$$
satisfying
$$\gamma(0_{D^\infty(\RR)})=0 \text{ and } \int_{\overline{M}_\ell} G(t)\gamma(\mathrm{d}G) <\infty \text{ for all } t\in(-\infty,\ell).$$
\end{enumerate} 
The relation of $\gamma$ and $\nu$ is given by
 $$\gamma(A)=\nu\lc\big\{x\in M_\ell\mid \big(1-\exp(-x(\cdot))\big)\in A\big\}\rc \text{ for every } A\in\bc \lc \overline{M}_\ell\rc$$
\end{cor}

\begin{proof}
The proof of Proposition \ref{propositionltchronometer} shows that the L\'evy measure of a nnnd c\`adl\`ag id-process is concentrated on nnnd c\`adl\`ag functions. The rest of the proof is a combination of Theorem~\ref{theoremchroncorrespondstominiddist}, Theorem~\ref{theoremexponentmeasureisiidmixture} and Lemma~\ref{lemmatrafomaxtominid}.
\end{proof}

% Corollary~\ref{corollarymainresult} can be refined for exchangeable min-id sequences with continuous marginal distribution.

% \begin{cor}[Continuous exchangeable min-id sequences have stochastically continuous chronometer]
% \label{cormainresultstochcont}
% An exchangeable min-id sequence $\bmX$ has continuous marginal distributions if and and only if the associated extended chronometer $H$ is stochastically continuous.
% \end{cor}

% \begin{proof}
% Follows immediately from Corollary \ref{corcontdistrcorresstochcontchron}.
% \end{proof}

A version of %Corollaries~\ref{cormainresultstochcont} and
Corollary \ref{corollarymainresult} for exchangeable max-id sequences can be easily deduced from Lemma~\ref{lemmatrafomaxtominid} and Example \ref{examplemintomaxid}.

\section{Established families unified under the present umbrella}
\label{sectionexamples}
In this section we present several important examples of exchangeable min-id sequences. Moreover, we investigate the dependence structure of exchangeable min-id sequences and characterize extendible min-id random vectors with finite exponent measure.

\subsection{Independence and comonotonicity}
Corollary~\ref{corollarymainresult} shows that the random sequence $\bmX\in[-\infty,\infty]^\N$ with i.i.d.\ components distributed according to the survival function $\exp(-b(\cdot))$ corresponds to the deterministic process $H_t=b(t)$. The exponent measure of $\bmX$ is given by $\mu=\mu_{b}$.

Consider the comonotonic case $\bmX=(\bar{X},\bar{X},\ldots)$, where the random variable $\bar{X}\in\bR$ satisfies $\pp\lc \bar{X}=-\infty\rc<1$. The corresponding driftless extended chronometer $H$ is given by $\lc H_t\rc _{t\in(-\infty,\ell)}=\lc \infty \id_{\{\bar{X}\leq t\}}\rc _{t\in(-\infty,\ell)}$. The L\'evy measure $\nu$ of $H$ is supported on  
$$\{x\in D^\infty((-\infty,\ell))\mid x(\cdot)=\infty \id_{\{ \cdot \geq a\}}\text{ for some  }  a\in(-\infty,\ell)\}.$$ 
Moreover, for $t<\ell$, 
$$\nu\lc \{\infty  \id_{\{ \cdot \geq a\}}\mid  a\in (-\infty,t]\}\rc=-\log\lc \pp(\bar{X}>t) \rc.$$
Therefore, 
\begin{align*}
\mu\lc (\bm t,\binfty]^\complement\rc&=\bigintssss_{M_\ell } \lc 1-\exp\lc -\sum_{i\in\N} x(t_i)\rc\rc \nu(\mathrm{d}x)=-\log\lc \pp\big(\bar{X}>\max_{i\in\N} t_i\big) \rc\\
&=-\log\big( \pp(\bmX>\bm t) \big).
\end{align*}

\subsection{Exogenous shock models / additive processes}
\label{subsectiongenmarsholkin}
\cite{Sloot2020b,MaiSchenkSchererexchangeableexshock} prove that extendible exogenous shock models constitute a proper subclass of extendible min-id distributions. Fix $d\in\N$ and define a family of independent random variables $\lc\tau_{I}\rc_{I\subset \{1,\ldots,d\}}\in [0,\infty)^{2^d}$. Moreover, let the distribution of $\tau_{I}$ be continuous and solely dependent on $\vert I\vert$, i.e.\ the cardinality of the subset $I$ of $\{1,\ldots,d\}$. Then the $d$-dimensional random vector 
\begin{align}
(X_i)_{1\leq i\leq d}\sim \lc \min_{i\in I\subset \{1,\ldots,d\}}\tau_{I}\rc_{1\leq i\leq d} \label{eqngenmarsholkin}
\end{align}
is exchangeable and interpreted as an exogenous shock model. The random variable $\tau_I$ models the arrival time of an exogenous shock destroying all components $I$ and $X_i$ equals the first time point at which component $i$ is affected by a shock. Exchangeability boils down to our assumption that the shock arrival time distributions only depend on the number of components affected by the respective shocks. Thus, the model is parametrized by $d$ distribution functions, since there are $d$ different ``shock sizes''. If we let $d \rightarrow \infty$ in this construction, Kolmogorov's extension theorem guarantees the existence of an exchangeable sequence $\bmX\in(0,\infty)^\N$ with the just described $d$-dimensional marginal distributions. 
\cite{Sloot2020b,MaiSchenkSchererexchangeableexshock} prove that the associated extended chronometer is a stochastically continuous c\`adl\`ag processes with independent increments. Such processes are known as (possibly killed) additive subordinators, see \cite{satolevyinfinitely,bertoinsubordinators} for a detailed treatment. The corresponding L\'evy measure $\nu$ of $H$ is supported on the class of one-step functions $\{u\id_{\{s\leq\cdot\}}\mid s\in(0,\infty) ,u\in(0,\infty]\}$ \cite{rosinskiinfdivproc}. Assuming that the extended chronometer is driftless, this implies that the associated exponent measure $\mu$ is an infinite mixture of (probability) distributions in the set 
$$\big\{ \otimes_{i\in\N} \lc (1-\exp(-u))\delta_s+\exp(-u)\delta_\infty\rc \mid u\in(0,\infty],s\in(0,\infty) \big\},$$ 
where $\delta_s$ denotes the Dirac measure at $s$.

An important subclass of exogenous shock models is the class of Marshall--Olkin distributions. It is obtained by restricting the distribution of $\tau_I$ in (\ref{eqngenmarsholkin}) to exponential distributions. Furthermore, \cite{janphd} shows that the class of extended chronometers corresponding to extendible Marshall--Olkin distributions is precisely the class of killed L\'evy subordinators. Killed L\'evy subordinators $(H_t)_{t\in\RR}$ are precisely the class of extended chronometers with stationary and independent increments, which satisfy $H_0=0$ and have an independent exponential killing rate, see \cite{bertoinsubordinators} for more details. Example 2.23 in \cite{rosinskiinfdivproc} shows that the L\'evy measure of a killed L\'evy subordinator $H$ is the image measure of the map
\begin{align}
    f: \lc [0,\infty)\times(0,\infty];\bc\lc(0,\infty)\times(0,\infty] \rc;\lambda\otimes \upsilon\rc\to M^0_\infty,\ (s,u)\mapsto u\id_{\{s\leq \cdot\}} 
\end{align}
on $M^0_\infty$, where $\lambda$ denotes the Lebesgue measure and $\upsilon$ denotes the L\'evy measure of the infinitely divisible random variable $H_1$. The drift of $H$ is given by $b(t)=b(1) t$, where $b(1)\geq 0$ denotes the drift of $H_1$.

In comparison to \cite{Sloot2020b,MaiSchenkSchererexchangeableexshock}, our framework allows for some additional flexibility, since we neither assume that $H$ is indexed by $[0,\infty)$, nor stochastically continuous, nor that $\lim_{t\to\infty} H_t=\infty$. More precisely, we allow for non-continuously distributed failure times on $[-\infty,\infty]$ instead of continuously distributed failure times on $(0,\infty)$. An important observation in this regard is that the behavior of the stochastic model (\ref{eqngenmarsholkin}) under monotone, componentwise transformations of the $X_i$ is not necessarily well-behaved in the framework of \cite{Sloot2020b,MaiSchenkSchererexchangeableexshock}. For instance, if $H$ is a L\'evy subordinator, a componentwise transformation of the $X_i$ corresponds to a change from $H_t$ to $H_{f(t)}$, which is no longer a L\'evy subordinator unless $f$ is linear with $f(0)=0$. In contrast, Corollary~\ref{cormarginaltrafoistimechange} tells us that under the more general umbrella of id-processes such marginal transformations are well-behaved.

As a final note of caution we remark that non-decreasing c\`adl\`ag processes with independent increments are not necessarily infinitely divisible. More precisely, our framework includes all c\`adl\`ag processes with independent id increments, meaning that if $\lc \lim_{s\nearrow t}(H_t-H_s)\rc_{t\in\RR}$ is a collection of independent and id random variables then $H$ is id. Thus, our framework does not incorporate all non-decreasing c\`adl\`ag processes with independent increments, but only extended chronometers with independent increments.

\subsection{A common framework for extreme-value copulas and (reciprocal) Archimedean copulas}
\label{subseccomumbrellareciparchimminstable}
We consider a pair $(\kappa,\rho)$ of a Radon measure $\kappa$ on $[0,\infty)$ and a probability measure $\rho$ on $\{G \mid G\text{ is a distribution function of a random variable on }[0,\infty]\}\subset \overline{M}^0_\infty$ with $\rho\lc 0_{D^\infty(\RR)}\rc=0$. Define the measure $\gamma_{\kappa,\rho}$ on $\overline{M}^0_\infty$ via
\begin{gather*}
\gamma_{\kappa,\rho}(A):=\int_0^{\infty} \rho\big( \{G\,:\,G(s\cdot) \in A\}\big) \kappa(\mathrm{d}s).
\end{gather*}
Moreover, assume that $\gamma_{\kappa,\rho}$ satisfies 
\begin{align*}
\int_{\overline{M}^{0}_{\infty}}G(t)\,\gamma_{\kappa,\rho}(\mathrm{d}G) = \int_{\overline{M}^{0}_{\infty}}\int_0^{\infty}G\lc \frac{t}{s}\rc \kappa(\mathrm{d}s)\rho(\mathrm{d}G)<\infty \text{ for all } t\geq 0. 
\end{align*}
Theorem~\ref{theoremexponentmeasureisiidmixture} implies that 
\begin{gather}
\mu_{\kappa,\rho}:=\int \otimes_{i\in\N} \pp_G \,\gamma_{\kappa,\rho}(\mathrm{d}G)
\label{expmeasurekappa}
\end{gather}
defines a valid global exponent measure on $E^\N_\binfty$. The following proposition provides a series representation of the associated non-decreasing id-process.

\begin{prop}[$\bmX$ and $H$ associated with $\gamma_{\kappa,\rho}$]\label{propstrongidt}
A series representation of the id-process $H$ associated with the exponent measure $\mu_{\kappa,\rho}$ in (\ref{expmeasurekappa}) is given by
\begin{gather*}
(H_t)_{t\geq 0} = \lc \sum_{k \geq 1} -\log\lc 1-G_k\Big( \frac{t}{S_k}\Big)\rc \rc_{t\geq 0},
\end{gather*}
where $N:=\sum_{k \geq 1}\delta_{(S_k,G_k)}$ denotes a Poisson random measure on $[0,\infty) \times \overline{M}^{0}_{\infty}$ with mean measure $\kappa \otimes \rho$. The survival function of the associated exchangeable min-id sequence $\bmX$ is given by
\begin{gather}
\pp(\bmX>\bmt) = \exp\lc- \ee \lk \kappa\lc\lk 0,\max_{i\in\N}\frac{t_i}{Z_i}\rk\rc \rk \rc, \label{eqnstabtaildep1}
\end{gather}
where $\bm{Z}=(Z_1,Z_2,\ldots)$ denotes an exchangeable sequence of random variables with associated random distribution function $G \sim \rho$.
\end{prop}

\begin{proof}
The proof can be found in Appendix \ref{appendixproofs}.
\end{proof}
We find it educational to remark that the sequences $(S_k)_{k\in\N}$ and $(G_k)_{k\in\N}$ are independent and that $(G_k)_{k\in\N}$ is i.i.d.\ drawn from $\rho$ and $S_k \sim f^{\leftharpoonup}(\epsilon_1+\ldots+\epsilon_k)$, where $(\epsilon_k)$ are i.i.d.\ unit exponential and $f^{\leftharpoonup}(x):=\inf \{t\geq 0 \mid f(t)\geq x\}$ denotes the generalized inverse of the function $f(t):=\kappa([0,t])$.\\
Three prominent examples for the choice of the pair $(\kappa,\rho)$ can be found in the literature. 
\subsubsection{Exchangeable min-stable sequences}
Choosing $\kappa(\mathrm{d}s)=\mathrm{d}s$ yields 
\begin{gather*}
\gamma_{\kappa,\rho}(A)=\int_0^{\infty} \rho(\{G\,:\,G(s\cdot) \in A\}) \mathrm{d}s.
\end{gather*}
\cite[Theorem 4.2]{koppmolchanov} shows that this defines the exponent measure of a driftless strong-idt process. This means that $H$ satisfies $H_0=0$ and
\begin{align*}
    \big( H_t\big)_{t\geq 0}\sim \lc \sumn H^{(i)}_{\frac{t}{n}}\rc_{t\geq 0} \text{ for all }n\in\N ,
\end{align*}
where $\lc H^{(i)}\rc_{i\in\N}$ are i.i.d.\ copies of $H$. The associated L\'evy measure $\nu$ is given by 
\begin{gather*}
\nu(A)=\gamma_{\kappa,\rho}\lc \big\{G\mid G= 1-\exp\lc -x(\cdot)\rc \text{ for some }x\in A\big\}\rc. 
\end{gather*}
Equation $(\ref{eqnstabtaildep1})$ simplifies to
\begin{align}
\pp(\bmX>\bm t)=   \exp\lc- \ee \Big[ \max_{i \geq 1}\frac{t_i}{Z_i}\Big] \rc, 
\label{sfMSMVE1}
\end{align}
which is a well known representation of the stable tail dependence function of an exchangeable min-stable sequence \cite{haan1984}. A result of \cite{maicanonicalspecrepofstabtail} states that the probability law of $\bm{Y}=(1/Z_1,1/Z_2,\ldots)$ becomes unique, if we additionally postulate that $\rho$ is concentrated on distribution functions $G$ satisfying $\int_{(0,\infty]}s^{-1}\,\mathrm{d}G(s)=1$. Survival functions of the form (\ref{sfMSMVE1}) are called min-stable multivariate exponential, since they imply $\bm{X} \sim n\,\min_{i=1,\ldots,n}\{\bmX^{(i)}\}$, where $n \in \N$ is arbitrary and $\bmX^{(i)}$ denote independent copies of $\bmX$. The min-stability property plays a fundamental role in multivariate extreme-value theory, since these are the only possible limiting distributions of componentwise minima of i.i.d.\ random vectors, after appropriate componentwise normalization, see \cite{resnickextreme}. Under the normalizing assumption $\ee[1/Z_1]=1$,  \cite{maicanonicalspecrepofstabtail} shows that the presented construction of the function $\ell(\bm{t})= \ee [ \max_{i \geq 1} t_i/Z_i] $ in (\ref{sfMSMVE1}) is general enough to comprise all possible stable tail-dependence functions associated with exchangeable min-stable sequences. In other words, by \cite[Theorem 5.3]{MaiSchererexMSMVE}, all driftless nnnd strong-idt processes necessarily admit a series representation as in Proposition~\ref{propstrongidt}, where $S_k \sim \epsilon_1+\ldots+\epsilon_k$ for an i.i.d.\ sequence of unit exponential random variables $(\epsilon_i)_{i\in\N}$ and $\kappa(\mathrm{d}s)=\mathrm{d}s$.
However, it should be noted that not all exchangeable min-stable random vectors can be obtained by this construction, since some exchangeable min-stable random vectors cannot be embedded into an exchangeable sequence \cite[Example 3.2]{maischerersimcopulas2017}. \\
Regarding related examples from the literature: $\alpha$-idt processes, see \cite{hakassou2012alphaidt,davydov2008strictly}, aggregate self-similar processes, see \cite{barczy2015dilatively,igloibarczydilativelystable}, and translatively stable processes, see \cite{hakassou2013idt}, are strong-idt processes up to scaling and time change, which implies their infinite divisibility. As we have seen in Corollary~\ref{cormarginaltrafoistimechange} a deterministic time change of a non-decreasing strong-idt process solely changes the one dimensional marginal distribution of $\bmX$, whereas a scaling of the strong-idt process corresponds to a scaling of the drift and a linear change of variables of the L\'evy measure of the strong-idt process. Therefore, the nnd instances of these processes uniquely correspond to an exchangeable min-id sequence $\bmX$ and their L\'evy measure can be obtained as the image measure of the L\'evy measure of a strong-idt process.

\subsubsection{Reciprocal Archimedean copulas}
\label{subsecreciparchim}
Choose $\kappa$ such that $\kappa(\{0\})=0$, $\kappa([0,\infty))=\infty$ and $\rho=\delta_{\hat{G}}$, with $\hat{G}$ being the unit Fr\'echet distribution function $\hat{G}(t)=\exp(-1/t)\id_{\{t> 0\}}$. This implies that
\begin{gather*}
    \gamma_{\kappa,\rho}(A)=\int_0^{\infty} \id_{\{\hat{G}(s\cdot)\in A\}} \kappa(\mathrm{d}s)=\kappa\lc\bigg\{ s\ \bigg\vert \exp\lc-\frac{1}{s\cdot}\rc\id_{\{\cdot > 0\}} \in A\bigg\}\rc.
\end{gather*}
The associated exchangeable min-id sequence $\bmX$ has survival function 
\begin{gather}
\label{eqnreciprocarchim}
\pp\lc \bmX>\bm t\rc=\exp\lc- \int_0^\infty 1-\prod_{i\in\N} \lc 1-\exp\lc -\frac{s}{t_i}\rc\rc\kappa(\mathrm{d}s)\rc,
\end{gather}
which is exactly the representation of an exchangeable sequence with reciprocal Archimedean copula as introduced in \cite{genestreciparchim}. Concretely, with the notation 
$\phi(t)=\int_0^{\infty} \exp(-tx)\kappa(\mathrm{d}x)$, we may rewrite
\begin{gather*}
\pp\lc \bmX>\bm t\rc = C_{\phi}\big( e^{-\phi(1/t_1)},e^{-\phi(1/t_2)},\ldots \big),
\end{gather*}
where for $(u_1,u_2,\ldots) \in [0,1]^{\N}$ the function
$$C_{\phi}(u_1,u_2,\ldots):=\frac{\prod_{\substack{A\subset \N,\\ \mid A\mid<\infty \text{ odd}}}\exp\lc -\phi\lc \sum_{k\in A} \phi^{-1}\lc -\log(u_k)\rc\rc\rc}{\prod_{\substack{A\subset \N,\\ \mid A\mid<\infty \text{ even}}}\exp\lc -\phi\lc \sum_{k\in A} \phi^{-1}\lc -\log(u_k)\rc \rc\rc } , $$
is the distribution function of an exchangeable sequence $\bm{U}$ and each component $U_i \sim \exp(-\phi(1/X_i))$ is uniformly distributed on $[0,1]$. The associated finite-dimensional margins of $C_{\phi}$ are called reciprocal Archimedean copulas, where the nomenclature is justified by some striking analogies with the concept of Archimedean copulas. E.g.\ the Galambos copula is obtained by choosing
\begin{gather*}
\kappa(\mathrm{d}s)= \frac{1}{\theta\,\Gamma(1+1/\theta)}\,s^{\frac{1}{\theta}-1}\,\mathrm{d}s
\end{gather*}
for some parameter $\theta >0$, which yields $\phi(t)=t^{-1/\theta}$. We refer the interested reader to \cite{genestreciparchim} for more details about reciprocal Archimedean copulas.

\subsubsection{Archimedean copulas with log-c.m.\ generator}
\label{exampleextarchimcop}
Let $\upsilon_M$ denote the univariate L\'evy measure of an infinitely divisible random variable $M\in(0,\infty)$. Define $\kappa$ by $\kappa(A);=\int_0^\infty \id_A(1/s) \upsilon_M(\mathrm{d}s)$ for all $A\in\bc([0,\infty))$ and choose $\rho=\delta_{\hat{G}}$, where $\hat{G}(t)=\lc 1-\exp(-t)\rc\id_{\{t> 0\}}$ denotes the distribution function of the unit exponential distribution. This implies that 
\begin{gather*}
    \gamma_{\kappa,\rho}(A)=\int_0^{\infty} \id_{\{1-\exp(-s\cdot)\in A\}} \upsilon_M(\mathrm{d}s)
\end{gather*}
and the associated exchangeable min-id sequence has survival function
\begin{gather*}
\pp\lc \bmX>\bm t\rc=\exp\lc- \int_0^\infty 1- \exp\lc -s\sum_{i\in\N}t_i\rc \upsilon_M(\mathrm{d}s)\rc.
\end{gather*}
It is easy to see that $\bmX$ has marginal survival function $\psi(t):=\ee\lk\exp(-tM)\rk$ and survival copula
$$C_{\psi}(\bm u)=\psi\lc\sum_{i\in\N} \psi^{-1}(u_i)\rc,\ \bm u\in [0,1]^\N,$$
which is an exchangeable Archimedean copula with completely monotone generator $\psi$ \cite{mcneil2009}. The definition of $\psi$ implies that $\psi$ is even log-completely monotone, which means that it has the special representation $\psi(t)=\exp(-g(t))$, where $g:[0,\infty)\to [0,\infty)$ is called the Laplace exponent of $\psi$ and $g$ is a Bernstein function, i.e.\ $\frac{\mathrm{d}}{\mathrm{d} t}g(t)$ is completely monotone. For example, the Gumbel copula, which is the only max-stable Archimedean copula, corresponds to $M\sim \alpha\textit{-stable}$ with $\psi(t)=\exp\lc-t^{\alpha}\rc, \alpha\in(0,1]$ and $g(t)=t^\alpha$.

To obtain all exchangeable Archimedean copulas on $[0,1]^\N$ define an extended chronometer via $H_t:=bt+Mt\id_{\{t\geq 0\}}$, where $b\geq 0$. $H$ has drift $b(t)=b\,t$ and L\'evy measure  
\begin{gather*}
\nu\lc \{x\in D^\infty(\RR) \mid x(t) = a\,t\,1_{\{t \geq 0\}},\,a \in A \}\rc=\upsilon_M(A),
\end{gather*}
which is the unique L\'evy measure corresponding to $\gamma_{\kappa,\rho}$.
The exchangeable min-id sequence $\tilde{\bmX}\in(0,\infty)^\N$ associated with $H$ has survival function
\begin{align*}
&\pp\lc \tilde{X}_1>t_1,\tilde{X}_2>t_2,\ldots\rc=\ee\lk \exp\lc-(M+b)\sum_{i\in\N} t_i\rc\rk=C_{\tilde{\psi}}\big( \tilde{\psi}(t_1),\tilde{\psi}(t_2),\ldots\big),\\
&\ \bmt \in [0,\infty)^\N,
\end{align*}
where $\tilde{\psi}(t):=\ee\lk\exp(-(b+M)t\rk=\exp(-bt)\psi(t)$ is log-completely monotone. Therefore, $\tilde{\bmX}$ has Archimedean survival copula $C_{\tilde{\psi}}$ and marginal survival function $\tilde{\psi}$. It is easy to see that with appropriate choices of the tuple $(b,\upsilon_M)$, resp. $(b,\psi)$, every log-completely monotone generator of an exchangeable Archimedean copula on $[0,1]^\N$ may be obtained by this construction.

The Archimedean copula $C_{\tilde{\psi}}$ itself is max-id, since decreasing transformations of min-id sequences are max-id by Lemma \ref{lemmatrafomaxtominid}. On the other hand, if we assume that an Archimedean copula $C_{\tilde{\psi}}$ is max-id, it is easy to deduce that $\tilde{\psi}$ necessarily corresponds to the Laplace transform of a non-negative infinitely-divisible random variable. Thus, an Archimedean copula on $[0,1]^\N$ is max-id if and only if $\tilde{\psi}$ is log-completely monotone.\footnote{Note that max-id Archimedean copulas are not the only positive lower orthant dependent (PLOD) Archimedean copulas. The set of completely monotone generators, a superset of log-completely monotone generators, corresponds to the class of extendible Archimedean copulas that are PLOD. Thus, a PLOD Archimedean copula is not necessarily max-id.}

\subsection{Subordination of L\'evy process by extended chronometers}
\cite{BarndorffNielsen2006InfiniteDF} prove that a L\'evy process subordinated by an extended chronometer remains infinitely divisible. Formally, let $(L_t)_{t \geq 0}$ denote a L\'evy process and let $(H_t)_{t \geq 0}$ denote an extended chronometer. If $L$ is a subordinator $(Y_t)_{t \geq 0}:=(L_{H_t})_{t \geq 0}$ defines an extended chronometer by \cite[Theorem~ 7.1]{BarndorffNielsen2006InfiniteDF}. It can be shown that $Y$ has independent increments if $H$ has independent increments and that $Y$ is strong-idt if $H$ is strong-idt. 

Except for the quite simple example of (reciprocal) Archimedean copulas we have only seen examples of (killed) strong-idt and (killed) additive subordinators. An example of a c\`adl\`ag id-process which is (usually) neither strong-idt nor has independent increments is given by
the non-negative solution to the Ornstein--Uhlenbeck type stochastic differential equation
\begin{align}
\label{eqnouprocess}
\mathrm{d}V_t=-\lambda V_t\mathrm{d}t+\mathrm{d}Z_{\lambda t},
\end{align}
where $\lambda>0$ and $(Z_t)_{t \geq 0}$ is a L\'evy process satisfying 
$$\int_{D^\infty\big([0,\infty)\big)}\max\{0,\log(\vert x(t)\vert )\}\nu_Z(\mathrm{d}x)<\infty$$ for all $t \geq 0$, where $\nu_Z$ denotes the L\'evy measure of $Z$. In case $(Z_t)_{t \geq 0}$ is a L\'evy subordinator, $V$ is the square of the stochastic volatility process of the Barndorff-Nielsen--Shepard model, see \cite{carr2003stochastic,barndorffnielsenshepard}. Moreover, the process $H_t:=\int_0^t V_s\mathrm{d}s$ remains infinitely divisible \cite[Section 4]{BarndorffNielsen2006InfiniteDF}. $H_t$ is known as the integrated or cumulated volatility up to time $t$, which is important when analyzing the realized variance or the quadratic variation of a pricing model, see \cite{Komm2016}. E.g.\ \cite{duan2010} model the CBOE Volatility Index at time $t$ via $const.+\frac{1}{t}\ee_Q\lk\int_0^t V_s \mathrm{d}s\rk$, where $\ee_Q$ denotes expectation w.r.t.\ some risk neutral probability measure $Q$. Interestingly, we can generalize the ideas of \cite{barndorffnielsenshepard} and \cite[Example (2.2)]{mansuy2005processes} to id-processes to obtain that
$$\lc H^{(\kappa)}_t\rc_{t\geq 0} :=\lc \int_0^t V_s \kappa(\mathrm{d}s)\rc_{t\geq 0}$$
defines an extended chronometer for every non-negative c\`adl\`ag id-process $V$ and Radon measure $\kappa$ on $[0,\infty)$, which is usually neither strong-idt nor has independent increments.  

%\textcolor{red}{L\"OSCHEN ODER? By Lebesgue's decomposition Theorem $\kappa$ can be decomposed into a singular measure $\kappa_s$ and an absolutely continuous measure $\kappa_c$ (w.r.t.\ the Lebesgue measure). Moreover assume that $\kappa_s$ is a pure point measure, i.e.\ $\kappa_s=\sum_{i\in\N} c_i\delta_{x_i}$. In this case $$H^{(\kappa)}_t=\int_{0y}^t V_s f_{\kappa_c}(s)\mathrm{d}s+\sum_{x_i\leq t}c_i V(x_i).$$
%The decomposition has the following interpretation: $\int_{-\infty}^t V_s f_{\kappa_c}(s)\mathrm{d}s$ is the continuous part of $H^{(\kappa)}$, whereas $\sum_{x_i\leq t}c_i V(x_i)$ is the jump part of $H^{(\kappa)}$ with random jump heights $c_i V(x_i)$. }

\begin{prop}
\label{propositionintegratedidprocess}
Let $(V_s)_{s \geq 0}\in D^\infty \lc[0,\infty) \rc $ denote a non-negative c\`adl\`ag id-process with drift $b_V$ and L\'evy measure $\nu_V$. Moreover, let $\kappa$ denote a measure on $[0,\infty)$ such that $\kappa\lc[0,t]\rc<\infty$ for all $t \geq 0$. Then,
$$(H^{(\kappa)}_t)_{t \geq 0}:=\lc \int_{0}^t V_s \kappa(\mathrm{d}s)\rc_{t \geq 0}$$
defines an extended chronometer with L\'evy measure 
$$\nu_\kappa\lc A\rc:=\nu_V \lc \bigg\{ x\in D^\infty\big([0,\infty)\big) \ \bigg\vert\ \int_{0}^\infty x(s)\kappa(\mathrm{d}s)\in A,\ \int_{0}^\infty x(s)\kappa(\mathrm{d}s)>0\bigg\}\rc $$
for all $A\in \bc\lc D^\infty \lc [0,\infty)\rc \rc$ and drift $b^{(\kappa)}(\cdot)=\int_{0}^\cdot b_V(s)\kappa(\mathrm{d}s)$.
\end{prop}

\begin{proof}
The proof can be found in Appendix \ref{appendixproofs}.
\end{proof}

If we subordinate a L\'evy subordinator $L$ by $H^{(\kappa)}$ we can obtain an explicit representation of the survival function of the associated exchangeable min-id sequence $\bmX^{(L,V,\kappa)}$. Let $\psi(a):=-\log\lc \ee\lk \exp\lc-aL_1\rc\rk\rc$ denote the Laplace exponent of $L$. The min-id sequence $\bmX^{(L,V,\kappa)}$ associated with $\lc L_{H^{(\kappa)}_t}\rc_{t \geq 0}$ has the following survival function:
\begin{align*}
    &\pp\lc X_1>t_1,\ldots ,X_d>t_d\rc = \ee\lk \exp\lc -\sumd L_{ H^{(\kappa)}_{t_i}} \rc \rk \\
    &=\ee\lk \ee\lk\exp\lc  -\sumd L_{ H^{(\kappa)}_{t_i}} \rc \bigg\vert H \rk \rk\\
    &=\ee\lk -\exp\lc \sumd H^{(\kappa)}_{t_i}\big( \psi(d-i+1)-\psi(d-i) \big) \rc \rk\\
    &=\exp\Bigg( - \sumd \big( \psi(d-i+1)-\psi(d-i) \big)b^{(\kappa)}_{t_i}   \\
    &-\int_{M_\ell} 1-\exp\lc - \sumd \big( \psi(d-i+1)-\psi(d-i) \big) x^{(\kappa)} _{t_i}\rc \nu_V(\mathrm{d}x)\Bigg),
\end{align*}
where $x^{(\kappa)}_t:=\int_0^t x(s)\kappa(\mathrm{d} s) $. Thus, if $\nu_V$ and $\psi$ are known, we obtain an explicit analytic representation of the survival function of $\bmX^{(L,V,\kappa)}$. In particular this is the case if we start with ``simple'' processes $L$ and $V$. E.g.\ choosing $V$ as a Cox--Ingersoll--Ross process, see \cite{cir1985} yields an explicit representation of the surival function of $\bmX^{(L,V,\kappa)}$, since the L\'evy measure of such processes can be obtained by an application of Proposition \ref{propositionintegratedidprocess} to a scaled and time changed squared Bessel process, see \cite[Example 2.24]{rosinskiinfdivproc} for more details on the L\'evy measure of squared Bessel processes.  

Generally, this approach yields a flexible way to combine two ``simple'' extended chronometers to an extended chronometer which is usually neither strong-idt nor additive.

\subsection{Finite exponent measures}
%This example investigates exchangeable min-id random vectors and sequences with finite exponent measures and, as a byproduct, yields that $\bmX$ can satisfy $\rho_d^u=1$ for every $d\in\N$ without satisfying $\bmX=(\bar{X},\bar{X},.\ldots)$.

Let $\bmX\in\minf^\N$ denote a min-id sequence with $0<\pp(\bmX=\binfty)<1$. The associated global exponent measure $\mu$ is a finite measure on $E^\N_\binfty$ with total mass $c:=-\log\lc \pp\lc \bmX =\binfty\rc\rc$. \cite[Example 5.6]{resnickextreme} shows that there exists a sequence of i.i.d.\ sequences $(\bm Z^{(i)})_{i\in\N}\in\minf^{\N\times \N}$ with $\bm Z^{(1)}\sim \bm Z\sim \mu/c=:\tilde{\pp}$ and an independent Poisson random variable $N$ with mean $c$ such that $ \min_{1\leq i\leq N} \bm Z^{(i)}\sim \bmX$. This can be easily verified by
\begin{align*}
    \pp\lc  \min_{1\leq i\leq N} \bm Z^{(i)}>\bmx\rc&=\exp\lc-c\rc \sum_{i\in\N} \frac{ c^i \tilde{\pp}\lc \bm Z >\bmx\rc^i}{i!}=\exp\lc -c\lc 1-\tilde{\pp}\big( (\bmx,\binfty]\big) \rc \rc \\
        &=\exp\lc -\mu\lc (\bmx,\binfty]^\complement\rc  \rc =\pp\lc \bmX>\bmx\rc.
\end{align*}
Obviously, $\bmX$ is exchangeable if and only if $\bm Z$ is exchangeable. Then, Theorem~\ref{theoremexponentmeasureisiidmixture} tells us that $\mu$ can be decomposed into $\mu=\mu_{b}+\mu_\gamma$. Note that the existence of some $t\in\RR$ such that $$\mu_{b,n}\lc\big((\underbrace{t,\ldots,t}_{n\text{ times}}),\binfty\big]^\complement\rc=nb(t)\xrightarrow{n\to\infty}\infty$$
is equivalent to $b\not=0_{D^\infty(\RR)}$. Therefore, if $\mu$ is finite, $\mu_{b}=0$ and the associated id-process $H$ is driftless. Thus, $\mu$ is given by $\mu=\mu_\gamma=\int_{\overline{M}^0_\infty} \otimes_{i\in\N} \pp_G\gamma(\mathrm{d}G) $. Furthermore, Theorem~\ref{theoremexponentmeasureisiidmixture} implies the existence of a unique L\'evy measure $\nu$ such that $\nu(A)=\gamma(\{G\in \overline{M}^0_\infty\mid G=1-\exp(-x(\cdot))\text{ for some } x\in A\})$. An application of the monotone convergence theorem shows that
\begin{align*}
    c&=\mu\lc E^\N_\binfty\rc=\lim_{t\to\infty}\mu\lc (\bm t,\binfty]^\complement\rc =\lim_{t\to\infty}\int_{M^0_\infty} 1-\prod_{i\in\N} \exp\lc-x(t)\rc\nu(\mathrm{d}x)\\
    &\underset{x\not= 0}{=}\nu\lc M^0_\infty\rc,
\end{align*}
i.e. $\nu$ is finite as well. Now, Example~\ref{examplefinitelevymeasure} implies that the associated (driftless) id-process $H$ is given by 
$$H\sim \sum_{i=1}^N h_i,$$
where $(h_i)_{i\in\N}$ are i.i.d.\ stochastic processes on $M^0_\infty$ with distribution $\nu/c$ and $N$ is an independent Poisson random variable with mean $c$. Therefore, if $\mu$ is finite, $\bmX$ (resp.\ $H$) admits a simple construction method via Poisson maxima (resp.\ sums) of i.i.d.\ objects.

The correspondence of the exchangeable sequence $\bm Z$ with the finite exponent measure $\mu$ seems to imply that finite extendible exponent measures $\mu_d$ can be represented via extendible random vectors. However, as the following paragraphs show, this is not always possible, since the global exponent measure $\mu$ can have infinite mass even if $\mu_d$ is finite for every $d\in\N$. To see this assume that $\bmX_d\in\minf^d$ is exchangeable and min-id with finite exponent measure $\mu_d$. Moreover, assume that $\bmX_d$ is extendible to an exchangeable min-id sequence $\bmX$, which w.l.o.g.\ satisfies Condition $(\Diamond)$. This ensures the existence of a global exponent measure $\mu$ such that $\mu_d$ satisfies the properties in Proposition \ref{propositionextendibleexponentmeasure}.

\cite[Example 5.6]{resnickextreme} and Proposition~\ref{propositionexchangeabilityexponentmeasure} imply that there exists a Poisson random variable $N_d$ with mean $c_d=\mu_d\lc E^d_\binfty\rc$ and i.i.d.\ exchangeable random vectors $(\bm Z^{(i)}_d)_{i\in\N}\in\minf^d$ with distribution $\mu_d/c_d$ such that $\bmX_d\sim \min_{1\leq i\leq N_d} \bm Z^{(i)}_d$. Therefore, $\mu_d$ can always be represented by a random vector $\bm Z^{(1)}_{d}\sim \bm Z_d\sim \mu_d/c_d$ and a constant $c_d$. 

In the elaborations above we have seen that a finite global $\mu$ implies that $\bm Z_d$ is extendible to a sequence $\bm Z$. Since we know that $\mu_d$ is extendible to a global $\mu$ it would be tempting to assume that $\bm Z_d$ is also extendible to a sequence $\bm Z$, independent of the total mass of $\mu$. However, as the following calculations show, $\mu$ being finite is also a necessary condition for $\bm Z_d$ being extendible.

We begin with an analysis of $\mu_d$. Independently of the total mass of $\mu$ it is always possible to decompose $\mu_d$ into $\mu_{b,d}+\mu_{d,\gamma}$ by Theorem~\ref{theoremexponentmeasureisiidmixture}. Therefore, we can define the constant  $a_d=\mu_{d,\gamma}(E^d_\binfty)/c_d=1-\mu_{b,d}(E^d_\binfty)/c_d\in[0,1]$. Now, $\mu_d$ can be generated as follows:
\begin{enumerate}
    \item Draw a Bernoulli random variable $B$ with success probability $a_d$.
    \item If $B=1$, draw a random variable $\bm Y^{(d,1)}$ with distribution $\mu_{d,\gamma}(\cdot)/(c_d a_d)$.
    \item If $B=0$, draw a random variable $\bm Y^{(d,2)}$ with distribution $\mu_{b,d}(\cdot)/\big(c_d(1-a_d)\big)$. Note that $\bm Y^{(d,2)}$ is supported on $\{ x\in\minf^d \mid x_i=\infty \text{ for all but one }i\}$.
    \item $\mu_d(A)=c_d\pp(B\bm Y^{(d,1)}+(1-B)\bm Y^{(d,2)}\in A).$
\end{enumerate}
Taking a closer look at $\bm Y^{(d,2)}$ reveals that $\bm Y^{(d,2)}$ can never be embedded in a random vector on $\minf^{d^\prime}$, $d^\prime>d$, if we do not allow for mass on $\times_{i=1}^d \{\infty\}$. Unfortunately, even if we allow for mass on $\times_{i=1}^d \{\infty\}$, an embedding of $\bm Y^{(d,2)}$ in a sequence requires $\mu_d(\{\infty\}^d)=\infty$ and $\mu(E^\N_\binfty)=\infty$, since $\mu_{d,b}(E^d_\binfty)=d\lim_{t\to\infty}b(t)\xrightarrow{d\to\infty}\infty$ if and only if $b\not = 0_{D^\infty(\RR)}$. Therefore, $\bm Y^{(d,2)}$ can never be extended to a sequence and the random variable $\bm Z_d=\bm Y^{(d,1)}+\bm Y^{(d,2)}$ can only be extendible if $a_d=0$. Note that the necessity of $a_d=0$ is not based on the fact that $\bm Y^{(d,2)}$ is not extendible, but rather on the fact that the presence of $\bm Y^{(d,2)}$ implies that the only possible extension of the distribution of $\bm Y^{(d,2)}$ is an infinite measure. Intuitively, this resembles the fact that $\bmX$ cannot have finite exponent measure $\mu$ if it can be represented as the minimum of a non-trivial i.i.d.\ sequence (represented by $\bm Y^{(d,2)}$) and an independent exchangeable min-id sequence (represented by $\bm Y^{(d,1)}$).

It remains to investigate under which circumstances $\bm Y^{(d,1)}$ is extendible. Observe that in case $a_d=0$ the extendibility of $\bm Y^{(d,1)}$ to a sequence $\bm Y^{(1)}$ implies that $\mu$ is finite. Therefore, $\mu$ being finite is not only a sufficient but also necessary criterion for the extendibility of $\bm Z_d$. Note that the distribution of the first $d$ components of $\bm Y^{(1)}$ may not be exactly $\mu_d/c_d$ due to the removal of $\times_{i=1}^d \{\infty\}$ and the possibility of $Y^{(1)}_1=\ldots =Y^{(1)}_d=\infty$. However, the distribution of  $\bm Z_d$ can be obtained as the conditional distribution of $(Y^{(1)}_1,\ldots,Y^{(1)}_d)$ given $(Y^{(1)}_1,\ldots,Y^{(1)}_d)\not =\binfty$. 

Our discussion is summarized in the following paragraph. Let $(E_i)_{i\in\N}$ denote a sequence of unit exponential random variables. The global extensions $\mu$ of $\mu_d$ can be classified into two possible cases:

\begin{enumerate}
    \item[Finite $\mu$:] In this case $\bm Z_d$ is extendible to an exchangeable sequence $\bm Z$ and the global exponent measure $\mu$ is in one-to-one correspondence with the tuple $(\bm Z,\pp(\bmX=\binfty))$. It is easy to see that the associated extended chronometer $H$ is driftless with $\pp\lc H=0_{D^\infty(\RR)}\rc>0$.
    \item[Infinite $\mu$:] In this case $\bm Z_d$ is not extendible. Nevertheless, we know that $\mu_d$ is extendible to a global $\mu$. Thus, $\pp(\bmX=\binfty)=\exp\lc-\mu\lc E_\binfty^\N\rc \rc=0$ and the associated extended chronometer $H$ satisfies 
    $$\pp\lc H=0_{D^\infty(\RR)}\rc=\pp\lc \big\{E_i>\lim_{t\to\infty} H_t \text{ for all } i\in\N\big\}\rc=\pp(\bmX=\binfty)=0.$$
     Additionally, the drift of $H$ satisfies $\lim_{t\to\infty}b(t)<\infty$, since $\lim_{t\to\infty}b(t)=\infty$ would require that $\pp(\bmX_d=\binfty)\leq \lim_{t\to\infty}\exp(-db(t))=0$. Moreover, we can deduce that 
     \begin{align*}
         0&<\pp(\bmX_d=\binfty)=\pp\lc \big\{E_i>\lim_{t\to\infty} H_t \text{ for all } 1\leq i\leq d\big\}\rc\\
         &\leq \pp\lc 0<\lim_{t\to\infty} H_t<\infty\rc.
     \end{align*}
     Therefore, $H$ is almost surely non-zero and bounded with positive probability.
\end{enumerate}

\subsection{Approximation of extended chronometer via extended chronometers with finite L\'evy measure}
\label{subsecapproxextchronfinlevymeasure}
Example \ref{examplefinitelevymeasure} shows that a driftless extended chronometer $H$ with finite L\'evy measure $\nu$ may be simulated via finitely many i.i.d.\ copies $\lc h_i\rc_{i\in\N}$ of a stochastic process $h\sim\nu(\cdot)/\nu\lc D^\infty(\RR)\rc$. $H$ may then be represented as $(H_t)_{t\in\RR}\sim \lc\sum_{i=1}^{\tilde{N}} h_i(t)\rc_{t\in\RR}$, where $\tilde{N}$ denotes a Poisson random variable with mean $\nu\lc D^\infty(\RR)\rc$. 
%In general, this representation is only useful when the simulation of the stochastic processes $h$ is simpler than the direct simulation of the extended chronometer $H$.

In case $\nu$ is an infinite measure, Proposition \ref{propositionltchronometer} implies that there exists a Poisson random measure $N:=\sum_{i\in\N} \delta_{h_i}$ on $\{ x\in D^\infty(\RR) \mid x \text{ nnnd}\} $ with intensity measure $\nu$ such that
$$\lc H_t \rc_{t\in\RR}:= \lc \sum_{i\in \N} h_i(t)\rc_{t\in\RR},$$
defines a nnnd id-process with L\'evy measure $\nu$ \cite[Proposition 2.10]{rosinskiinfdivproc}.
Clearly, for every $s\in\RR$ and $\epsilon>0$, the extended chronometer defined via
$$\lc H^{(s,\epsilon)}_t \rc_{t\in\RR}:=\lc \sum_{ \substack{i\in \N \\ h_i(s)>\epsilon}} h_i(t)\rc_{t\in\RR},$$
has finite L\'evy measure $\nu_{s,\epsilon}:=\nu\lc \cdot\cap \{ x(s)>\epsilon\}\rc$. This implies that the nnnd id-process
$$\lc H_t^{(s)}\rc_{t\in\RR}:=\lc \lim_{\epsilon\to 0} H^{(s,\epsilon)}_t\rc_{t\in\RR}=\lc\sum_{ \substack{i\in \N \\ h_i(s)>0}} h_i(t) \rc_{t\in\RR}$$ 
has L\'evy measure $\nu_{s,0}:=\nu\lc \cdot\cap \{ x(s)>0\}\rc$, which is generally an infinite measure that is not equal to $\nu$. However, it is important to observe that 
$$\lc H_t \rc_{t\leq s}= \lc H_t^{(s)}\rc_{t\leq s}.$$
Thus, if one is only interested in the path of the extended chronometer $H$ up to time $s$, one may simulate $H^{(s,\epsilon)}$ for some small $\epsilon>0$ to obtain an approximation of $\lc H_t\rc_{t\leq s}$\footnote{A similar reasoning also allows to approximate a general id-process by id-processes with finite L\'evy measure.}. 

The just described approximation procedure of an extended chronometer is well-known when $H$ is an additive process. In this case, $\nu_{s,\epsilon}$ can be represented as a finite Poisson random measure $\lc (U_i,A_i)\rc_{1\leq i\leq \tilde{N}}$ on $[0,s]\times (\epsilon,\infty]$ and $\lc H^{(s,\epsilon)}_{t}\rc_{t\in\RR}=\lc \sum_{1\leq i\leq \tilde{N}} A_i\id_{\{t\geq U_i\}}\rc_{t\in\RR}=\lc \sum_{U_i\leq t} A_i\rc_{t\in\RR}$ is known as a compound Poisson process. 
However, note that nnnd id processes with independent increments are the only nnnd id processes which may be approximated by compound Poisson processes $H^{(s,\epsilon)}$. This is due to the fact that compound Poisson processes have independent increments and the independent increments property is preserved under the just described approximation scheme when $\epsilon\to 0$.

The exchangeable min-id sequences $\bmX^{(s,\epsilon)}$, $\bmX^{(s)}$ and $\bmX$ associated to $H^{(s,\epsilon)}$, $H^{(s)}$ and $H$ obviously satisfy $\pp\lc \bmX^{(s,\epsilon)}>\bmt\rc\geq \pp\lc \bmX^{(s,0)} >\bmt\rc \geq \pp\lc \bmX>\bmt\rc$ for all $\bmt\in[-\infty,\infty)^\N$. Moreover, $\pp\lc \bmX^{(s,0)} >\bmt\rc = \pp\lc \bmX>\bmt\rc$ for all $\bmt\in[-\infty,s]^\N$, since $\lc H^{(s)}\rc_{t\leq s}= \lc H_t\rc_{t\leq s}$. Therefore, one may approximate the exchangeable min-id sequence $\bmX$ associated to the extended chronometer $H$ via a simulation of $\bmX^{(s,\epsilon)}$ associated to the extended chronometer $H^{(s,\epsilon)}$ with finite L\'evy measure.

%Note that it may happen that $\bmX^{s,\epsilon}=\binfty$ if one simulates $\bmX^{s,\epsilon}$ for an a priori fixed $\epsilon>0$. As a caveat, we want to remark that it is generally not possible to avoid this unpleasant behavior by discarding all realizations $\bmX^{s,\epsilon}=\binfty$ while only keeping those simulations where $\bmX^{(s,\epsilon)}\not=\infty$, since this would provide a simulation of the sequence $\bmX^{(s,\epsilon)}$ conditioned on the fact that $\bmX^{(s,\epsilon)}\not=\binfty$, whose distribution may not be min-id anymore. However, if one is able to consecutively simulate Poisson random measures with intensities $\nu\lc \cdot\cap \{ \epsilon_i\leq x(s)>\epsilon_{i+1}\}\rc$ for some strictly decreasing sequence of positive numbers $(\epsilon_i)_{i\in\N}$ such that $\epsilon_1=\infty$, one may stop the simulation as soon as ''sufficiently``\footnote{''sufficiently`` is to be understood in the sense that the stopping rule cannot depend on any properties of simulated Poisson random measures.} many atoms (functions) of the Poisson random measures were simulated and one obtains a realization of $\bmX$ which satisfies $\bmX\not=\infty$. \textcolor{red}{Is this correct?}

\subsection{Overview of established families under the present umbrella}
Figure \ref{Fig_subfamilies} provides a graphical overview of the established families of exchangeable min-id sequences introduced in this paper. 
\begin{figure}[H]
    \centering
\begin{tikzpicture}[scale=0.9]
\draw (3,10.5) node[right=1pt] {\bf{Exchangeable sequences}};
% big
\path[draw=black](3.0,0.0) rectangle (16.0,10.0);
\draw (3,0.5) node[right=1pt] {Min-id};

% Additive 
\path[draw=yellow,fill=yellow!80!black,opacity=0.5](4.0,1.0) rectangle (10.0,6.0);
\draw (4,1.5) node[right=1pt] {Exogenous shock model};

% Sato
\path[fill=green!80!black,opacity=0.5](8.0,2.0) rectangle (10.0,6.0);
\draw (8.25,3.5) node[right=1pt] {Sato-};
\draw (8.25,3) node[right=1pt] {frailty};

% EVC
\path[draw=red,thick,fill=red!80!black,opacity=0.5](5.0,4.0) rectangle (12.0,8.0);
\draw (5,7.5) node[right=1pt] {Min-stable};
% LFC
\path[draw=yellow,fill=yellow!80!black,opacity=0.5](5.0,4.0) rectangle (10.0,6.0);
\draw (5,4.5) node[right=1pt] {Marshall--Olkin};

% Archimedean
\path[draw=blue,fill=blue!80!black,opacity=0.5](10.0,2) rectangle (15.0,6.0);
\draw (11,3) node[right=1pt] {Archimedean copulas};
\draw (11,2.5) node[right=1pt] {(log-c.m. generator)};

% Gumbel
\draw (10,5.5) node[right=1pt] {Gumbel};
\draw (10,5) node[right=1pt] {copula};

% reciprocal Archimedean
\path[draw=brown,fill=brown!80!black,opacity=0.5](10.0,6) rectangle (15.0,9);
\draw (11,8.5) node[right=1pt] {recip.\ Archim.\ copula};

% Galambos
\draw (10,7.5) node[right=1pt] {Galambos};
\draw (10,7) node[right=1pt] {copula};

% Alpha-stable sub
\draw (8.0,5) node[right=1pt] { $\alpha$-stable};

%
%\path[draw=pink,fill=pink!80!black,opacity=0.5](10.0,0.5) rectangle (15.0,2.0);
%\draw (11,1.2) node[right=1pt] {Narrow sense geometric};

\draw[fill=gray!75] (10,6) circle (2pt);
\draw[->] (9.5,9.25) -- (10,6);
\draw (9.5,9.25) node[above=1pt] {independence \& comonotonicity};

\end{tikzpicture}
\caption{Overview of established exchangeable min-id sequences under the present umbrella.}
\label{Fig_subfamilies}
\end{figure}
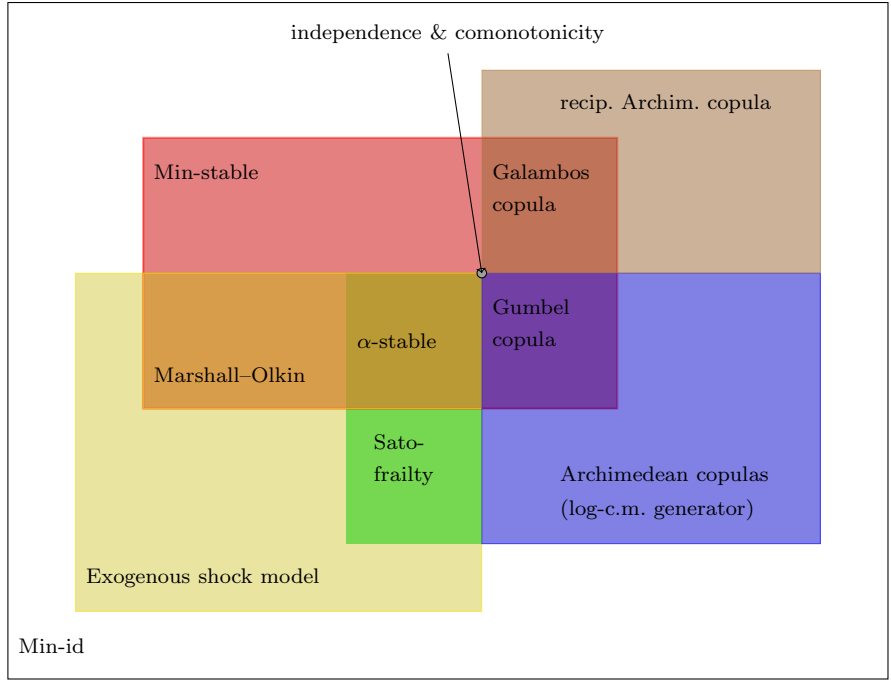

Section \ref{subsectiongenmarsholkin} provides that the class of extendible exogenous shock models is contained in the class of exchangeable min-id sequences. Since extendible exogenous shock models are associated to the class of nnnd additive processes \cite{Sloot2020b,MaiSchenkSchererexchangeableexshock}, we deduce that the class of extendible Marshall-Olkin distributions is contained in the class of extendible exogenous shock models as the class of extendible Marshall-Olkin distributions corresponds to the class of nnnd L\'evy processes \cite{janphd}. Furthermore, the class of exchangeable Sato-frailty sequences is included in the class of extendible exogenous shock models as it is associated to the class of (additive) Sato-subordinators \cite{Maiselfdecomp}. Moreover, the class of exchangeable Sato-frailty sequences only intersects with the class of Marshall-Olkin distributions when the associated Sato-subordinator is the $\alpha$-stable subordinator, which is the only Sato-subordinator that is also a L\'evy process \cite[Remark 16.2]{satolevyinfinitely}. 

Section \ref{subseccomumbrellareciparchimminstable} shows that the class of exchangeable min-stable sequences with exponential margins is contained in the class of exchangeable min-id sequences and \cite{MaiSchererexMSMVE} show that it corresponds to the class of nnnd strong-idt processes.  Since the class of nnnd L\'evy processes is precisely the intersection of the classes of nnnd additive processes and nnnd strong-idt processes, the class of extendible Marshall--Olkin distributions is precisely the intersection of the class of extendible exogenous shock models and the class of exchangeable min-stable sequences with exponential margins. Furthermore, \cite{genestreciparchim} show that the Galambos copula is the only copula which belongs both to the class of copulas of exchangeable min-stable sequences (also called extreme value copulas) and to the class of reciprocal Archimedean copulas. Similarly, \cite{genestrivest1989} show that the Gumbel copula is the only copula which belongs both to the class of extreme value copulas and to the class of Archimedean copulas.

The class of extendible exogenous shock models does not intersect with the class of Archimedean copulas with log-completely monotone generator except for independence and comonotonicity, since the nnnd id-process associated to an Archimedean copula with log-completely monotone generator cannot have independent increments according to Section \ref{exampleextarchimcop}. A similar statement applies to reciprocal Archimedean copulas according to Section \ref{subsecreciparchim}. Comparing the L\'evy measures associated to reciprocal Archimedean copulas  and Archimedean copulas with log-completely monotone generator, we obtain that their intersection can only consist of independence and comonotonicity as well.

\section{Conclusion}
\label{sectionconclusion}

We have shown that every exchangeable min-id sequence is in one-to-one correspondence with a nnnd infinitely divisible c\`adl\`ag process. Doing so, we have unified the work of \cite{MaiSchererextendibilityMO,MaiSchererexMSMVE,MaiSchenkSchererexchangeableexshock,Sloot2020b,janphd,Maiselfdecomp} under one common umbrella. Furthermore, we have shown that the exponent measure of an exchangeable min-id sequence associated to a driftless nnnd id process is a mixture of product probability measures. Therefore, de Finetti's Theorem is extended to exchangeable exponent measures. Several important examples of exchangeable min-id sequences have been presented and the existing literature has been embedded into our framework. A summary of these correspondences is given in Figures \ref{Venndiagcorrespondencesminididexpmeasurelevymeasure}, \ref{Venndiagprocessvssequence} and \ref{Fig_subfamilies}. As a byproduct we have shown that c\`adl\`ag id-processes can be represented as the sum of arbitrarily many i.i.d.\ c\`adl\`ag processes and that the L\'evy measure of nnnd c\`adl\`ag id-processes is concentrated on nnnd c\`adl\`ag functions.

There are various well known subclasses of nnnd infinitely divisible stochastic processes, such as additive and strong-idt processes. However, it remains an interesting open problem to find nnnd infinitely divisible stochastic processes outside of these subclasses, which can be conveniently described analytically.
Moreover, even if one would have an analytical characterization of such nnnd id processes at hand, their simulation would probably only be feasible in an approximate fashion, similar to the approach discussed in Section \ref{subsecapproxextchronfinlevymeasure}. Nonetheless, even their approximate simulation would most likely still pose a quite challenging problem due to possible path dependencies of the involved stochastic processes. Thus, the simulation of the associated exchangeable min-id sequence remains challenging is such situations.

To circumvent this problem, \cite{brueckexactsim2022} provides a simulation algorithm for the class of real-valued continuous max-id processes, which ensures that a user-specified number of locations of the max-id process are simulated exactly. Since every min-id sequence $\bmX$ may be transformed into a continuous max-id-process $1/\bmX$ with index set $\N$, the results of \cite{brueckexactsim2022} may be applied to simulate an exchangeable min-id sequence $\bmX$. The key ingredient of the proposed simulation algorithm in \cite{brueckexactsim2022} is the exponent measure of $\bmX$ (or $1/\bmX$), which can be easily constructed according to (\ref{eqnconstmugamma}) or deduced from the L\'evy measure of the associated nnnd id-process by Theorem \ref{theoremexponentmeasureisiidmixture}. Thus, the results of this paper allow to construct exchangeable min-id sequences in terms of their associated exponent or L\'evy measures, while \cite{brueckexactsim2022} provides a method for the exact simulation of their $d$-dimensional margins.

\begin{appendix}
%\appendix
\section{Proofs}
\label{appendixproofs}

\subsection{Proof of Corollary \ref{corcontdistrcorresstochcontchron}}
\label{approofcorcontdistrcorresstochcontchron}
\begin{proof}
\textcolor{white}{a}
\begin{itemize}
\item[``$\Leftarrow$'']
Assume that $H$ is stochastically continuous. From the property that $H$ is non-decreasing and c\`adl\`ag with $\lim_{s\nearrow t}H_s= H_t$ in probability we obtain that $\lim_{s\nearrow t}H_s= H_t$ almost surely. This is due to the fact that there is at least one sequence $(s_n)_{n\in\N}$ with $s_n\nearrow t$ such that $\lim_{n\to\infty}H_{s_n}=H_t$ almost surely, since $\lim_{s\nearrow t}H_s= H_t$ in probability. This immediately implies that every $(s^\prime_n)_{n\in\N}$ with $s^\prime_n\nearrow t$ satisfies $\lim_{n\to\infty} H_{s^\prime_n}=H_t$ almost surely due to the non-decreasingness of $H$. Therefore, the Laplace transforms of $\lim_{s\nearrow t}H_s$ and $H_t$ coincide, which implies that 
$$\pp(X_1\geq t) =\lim_{s\nearrow t} \pp(X_1>s)=\lim_{s\nearrow t}\ee\lk \exp(-H_s)\rk=\ee\lk \exp(-H_t)\rk =\pp(X_1>t).$$
Thus, the distribution of $\bmX$ is continuous.

\item[``$\Rightarrow$'']
Assume that $\bmX$ follows a continuous distribution. Seeking a contradiction we assume that $H$ is not stochastically continuous, i.e.\ there exists $t\in\RR$ and $\delta,\epsilon>0$ such that $\lim_{s\nearrow  t}\pp( H_t-H_s$ $>\delta)>\epsilon$. Note that the limit exists, because $H$ is increasing. Now, there exist $0<q_1< q_1+\delta<q_2$ and $\eta>0$ such that $\lim_{s\nearrow  t}\pp(H_s<q_1 , H_t\geq q_2)>\eta$. This implies that 
$$\pp\lc X_1=t\rc \geq \pp(E_1\in (q_1,q_2), \lim_{s\nearrow t}H_s<q_1,H_t\geq q_2)>\eta\pp(E_1\in (q_1,q_2))>0,$$
which is a contradiction. Therefore, $H$ must be stochastically continuous.
\end{itemize}
\end{proof}

\subsection{Proof of Proposition \ref{propositionexchangeabilityexponentmeasure}}
\label{appproofpropositionexchangeabilityexponentmeasure}
\begin{lem}
\label{permutationlemma}
The following properties are valid:
\begin{enumerate}
\item $\pi(\bmx)\in A \Leftrightarrow \bmx\in \pi^{-1}(A)$ and $\pi^{-1}(\bmx)\in A \Leftrightarrow \bmx\in \pi(A)$.
\item $\pi(\pi^{-1}(A))=A$.
\item $\pi(A)=\{\pi(\bmx)\mid \bmx\in A\}=\{\bm y \mid \pi^{-1}(\bm y)\in A\}=\{\bm y\mid  \bm y \in \pi(A)\}$.
\item $\bar{\RR}^d \setminus \pi(A)=\pi(\bar{\RR}^d\setminus A)$ and $\pi(B)\setminus\pi(A) =\pi(B\setminus A)$. 
\item $\pi(\cup_{i\in\N} A_i)=\cup_{i\in\N} \pi(A_i)$.\
\item Let $E_\bell$ be the support of an exchangeable exponent measure. Obviously, $E_\bell=\pi(E_\bell)$ and  $\pi\lc (\bmx,\binfty]^\complement\rc=(\pi(\bmx),-\binfty]^\complement$.
\end{enumerate}
\end{lem}

\begin{proof}[Proof of Lemma \ref{permutationlemma}]
\hspace{0.5cm}
\begin{itemize}
\item[1.\ + 2.] Obvious.
\item[3.] Follows from 1.
\item[4.] $\bar{\RR}^d \setminus \pi(A)=\{\bmx\mid x\not\in  \pi(A)\}=\{\bmx\mid \pi^{-1}(\bmx)\not\in A\}=\{\bmx\mid \pi^{-1}(\bmx)\in\bar{\RR}^d \setminus  A\}=\pi(\bar{\RR}^d \setminus  A)$. The second assertion follows analogously.
\item [5.] $\pi(\cup_{i\in\N} A_i)=\{\bmx\mid \pi^{-1}(\bmx)\in A_i \text{ for some }i\in\N\}=\{\bmx\mid \bmx\in \pi(A_i) \text{ for some }i\in\N\}$\\
$=\cup_{i\in\N} \pi(A_i)$.
\item[6.] Since $E_\bell$ is of the form $[-\binfty,\bell]\setminus \bell$ with $\bell=(\ell,\ldots,\ell)$ we get $\pi\lc (\bmx,\binfty]^\complement\rc=\pi\lc E_\bell\setminus (\bmx,\binfty]\rc=\{\pi(\bm y)\mid \bm y \in E_\bell , y_i\leq x_i \text{ for some } i\}=\{\bm y\mid  \bm y \in E_\bell , y_i\leq \pi(x_i) \text{ for some } i\}=(\pi(\bmx),\binfty]^\complement$.
\end{itemize}
\end{proof}

\begin{proof}[Proof of Proposition \ref{propositionexchangeabilityexponentmeasure}]
Throughout the proof we frequently use properties of the permutation operator $\pi$ derived in Lemma~\ref{permutationlemma}.
\begin{itemize}
\item[ ``$\Leftarrow$'']
The exchangeability of $\mu_d$ translates into the exchangeability of the survival function $\fb$ of $\bmX_d$, which in turn implies that $\bmX_d$ is exchangeable. We provide the precise reasoning behind this argument, since we need to carry out the same steps for finite exponent measures in  ``$\Rightarrow$''. For ease of notation we write $\pp(\bmX_d\in A)=:\pp(A)$. 

Let 
$$\ac:=\big\{A\in \bc\lc(-\infty,\infty]^d\rc \mid \pp(\pi(A))=\pp(A) \ \forall \text{ permutations } \pi \text{ on } \{1,\ldots,d \} \big\}$$
denote the collection of $\pp$-exchangeable sets. We show that $\ac$ is a Dynkin system containing a $\Pi$-stable generator of $\bc( (-\infty,\infty]^d)$, which implies that $\bc( (-\infty,\infty]^d)\subset \ac$. Obviously, the sets $(-\infty,\infty]^d$ and $\emptyset$ are both included in $\ac$. Now consider some arbitrary set $A\in\ac$. We get
\begin{align}
\pp\lc (-\infty,\infty]^d\setminus A\rc&=1-\pp(A)=1-\pp(\pi(A))=\pp\lc(-\infty,\infty]^d\setminus \pi(A)\rc \nonumber \\
&= \pp\lc\pi((-\infty,\infty]^d\setminus A)\rc. \label{eqnexchexpm}
\end{align}
Therefore, $(-\infty,\infty]^d\setminus A\in\ac$. Next, consider $(A_i)_{i\in\N}\in \ac$ with $A_i\cap A_j=\emptyset$ for $i\neq j$. Using that every measure is continuous from below we obtain
\begin{align*}
\pp\lc\cup_{i\in\N} A_i\rc&=\lim_{n\to\infty} \pp\lc\cup_{i=1}^n A_i\rc=\lim_{n\to\infty}\pp\lc\cup_{i=1}^n \pi(A_i)\rc \\
&=\pp\lc\cup_{i\in\N} \pi(A_i)\rc=\pp\lc\pi\lc\cup_{i\in\N} A_i\rc\rc,
\end{align*}
which establishes that $\ac$ is a Dynkin system. From the exchangeability of $\mu_d$ we deduce that
\begin{align*}
\pp\big( (\bmx,\binfty]\big)&=\fb(\bmx)=\exp\lc-\mu\lc(\bmx,\binfty]^\complement\rc\rc =\exp\lc-\mu\lc\pi\lc(\bmx,\binfty]^\complement\rc\rc\rc \\
&=\exp\lc-\mu\lc(\pi(\bmx),\binfty]^\complement\rc\rc =\fb(\pi(\bmx))=\pp\big((\pi(\bmx),\infty]\big)\\
&=\pp\big( \pi((\bmx,\binfty])\big).
\end{align*}
Therefore, $\ac$ contains a $\Pi$-stable generator of $\bc( (-\infty,\infty]^d)$. This implies that $\pp(A)=\pp(\pi(A))$ for all $A\in \bc( (-\infty,\infty]^d)$ and all permutations $\pi$ on $\{1,\ldots,d\}$, which proves that $\fb$ is the survival function of an exchangeable random vector.

\item[``$\Rightarrow$''] 
If $\mu_d$ is a finite measure, the proof is simply a repetition of the arguments in ``$\Leftarrow$'' to show that 
$$\ac:=\big\{A\in \bc\lc E_\binfty^d \rc \mid \mu_d(\pi(A))=\mu_d(A) \ \forall \text{ permutations } \pi \text{ on } \{1,\ldots,d \} \big\}$$
is a Dynkin system containing a $\Pi$-stable generator of $\bc\lc E_\binfty^d \rc$, since
\begin{align*}
\mu\lc(\bmx,\binfty]^\complement\rc&=-\log\lc \pp\big( (\bmx,\binfty]\big) \rc=-\log\lc \pp\big( \pi((\bmx,\binfty])\big) \rc\\
&=\mu\lc(\pi(\bmx),\binfty]^\complement\rc=\mu\lc\pi\lc(\bmx,\binfty]^\complement\rc\rc. 
\end{align*}
Therefore, we may focus on the case $\mu_d(E^d_\binfty)=\infty$. In this case the sets with infinite measure have non-empty intersection with an open neighborhood of $\binfty$ implying $\mu_d$ to be $\sigma$-finite. Note that a similar reasoning as in (\ref{eqnexchexpm}) cannot be applied here since $\mu_d\lc E^d_\binfty\rc=\infty$, which is why we have to resort to a different reasoning to prove that $\mu_d$ is exchangeable. Instead, our goal is to show that $\mu_d$ is the pointwise limit of a sequence of exchangeable finite measures. The following paragraphs are dedicated to the proof of this statement.

Let $(c_n)_{n\in\N}\in\RR$, $c_n<\infty$ with $c_n\nearrow \infty$. Define 
$$\mu^{(n)}(\cdot):=\mu_d\lc \cdot \cap \big\{(c_n,\infty]^d\big\}^\complement\rc.$$
We claim that $\mu^{(n)}$ is exchangeable and that $\mu_d=\lim_{n\to\infty} \mu^{(n)}$. Note, that $\mu^{(n)}$ is an increasing sequence of measures with $\mu^{(n)}(E^d_\binfty)=\mu_d\lc\{(c_n,\infty]^d\}^\complement\rc<\infty$. An application of \cite[Chapter 10, Theorem 1a)]{doob2012measure} yields that $\tilde{\mu}:=\lim_{n\to\infty} \mu^{(n)}$ is a measure.  The construction of $\mu^{(n)}$ and the continuity from below of any measure show that 
$$\tilde{\mu}(A)=\lim_{n\to\infty}\mu_d\lc A\cap \big\{(c_n,\infty]^d\big\}^\complement\rc=\mu_d\lc \bigcup_{n\in\N} \bigg\{A\cap \big\{(c_n,\infty]^d\big\}^\complement\bigg\}\rc=\mu_d\lc A \rc$$
for every $A\in \bc(E^d_\binfty)$, since $\binfty\not\in E^d_\binfty$.
Therefore $\mu_d=\tilde{\mu}=\lim_{n\to\infty} \mu^{(n)}$ is a pointwise limit of measures.

Next, we show that each $\mu^{(n)}$ is exchangeable. Consider the collection of $\mu^{(n)}$-exchangeable sets 
$$\ac_{\mu^{(n)}}:=\big \{A\in \bc(E^d_\binfty)\mid \mu^{(n)}(\pi(A))=\mu^{(n)}(A) \ \forall \text{ permutations }\pi \text{ on } \{1,\ldots,d\}\big\}.$$ Similar to ``$\Leftarrow$'', we can show that $\ac_{\mu^{(n)}}$ contains $E^d_\binfty$, $\emptyset$ and countable unions of pairwise disjoint sets from $\ac_{\mu^{(n)}}$. Moreover, for any $A\in\ac_{\mu^{(n)}}$, we have
\begin{align*}
\mu^{(n)}(E^d_\binfty\setminus A)&=\mu^{(n)}(E^d_\binfty)-\mu^{(n)}(A) =\mu^{(n)}(E^d_\binfty)-\mu^{(n)}(\pi(A))\\
&=\mu^{(n)}(E^d_\binfty\setminus \pi(A))=\mu^{(n)}(\pi(E^d_\binfty\setminus A)),
\end{align*}
since $\mu^{(n)}(A)<\infty$. Thus, $\ac_{\mu^{(n)}}$ is a Dynkin system. 
Moreover, for every $\bm a \in\RR^d$ with $\bm a\leq c_n$, we have 
\begin{align*}
    \mu^{(n)}\lc (\bm a,\binfty]^\complement\rc& =\mu\lc (\bm a,\infty]^\complement\rc=\mu\lc (\pi(\bm a),\infty]^\complement\rc\\
    &=\mu^{(n)}\lc (\pi(\bm a),\binfty]^\complement\rc =\mu^{(n)}\lc \pi\lc (\bm a,\binfty]^\complement\rc\rc 
\end{align*}
by the exchangeability of $\bmX_d$. One can invoke a similar argument for all $\bm a\in\RR^d$ which do not satisfy $\bm a\leq c_n$ to obtain $\mu^{(n)}\lc (\bm a,\binfty]^\complement\rc =\mu^{(n)}\lc \pi\lc (\bm a,\binfty]^\complement\rc\rc $.
An inclusion-exclusion principle argument yields that $A_{\mu^{(n)}}$ contains a $\Pi$-stable generator of $\bc(E^d_{\binfty})$. Therefore, $\mu^{(n)}$ is exchangeable. 

Combining the arguments above we have 
$$\mu_d(A)=\nlim \mu^{(n)}( A)=\nlim \mu^{(n)}(\pi(A))=\mu_d(\pi(A)),$$
for any $A\in\bc(E^d_\binfty)$, which shows that $\mu_d$ is exchangeable.
\end{itemize}
\end{proof}

\subsection{Proof of Corollary \ref{cortaildependence}}
\begin{proof}
We calculate
\begin{align*}
    \pp(X_1>t\mid X_2>t,\ldots ,X_{d^\prime}>t)&=\frac{\pp(X_1>t, X_2>t, \ldots ,X_{d^\prime}>t)}{\pp(X_2>t, \ldots ,X_{d^\prime}>t)}\\
    &=\frac{\exp\lc -\mu_{d^\prime} \lc  (\bm t,\binfty]^\complement\rc\rc} {\exp\lc -\mu_{{d^\prime}-1} \lc  (\bm t,\binfty]^\complement\rc\rc}\\
    &=\frac{  \exp\big(-\mu_{d^\prime} \big(  (\bm t,\binfty]^\complement\big)\big)  }{ \exp\lc -\mu_{{d^\prime}} \lc [-\infty,\infty] \times\lc ( t,\infty]^{d^\prime-1}\rc^\complement\rc\rc}\\
    &= \exp\lc- \mu_{d^\prime}\lc [-\infty, t]\times  (t,\infty]^{{d^\prime}-1}\rc\rc
\end{align*}
Now, $\rho^u_{d^\prime}$ is obtained by letting $t$ tend to $\infty$. The relation $\rho^u_{d^\prime+1}\geq \rho^u_{d^\prime}$ is obvious by Proposition \ref{propositionextendibleexponentmeasure} and
\begin{align*}
   & \exp\lc- \mu_{d^\prime}\lc [-\infty, t]\times  (t,\infty]^{{d^\prime}-1}\rc\rc\\
   &= \exp\lc- \mu_{d}\lc [-\infty, t]\times  (t,\infty]^{{d^\prime}-1} \times \minf^{d-d^\prime}\rc\rc.
\end{align*}
\end{proof}

\subsection{Proof of Lemma \ref{lemcadlagsums}}
\label{appprooflemcadlagsums}
\begin{proof}

Fix $m\in\N$. We know that $H\sim \sum_{i=1}^{m} H^{(i,1/m)}$ for some i.i.d.\ processes $H^{(i,1/m)}\in\bR^\RR$. 
W.l.o.g.\ assume that $H$ is defined on a probability space $(\Omega,\mathcal{F},\pp)$ and $\lc H^{(i,1/m)}\rc_{1\leq i\leq m}$ is a random element in $\lc \lc\bR^\RR\rc^{m},\tilde{\mathcal{F}},\tilde{\pp}\rc$, where $\tilde{\mathcal{F}}$ denotes the product $\sigma$-algebra generated by the finite dimensional projections. The idea of the proof is as follows: \\
Prove that the c\`adl\`ag property of $H$ transfers to $H^{(i,1/m)}$ if we restrict the processes to rational time indices and define i.i.d.\ c\`adl\`ag processes $\tilde{H}^{(i,1/m)}$ as rational time limits of $H^{(i,1/m)}$.

W.l.o.g.\ we can choose the metric $\bar{\tau}(x,y)=\vert \arctan{(x)}-\arctan{(y)} \vert$ to define distances (and thus continuity) on $\bR$. Denote the (measurable) set of $\QQ$-right-continuous paths of $H^{(i,1/m)}$ as 
$$A_{rc,i}:=\{w \mid w_i \text{ is right-continuous as a function on } \QQ\}.$$
We have
\begin{align*}
\tilde{\pp}\lc A_{rc,i}\rc&=\tilde{\pp}\lc \bigcap_{q\in\QQ}\bigcap_{\substack{\epsilon>0 \\ \epsilon\in\QQ}}\bigcup_{\substack{\delta>0 \\ \delta\in\QQ}}\bigcap_{\substack{q_1\in(q,q+\delta) \\ q_1\in\QQ}}\bigg\{\omega \ \bigg\vert\  \bar{\tau} \lc \omega_i(q_1),\omega_i(q)\rc <\epsilon\bigg\} \rc\overset{(\star)}{=}1.
\end{align*}
It remains to prove $(\star)$. Therefore, assume that $(\star)$ does not hold. In this case there exists $\bar{q}\in\QQ$ such that $H^{(i,1/m)}$ is not right-continuous at $\bar{q}$ with positive probability, i.e.\ there exists $\bar{q}\in\QQ$ and $\epsilon>0$ such that
$$\tilde{\pp}\lc \bigcap_{\substack{\delta>0 \\ \delta\in\QQ}}\bigcup_{\substack{q_1\in(\bar{q},\bar{q}+\delta) \\ q_1\in\QQ}} \bigg\{\omega \ \bigg\vert\  \bar{\tau} \lc \omega_i(q_1),\omega_i(\bar{q})\rc  >\epsilon\bigg\}\rc>0, $$
which is equivalent to 
$$\tilde{\pp}\lc \bigcup_{\substack{(q_n)_{n\in\N}\in\QQ^\N \\ q_n>\bar{q}\\ \lim_{n\to\infty} q_n=\bar{q} }} \bigcap_{N\in\N} \bigcup_{n\geq N} \bigg\{\omega \ \bigg\vert\  \bar{\tau} \lc \omega_i(q_n),\omega_i(\bar{q})\rc >\epsilon\bigg\}\rc>0. $$
Every fixed sequence $(q_n)_{n\in\N}$ such that for all $N\in\N$ there exists some $n\geq N$ with $\bar{\tau} ( \omega_i(q_n),$ $\omega_i(\bar{q}))>\epsilon$ satisfies one of the following $4$ cases:
\begin{enumerate}
    \item $\vert \omega_i(\bar{q})\vert <\infty$ and there exists some $\bar{\epsilon}>0$ such that $\omega_i(q_n)-\omega_i(\bar{q})>\bar{\epsilon}$ for infinitely many $n$ or
    \item $\vert \omega_i(\bar{q})\vert <\infty$ and there exists some $\bar{\epsilon}>0$ such that $\omega_i(q_n)-\omega_i(\bar{q})<-\bar{\epsilon}$ for infinitely many $n$ or
    \item $ \omega_i(\bar{q})=\infty$ and $\omega_i(q_n)<C$ for infinitely many $n$ and some constant $C\in\RR$ or
    \item $ \omega_i(\bar{q})=-\infty$ and $\omega_i(q_n)>C$ for infinitely many $n$ and some constant $C\in\RR$.
\end{enumerate}
Thus, there exists an $\bar{\epsilon}>0$ or $C>0$ such that either
\begin{enumerate}
    \item $$\tilde{\pp}\lc \bigcup_{\substack{(q_n)_{n\in\N}\in\QQ^\N \\ q_n>\bar{q}\\ \lim_{n\to\infty} q_n=\bar{q} }} \bigcap_{N\in\N} \bigcup_{n\geq N} \bigg\{\omega\ \bigg\vert\  \omega_i(q_n)-\omega_i(\bar{q})>\bar{\epsilon} , \ \vert \omega_i(\bar{q})\vert <\infty\bigg\}\rc>0 \text { or}$$
    \item $$\tilde{\pp}\lc \bigcup_{\substack{(q_n)_{n\in\N}\in\QQ^\N \\ q_n>\bar{q}\\ \lim_{n\to\infty} q_n=\bar{q} }} \bigcap_{N\in\N} \bigcup_{n\geq N} \bigg\{\omega \ \bigg\vert\   \omega_i(q_n)-\omega_i(\bar{q}) <-\bar{\epsilon}, \ \vert \omega_i(\bar{q})\vert <\infty  \bigg\}\rc>0 \text { or} $$
    \item $$\tilde{\pp}\lc \bigcup_{\substack{(q_n)_{n\in\N}\in\QQ^\N \\ q_n>\bar{q}\\ \lim_{n\to\infty} q_n=\bar{q} }} \bigcap_{N\in\N} \bigcup_{n\geq N} \bigg\{\omega \ \bigg\vert\   \omega_i(q_n) <C,\ \omega_i(\bar{q})=\infty \bigg\}\rc>0 \text { or} $$
    \item $$\tilde{\pp}\lc \bigcup_{\substack{(q_n)_{n\in\N}\in\QQ^\N \\ q_n>\bar{q}\\ \lim_{n\to\infty} q_n=\bar{q} }} \bigcap_{N\in\N} \bigcup_{n\geq N} \bigg\{\omega \ \bigg\vert\   \omega_i(q_n) >C , \ \omega_i(\bar{q})=-\infty  \bigg\}\rc>0 . $$
\end{enumerate}
W.l.o.g.\ we assume that the first assertion holds, since the other cases are treated similarly. The assertion implies that there exists a (fixed) sequence $(\bar{q}_n)_{n\in\N}$ with $\lim_{n\to\infty}\bar{q}_n=\bar{q}$ such that 
$$\tilde{\pp}\lc  \bigcap_{N\in\N} \bigcup_{n\geq N} \bigg\{\omega \ \bigg\vert\   \omega_i(\bar{q}_n)-\omega_i(\bar{q}) >\bar{\epsilon},\ \vert \omega_i(\bar{q})\vert <\infty\bigg\}\rc>0.$$
Since the $\omega_i$ are i.i.d.\ we obtain that
\begin{align*}
0&= \pp\lc \bigcap_{N\in\N} \bigcup_{n\geq N} \bigg\{H_{\bar{q}_n}-H_{\bar{q}} >m\bar{\epsilon},\ \vert H_{\bar{q}}\vert <\infty \bigg\}\rc\\
&\geq\tilde{\pp}\lc \bigcap_{N\in\N} \bigcup_{n\geq N} \bigg\{ \omega \ \bigg\vert \ \omega_1(\bar{q}_n)-\omega_1(\bar{q}) >\bar{\epsilon},\ \vert \omega_1(\bar{q})\vert <\infty\bigg\}\rc^{m} >0,
\end{align*}
which is a contradiction. Therefore, $(\star)$ is valid and $A_{rc}:=\cap_{i=1}^m A_{rc,i}$ satisfies $\tilde{\pp}\lc A_{rc}\rc=1 $. 

Next, define a finally one-sided Cauchy sequence as a Cauchy sequence $(q_n)_{n\in\N}$ for which there exists some $N\in\N$ such that $q_n>\lim_{n\to\infty} q_n$ or $q_n<\lim_{n\to\infty} q_n$ for all $n\geq N$. Denote the set of existing left and right limits of $w_i$ for finally one-sided $\QQ$-Cauchy sequences as
\begin{align*}
A_{fC,i} &:=\big\{ \omega\ \big\vert\ \big( \omega_i(q_n)\big)_{n\in\N} \text{ is a Cauchy sequence (w.r.t.\ }\bar{\tau}\text{)}\\
&\ \ \ \ \ \ \ \ \ \ \ \ \ \text{ for all finally one-sided Cauchy sequences } \lc q_n\rc_{n\in\N}\in \QQ^\N \big\}\\
        &=\bigcap_{\substack{(q_n)\in\QQ^\N \\ \text{finally one-} \\ \text {sided Cauchy}}}\bigcap_{\substack{\epsilon>0 \\ \epsilon\in\QQ}} \bigcup_{N\in\N} \bigcap_{m,n\geq N} \big\{ \omega\big\vert \ \bar{\tau}\lc \omega_i(q_n),\omega_i(q_m)\rc <\epsilon \big\} .
\end{align*}
Note that we explicitly allow for rational Cauchy sequences with irrational limit. Similar to the proof of $\pp\lc A_{rc,i}\rc=1$ we can show that $\pp\lc A_{fC,i}\rc=1$, which implies that $A_{fC}:=\cap_{i=1}^m A_{fC,i}$ satisfies $\tilde{\pp}\lc A_{fC}\rc=1$. 

Finally, define
\begin{align}
\label{eqndefcadlagversion}
    \tilde{H}_t^{(i,1/m)}(\omega):=\lim_{\substack{q\to t\\ q\in\QQ\\ q> t}}\omega_i(q)\id_{\{A_{rc}\}}(\omega)\id_{\{A_{fC}\}}(\omega),
\end{align}
which is measurable as the pointwise limit of measurable functions, if we can show that the limit exists and is independent of the chosen sequence. Therefore, choose two sequences $\lc q^{(1)}_n \rc_{n\in\N}, \lc q^{(2)}_n \rc_{n\in\N}\in \lc \QQ\cap(t,\infty)\rc^\N$ with limit $t\in\RR$. 
$ q^{(1)}_n,q^{(2)}_n>t$ for all $n\in\N$ implies that both sequences are finally one-sided Cauchy sequences. Therefore, both limits $\lim_{n\to\infty}\omega_i\lc q^{(1)}_n\rc \id_{\{A_{rc}\}}(\omega)\id_{\{A_{fC}\}}(\omega)$ and $\lim_{n\to\infty}\omega_i\lc q^{(2)}_n\rc \id_{\{A_{rc}\}}(\omega)\id_{\{A_{fC}\}}(\omega)$ exist. Moreover the combined sequence 
$$\lc \tilde{q}_n\rc_{n\in\N}:=\lc \begin{cases} q^{(1)}_n & n \text{ even} \\ q^{(2)}_n & n \text{ odd} \end{cases}\rc_{n\in\N}$$ is a finally one-sided Cauchy sequence. Therefore, the limit 
$$\lim_{n\to\infty}\omega_i(\tilde{q}_n)\id_{\{A_{rc}\}}(\omega)\id_{\{A_{fC}\}}(\omega)$$
exists as well and 
$$\lim_{n\to\infty}\omega_i\lc q^{(1)}_n\rc\id_{\{A_{rc}\}}(\omega)\id_{\{A_{fC}\}}(\omega)=\lim_{n\to\infty}\omega_i\lc q^{(2)}_n\rc\id_{\{A_{rc}\}}(\omega)\id_{\{A_{fC}\}}(\omega).$$
Thus, the limit in Equation (\ref{eqndefcadlagversion}) exists and is independent of the chosen sequences, which implies that $\lc \tilde{H}^{(i,1/m)}\rc_{1\leq i\leq m}$ define valid stochastic processes. Observe that 
\begin{align}
    \tilde{\pp}\lc \tilde{H}^{(i,1/m)}_q(\omega) =w_i(q) =H^{(i,1/m)}_q(\omega)\ \text{ for all } q\in\QQ\rc=1. \label{eqncadlagsumsonQ}
\end{align}
Thus, $\tilde{H}^{(i,1/m)}$ and $H^{(i,1/m)}$ almost surely coincide on $\QQ$, which follows from the fact that $\omega_i$ is only considered to be non-zero on $\QQ$-right-continuous paths. We claim that $\tilde{H}^{(i,1/m)}\in D^\infty(\RR)$.

Firstly, we prove that $\tilde{H}^{(i,1/m)}(\omega)$ is right-continuous for all $\omega\in\Omega$ and $1\leq i\leq m$. To this purpose choose some strictly decreasing sequence $(t_n)_{n\in\N}\in\RR^\N$ with $\lim_{n\to\infty} t_n=t$ and let $\epsilon>0$ be arbitrary. Choose $t_n<q_n=q_n(\epsilon,\omega)\in\QQ$ such that $\bar{\tau}\lc  \tilde{H}^{(i,1/m)}_{t_n},\tilde{H}^{(i,1/m)}_{q_n}\rc <\epsilon$ and $\vert t_n-q_n\vert <1/n$, which is possible since $\tilde{H}^{(i,1/m)}_{t_n}$ is defined as the limit of a rational time evaluations of $H^{(i,1/m)}$. Since $(q_n)_{n\in\N}$ is a finally one-sided Cauchy sequence with limit $t$ there exists some $N(\omega,\epsilon)\in\N$ such that for all $n\geq N$ we have that $\bar{\tau}\lc \tilde{H}^{(i,1/m)}_{q_n}, \tilde{H}^{(i,1/m)}_t\rc <\epsilon$. Therefore, for $n\geq N$, we obtain
$$\bar{\tau}\lc \tilde{H}^{(i,1/m)}_t, \tilde{H}^{(i,1/m)}_{t_n}\rc \leq \bar{\tau}\lc \tilde{H}^{(i,1/m)}_t, \tilde{H}^{(i,1/m)}_{q_n}\rc+ \bar{\tau}\lc \tilde{H}^{(i,1/m)}_{q_n}, \tilde{H}^{(i,1/m)}_{t_n}\rc <2\epsilon, $$
which yields that $\lim_{n\to\infty}\tilde{H}^{(i,1/m)}_{t_n}=\tilde{H}^{(i,1/m)}_t$. Thus, $\tilde{H}^{(i,1/m)}$ is right-continuous.

Secondly, we show that $\tilde{H}^{(i,1/m)}(\omega)$ has left limits for all $\omega\in\Omega$ and $1\leq i\leq m$.
To see this choose some arbitrary $t\in\RR$, $\epsilon>0$ and a sequence $(t_n)_{n\in\N}\in(-\infty,t)^\N$ with limit $t$. Define $q_n (\epsilon,\omega)$ as some rational number in $[t_n,t)$ such that $\bar{\tau}\lc \tilde{H}^{(i,1/m)}_{t_n},\tilde{H}^{(i,1/m)}_{q_n}\rc<\epsilon$, which is possible since $\tilde{H}^{(i,1/m)}$ is right-continuous. If $n,m$ are large enough such that $t_n$ and $t_m$ are close to $t$ we have that $q_n$ and $q_m$ are also close to $t$. Therefore, $(q_n)_{n\in\N}$ is a finally one-sided Cauchy sequence. Thus, we can find an $N(\omega,\epsilon)\in\N$ such that for all $m,n\geq N$ we have $\bar{\tau}\lc \tilde{H}^{(i,1/m)}_{q_n},\tilde{H}^{(i,1/m)}_{q_m}\rc<\epsilon$, which implies
\begin{align*}
    \bar{\tau}\lc \tilde{H}^{(i,1/m)}_{t_n},\tilde{H}^{(i,1/m)}_{t_m}\rc &\leq \bar{\tau}\lc \tilde{H}^{(i,1/m)}_{t_n},\tilde{H}^{(i,1/m)}_{q_n}\rc + \bar{\tau}\lc \tilde{H}^{(i,1/m)}_{t_m},\tilde{H}^{(i,1/m)}_{q_m}\rc\\
    &+\bar{\tau}\lc\tilde{H}^{(i,1/m)}_{q_n},\tilde{H}^{(i,1/m)}_{q_m}\rc\\
    &<3\epsilon.
\end{align*} 
Since $\epsilon>0$ was arbitrary we have shown that $\tilde{H}^{(i,1/m)}$ has left limits for every $\omega$.

Obviously, $\lc \tilde{H}^{(1,1/m)}\rc_{1\leq i\leq m}$ are i.i.d.\ as almost sure pointwise limits of i.i.d.\ objects. It remains to prove that $H^{(1,1/m)}\sim \tilde{H}^{(1,1/m)}$. The characteristic functional of $H^{(1,1/m)}$, denoted as $CF_{H^{(1,1/m)}}(\bm z,\bm t)$ and the characteristic functional of $\tilde{H}^{(1,1/m)}$, denoted as $CF_{\tilde{H}^{(1,1/m)}}(\bm z,\bm t)$, coincide for $\bm z\in\RR^d $ and $\bm t\in\QQ^d$ by Equation (\ref{eqncadlagsumsonQ}). For arbitrary $\bm z\in\RR^d$ and $\bm t\in\RR^d$ let $CF_H(\bm z,\bm t)$ denote the characteristic functional of $H$. We use the fact that $H$ is right-continuous to obtain
\begin{align*}
&CF_{H^{(1,1/m)}}(\bm z,\bm t)=CF_H(\bm z,\bm t)^{\frac{1}{m}}=\lim_{\substack{\bm s\searrow t\\ \bm s\in \QQ^d\\ \bm s>\bm t}} CF_H(\bm z,\bm s)^{\frac{1}{m}}=\lim_{\substack{\bm s\searrow t\\ \bm s\in \QQ^d\\ \bm s>\bm t}} CF_{H^{(1,1/m)}}(\bm z,\bm s) \\
&=\lim_{\substack{\bm s\searrow t\\ \bm s\in \QQ^d\\ \bm s>\bm t}}\ee_{\tilde{\pp}}\lk \exp\lc \sum_{j=1}^d  \mathrm{i}z_j H^{(1,1/m)}_{s_j}\rc\rk 
=\lim_{\substack{\bm s\searrow t\\ \bm s\in \QQ^d\\ \bm s>\bm t}}\ee_{\tilde{\pp}}\lk \exp\lc \sum_{j=1}^d  \mathrm{i}z_j\tilde{H}^{(1,1/m)}_{s_j}\rc\rk\\ &=CF_{\tilde{H}^{(1,1/m)}}(\bm z,\bm t),
\end{align*}
where the second to last equality uses that $\tilde{H}^{(1,1/m)}$ and $H^{(1,1/m)}$ almost surely coincide on $\QQ$ and the last equality uses the fact that $\tilde{H}^{(1,1/m)}$ is right-continuous. This proves that $H^{(1,1/m)}\sim \tilde{H}^{(1,1/m)}$.
\end{proof}

\subsection{Proof of Corollary \ref{corcadlagsums}}
\label{appproofcorcadlagsums}
\begin{proof}
We use the same notation as in the proof of Lemma \ref{lemcadlagsums}. Denote the set of $\QQ$-non-negative paths as 
$$A_{\geq 0}:=\big\{ \omega\mid \omega_i(q)\geq 0 \text{ for all } q\in\QQ \text{ and } 1\leq i\leq m\big\}.$$
Furthermore, denote the set of $\QQ$-non-decreasing paths of as 
$$A_{\nearrow}:=\{ \omega \mid \omega_i(\cdot) \text{ is non-decreasing on } \QQ\text{ for all }1\leq i\leq m \}.$$
By similar arguments as in the proof of Lemma \ref{lemcadlagsums} we obtain that
\begin{align*}
&\tilde{\pp}(A_{\nearrow})=\tilde{\pp}\lc \bigcap_{1\leq i\leq m} \bigcap_{\substack{ q_1,q_2\in\QQ \\ q_1\leq q_2}} \big\{ \omega \mid \omega_i(q_1)\leq \omega_i(q_2) \big\}\rc=1 \text{ and }\\
& \tilde{\pp}\lc A_{\geq 0}\rc=\tilde{\pp}\lc \bigcap_{1\leq i\leq m} \bigcap_{\substack{ q\in\QQ }} \big\{ \omega \mid \omega_i(q)\geq 0 \big\}\rc=1,
\end{align*}
since intersections of countably many sets with probability $1$ have probability $1$.
Next, define
\begin{align}
    \hat{H}^{(i,1/m)}_t(\omega):=\tilde{H}^{(i,1/m)}_t(\omega) \id_{\{A_{\nearrow}\}}(\omega)\id_{\{A_{\geq 0}\}}(\omega).
\end{align}
Obviously, $\hat{H}^{(i,1/m)}$ is non-negative and non-decreasing. Similar to the proof of Lemma \ref{lemcadlagsums} we can show that the Laplace transforms of $\hat{H}^{(i,1/m)}$ and $H^{(i,1/m)}$ coincide and the claim follows.
\end{proof}

\subsection{Proof of Proposition \ref{propositionltchronometer}}
\begin{proof}
Let $\tilde{b}$ and $\tilde{\nu}$ denote the drift and L\'evy measure of $H$ given by \cite[Theorem 2.8]{rosinskiinfdivproc}. Note that $\tilde{\nu}$ is a measure on $\bR^\RR$ equipped with the $\sigma$-algebra generated by the finite dimensional projections, since \cite{rosinskiinfdivproc} views id-processes as processes in $\RR^\RR$.
There are two things that need to be shown:
\begin{enumerate}
\item $\tilde{\nu}$ can be restricted to a measure on $M^0_\infty$.
\item The integral over $\tilde{\nu}$ in Equation (\ref{laplaceexponentidprocess}) can be defined without the compensating term $\sum_{i=1}^d z_i x(t_i)$ $\id_{\{\vert x(t_i)\vert <\epsilon\}}$ and $b\in M^0_\infty \cap D(\RR)$.
\end{enumerate}

We start with the first statement. Similar to the proof of \cite[Theorem 3.4]{rosinskiinfdivproc} we can show that there exists an exact representation $\nu$ of $\tilde{\nu}$ defined 
%\textcolor{red}{on $D^\infty(\RR)$, since $D^\infty(\RR)$ is a standard Borel space and an algebraic group under addition. }
on $$C_0:=\bigg\{x\in D^\infty(\RR)\ \big\vert\ \lim_{t\to-\infty} x(t)=0\bigg\},$$ since $C_0$ is an algebraic group under addition and a standard Borel space as a measurable subset of a standard Borel space. For additional information on exact representations of L\'evy measures, see \cite[Definition 2.20]{rosinskiinfdivproc}. For our purposes it suffices to view an exact representation of a L\'evy measure as a restriction of a L\'evy measure to a smaller domain. To prove the first statement it suffices to show that $\nu$ vanishes on the set
\begin{align*}
C_1:=&\{ x\in D^\infty(\RR) \mid x \text{ is non-decreasing}\}^\complement, %\text{ and } 
%C_2:=& \{ x\in D^\infty(\RR) \mid x \text{ is non-negative}\}^\complement ,
\end{align*}
%$$\textcolor{red}{C_3:= \{ x\in D^\infty(\RR) \mid\lim_{t\to-\infty} x(t)=0 \}^\complement ,}$$
since $C_1$ % and $C_2$ 
%\textcolor{red}{and $C_3$} 
is measurable (in $D^\infty(\RR)$). In particular, $M^0_\infty=C_0 \cap C_1^\complement$. \\%\cap C_2^\complement $ \textcolor{red}{\cap C_3^\complement}$ is measurable. 
For $I=(t_1,\ldots,t_d)\in\RR^d$ (w.l.o.g.\ $ t_1\leq \ldots \leq t_d$) and $A\in\bc\lc\bR^d\rc$ define
$$\nu_I(A):=\nu\lc\{ x\in D^\infty(\RR)\mid (x(t_1),\ldots,x(t_d))\in A\}\rc.$$
$\tilde{\nu}_I(A)$ is defined analogously. Observe that $\nu_I(A)=\tilde{\nu}_I(A)$ for all $A\in\bc\lc \bR^d\rc $ by construction. We show that
$$\nu(C_1)=  \nu\lc  \{ x\in D^\infty(\RR) \mid x \text{ is non-decreasing}\}^\complement  \rc=0.$$
\cite[Proposition 6.1]{BarndorffNielsen2006InfiniteDF} tells us that the L\'evy measure $\tilde{\nu}_{(t_1,\ldots,t_d)}$ of $(H_{t_1},\ldots ,H_{t_d})$ is concentrated on the cone $K_d:=\{ \bmx\in \bR^d \mid 0\leq x_1 \leq \ldots\leq x_d\}$, which implies that $\nu_{(t_1,\ldots,t_d)}$ is also concentrated on $K_d$. Now, assume that $\nu(C_1)=\nu\big( \{ x\in D^\infty(\RR) \mid x \text{ is non-decreasing}\}^\complement \big)>0$. Observe that
\begin{align*}
\{x\in D^\infty(\RR) \mid x \text{ is non-decreasing}\}^\complement =\bigcup_{\substack{t_1,t_2\in\QQ \\t_1<t_2}} \{x\in D^\infty(\RR)\mid x(t_1)>x(t_2)\}.
\end{align*} 
Therefore, there exist $\bar{t}_1<\bar{t}_2$ with $\nu\big( \{x\in D^\infty(\RR)\mid x(\bar{t}_1)>x(\bar{t}_2)\}\big)>0$. By the construction of $\nu$ we get 
\begin{align*}
&\nu\big( \{x\in D^\infty(\RR)\mid x(\bar{t}_1)>x(\bar{t}_2)\}\big)=\nu_{(\bar{t}_1,\bar{t}_2)}\big( \cup_{s\in\QQ} (s,\infty]\times [-\infty,s] \big)\\
&=\tilde{\nu}_{(\bar{t}_1,\bar{t}_2)}\big( \cup_{s\in\QQ} (s,\infty]\times [-\infty,s] \big)=0,
\end{align*} 
which is a contradiction. Therefore, $\nu(C_1)=\nu\big( \{ x\in D^\infty(\RR) \mid x \text{ is non-decreasing}\}^\complement \big)=0$ and $\nu$ is concentrated on non-decreasing functions which satisfy $\lim_{t\to-\infty}x(t)=0$. %Thus, the claim $\nu\big( \{ x\in D^\infty(\RR) \mid x \text{ is non-negative}\}^\complement \big)=0$ immediately follows, since non-decreasing functions which satisfy $\lim_{t\to-\infty}x(t)=0$ are necessarily non-negative.
Thus, $\nu$ is concentrated on $M^0_\infty= C_0 \cap C_1^\complement$%\cap C_2^\complement$
, i.e.\ 
$$\nu(A)=\nu (A\cap M^0_\infty ),\ \ \forall A\in \bc(D^\infty(\RR)).$$

Let us turn to the proof of the second statement.
For every $d\in\N$ and $\bm t\in\RR^d$ we obtain a non-negative drift vector $(b(t_1),\ldots, b(t_d))$ from the $d$-dimensional L\'evy--Khintchine triplet of the non-negative random vector $(H_{t_1},\ldots,H_{t_d})$ with truncation function $0$.  Condition $(\Diamond^\prime 2)$ and \cite[Proposition 6.1]{BarndorffNielsen2006InfiniteDF} imply the existence of a unique non-decreasing finite drift $b:\RR \to [0,\infty)$. Right-continuity and $\lim_{t\to-\infty} b(t)=0$ follow from the right-continuity of $H$ and $\lim_{t\to-\infty} H_t=0$. Thus, $b\in M^0_\infty\cap D(\RR)$.

Combining the above yields  
$$\ee\lk \exp\lc -\sumd z_i H_{t_i}\rc\rk=\exp\lc -\sumd z_i b(t_i)+ \int_{M^0_\infty} \lc \exp\lc-\sumd z_i x(t_i) \rc -1\rc \nu(\mathrm{d}x)  \rc$$
for every $\bm z\in[0,\infty)^d,\bm t\in\RR^d$.
\end{proof}

\begin{rem}[Implications for strong-idt processes]
An id-process $H$ is called strong-idt, if 
\begin{align*}
    \big( H_t\big)_{t\geq 0}\sim \lc \sumn H^{(i)}_{ \frac{t}{n}}\rc_{t\geq 0}
\end{align*}
 for all $n\in\N$,  where $\lc H^{(i)}\rc_{i\in\N}$ denote i.i.d.\ copies of $H$.
Such processes are studied, among others, in \cite{koppmolchanov,maicanonicalspecrepofstabtail}. \cite{koppmolchanov} study the L\'evy measure and series representations of real-valued strong-idt processes without focus on non-decreasing paths. \cite{maicanonicalspecrepofstabtail} refines these results in the special case of non-decreasing $H$, which might possibly also take the value $\infty$. However, \cite{maicanonicalspecrepofstabtail} does not formally prove the extension to extended real-valued processes, despite he uses the results of \cite{koppmolchanov}. Proposition~\ref{propositionltchronometer} fills this gap by formally justifying that the claimed extension is correct. Furthermore, whereas \cite{koppmolchanov} work on the space of c\`adl\`ag functions equipped with the Skorohod ($J1$) metric, \cite{maicanonicalspecrepofstabtail} works with the L\'evy metric defined for distribution functions. While it is known that the two metrics are not equivalent in general, one can actually prove that their induced Borel $\sigma$-algebras on the space of non-decreasing paths coincide. Thus, implicitly both references indeed work with the same objects. In particular, \cite[Lemma 1]{maicanonicalspecrepofstabtail} implicitly shows with a tedious and probabilistic proof that the L\' evy measure of a non-decreasing strong-idt process is concentrated on non-decreasing paths. In this regard, Proposition~\ref{propositionltchronometer} provides a more direct proof of this fact. Moreover, it is even more general, since it holds for arbitrary non-decreasing and non-negative id-processes and not just strong-idt processes.
\end{rem}

\subsection{Proof of Theorem \ref{theoremchroncorrespondstominiddist}}
\begin{proof}
\textcolor{white}{a}
\begin{itemize}
\item[``$\Rightarrow$''] Let $n\in\N$. Since $\bmX$ is min-id, there exist i.i.d.\ sequences $(\bmX^{(i,1/n)})_{\leqn}\in\minf^\N$ such that $\bmX\sim\min_{\leqn} \bmX^{(i,1/n)}$. First, we claim that the exchangeability of $\bmX$ implies the exchangeability of $\bmX^{(1,1/n)}$. Seeking a contradiction, we assume that $\bmX^{(1,1/n)}$ is not exchangeable. In this case, there exists $\{i_1,\ldots,i_d\}\subset \N$ such that $(X^{(1,1/n)}_{1,i_1}\ldots ,X^{(1,1/n)}_{1,i_d})$ is not exchangeable. By similar arguments as in the proof of Proposition~\ref{propositionexchangeabilityexponentmeasure} there exist $\bmx\in\RR^d$ and a permutation $\pi$ on $\{1,\ldots,d\}$ such that $\pp(\bmX_d^{(1,1/n)}>\bmx)\not =\pp(\bmX_d^{(1,1/n)}>\pi(\bmx))$. This yields \begin{align*}
    \pp(X_{i_1}>x_1,\ldots, X_{i_d}>x_d)&=\pp\lc X^{(1,1/n)}_{i_1}>x_1,\ldots, X^{(1,1/n)}_{i_d}>x_d \rc^n\\
    &\not =\pp\lc X^{(1,1/n)}_{i_1}>\pi(\bmx)_1,\ldots, X^{(1,1/n)}_{i_d}>\pi(\bmx)_d\rc^n\\
    &=\pp\lc X_{i_1}>\pi(\bmx)_1,\ldots, X_{i_d}>\pi(\bmx)_d\rc,
\end{align*}
which is a contradiction. Therefore, $\bmX^{(1,1/n)}$ is exchangeable. 

Now, de Finetti's Theorem yields the existence of i.i.d.\ nnnd c\`adl\`ag processes $\lc H^{(i,1/n)}\rc_{\leqn}$ $\in D^\infty(\RR)$ such that 
$$\lc X^{(i,1/n)}_{j}\rc_{j\in\N} \sim \inf\big\{ t\in\RR \mid H^{(i,1/n)}_t\geq E^{(i)}_{j}\big\},$$
where $((E^{(i)}_{j})_{j\in\N})_{\leqn}$ are i.i.d.\ unit exponential independent of $\lc H^{(i,1/n)}\rc _{\leqn}$. 
Obviously, $\lim_{t\to-\infty}H^{(i,1/n)}_t=0$ almost surely, since $\pp\lc X^{(1/n)}_{i,j}=-\infty\rc>0$ otherwise. Moreover, $\pp\lc H^{(i,1/n)}_t=\infty\rc<1$ for all $t\in\RR$, since $\bell=\binfty$. Therefore, $H^{(i,1/n)}$ satisfies Condition $(\Diamond^\prime)$. It remains to show that $H^{(1,1)}=:H$ is infinitely divisible and unique. Let $\fb_k$ denote the survival function of $(X_1,\ldots, X_{k})$ and recall that $\lc X^{(1,1/n)}_{1},\ldots,X^{(1,1/n)}_{k} \rc\sim \fb_k^{1/n}$. Choose $\bm z\in\N^d$ and $\bm t\in\RR^d$, then
\begin{align*}
L(\bm z,\bm t)&=\ee\lk \exp\lc-\sum_{j=1}^{d} z_j H_{t_j} \rc\rk=\ee\lk \exp\lc -\sum_{j=1}^{d} \sum_{k=1}^{z_j} H_{t_j} \rc\rk\\
&=\fb_{\sum_{j=1}^d z_j}(\underbrace{t_1,\ldots,t_1}_{z_1 \text{ times}},\ldots,\underbrace{t_d,\ldots,t_d}_{z_d \text{ times}})\\
&=\lc\fb_{\sum_{j=1}^d z_j}^{1/n}\rc^n(\underbrace{t_1,\ldots,t_1}_{z_1 \text{ times}},\ldots,\underbrace{t_d,\ldots,t_d}_{z_d \text{ times}})\\
&=\ee\lk  \exp\lc-\sum_{j=1}^{d} \sum_{k=1}^{z_j} H^{(1,1/n)}_{t_j}\rc \rk^n =\ee\lk \prod_{i=1}^n \exp\lc-\sum_{j=1}^{d} z_j H^{(i,1/n)}_{t_j} \rc\rk\\
&=\ee\lk \exp\lc- \sum_{j=1}^{d} z_j\sum_{i=1}^n H^{(i,1/n)}_{t_j} \rc\rk.
\end{align*} 
Using the fact that the Laplace transform of a non-negative random vector is uniquely determined by its values on $\N^d$, see \cite{kleiber2013multivariate} for more details, this shows that $H_\bmt\sim\sum^n_{i=1} H^{(i,1/n)}_\bmt$. Since $\bmt$ and $n$ were arbitrary, we get that $H$ is infinitely divisible. The uniqueness of $H$ follows from 
$$ \ee\lk \exp\lc-\sum_{j=1}^d z_j H_{t_j} \rc\rk=\fb_{\sum_{j=1}^d z_j}(\underbrace{t_1,\ldots,t_1}_{z_1 \text{ times}},\ldots,\underbrace{t_d,\ldots,t_d}_{z_d \text{ times}}).$$
\item[``$\Leftarrow $'']
We refer to the survival function of $(X_1,\ldots,X_d)$ by $\fb_d$. Since the exchangeability of $\bmX$ is obvious by the construction, it suffices to show that $\fb_d^{1/n}$ is a survival function of a random variable on $\minf^d$ for every $d,n\in\N$ and that $\bell=\binfty$.

By Corollary~\ref{corcadlagsums}, there exist i.i.d.\ extended chronometers $(H^{(i,1/n)})_{1\leq i\leq n}$ such that $H\sim \sum_{i=1}^n H^{(i,1/n)}$. It easily follows that $\lim_{t\to-\infty}H^{(1,1/n)}_t=0$, which implies that $H^{(1,1/n)}\in M_\infty^0$. For $\bm t\in\RR^d$, we get 
\begin{align*}
\fb_d(\bm t)&=\ee\lk \exp\lc -\sum_{j=1}^d H_{t_j}\rc\rk=\ee\lk \exp\lc -\sum_{j=1}^d\sum_{i=1}^n H^{(i,1/n)}_{t_j}\rc\rk\\
&=\ee\lk \exp\lc -\sum_{j=1}^d H^{(1,1/n)}_{t_j}\rc\rk^n.
\end{align*}
Since $H^{(1,1/n)}_{t_j}$ is c\`adl\`ag we obtain that $$\fb_d^{1/n}(\bm t)=\ee\lk \exp\lc -\sum_{j=1}^d H_{t_j}^{(1,1/n)}\rc\rk$$ is the survival function of the first $d$ components of the exchangeable sequence
$$\bmX^{(1,1/n)}:=\lc \inf \big\{t\in\RR\mid  H^{(1,1/n)}_t\geq E_i\big\}\rc_{i\in\N}\in (-\infty	,\infty]^\N.$$ 
Therefore, $\bmX$ is min-id. Moreover, $\bell=\binfty$, since $H$ satisfies Condition $(\Diamond^\prime 2)$. Thus, $\bmX$ satisfies Condition $(\Diamond)$. Since $d$ and $n$ were arbitrary, the claim follows.
\end{itemize}
\end{proof}

\subsection{Proof of Theorem \ref{theoremexponentmeasureisiidmixture}}
\begin{proof}

Theorem~\ref{theoremchroncorrespondstominiddist} and Proposition~\ref{propositionextendibleexponentmeasure} provide a one-to-one correspondence between the min-id sequence $\bmX$, a (unique) L\'evy measure $\nu$ on $M^0_\infty$ with drift $b$ and an exponent measure $\mu$ on $E^\N_\binfty$. Choosing $\bm t\in\RR^d$, $d\in\N\cup\{\infty\}$, we can rewrite this correspondence as
\begin{align*}
\pp&\lc X_1>t_1,\ldots,X_d >t_d\rc=\ee\lk \exp\lc -\sumd H_{t_i}\rc\rk \\
&=\exp\lc - \sumd b(t_i) -\bigintssss_{M^0_\infty} 1-\exp\lc -\sumd x(t_i)\rc \nu(\mathrm{d}x)   \rc\\
&=\exp\lc -\sumd b(t_i) -\bigintssss_{M^0_\infty} 1-\prodd \exp\lc - x(t_i)\rc \nu(\mathrm{d}x)   \rc\\
&=\exp\lc -\sumd b(t_i) -\bigintssss_{\overline{M}^0_\infty} 1-\prodd \gb(t_i) \gamma(\mathrm{d}G)    \rc\\
&=\exp\lc -\sumd b(t_i) -\bigintssss_{\overline{M}^0_\infty} \otimes_{i=1}^d\pp_G\lc (-\infty,\infty]^d\setminus (\bm t,\binfty]\rc \gamma(\mathrm{d}G) \rc,\\
\end{align*}
where $\gb=1-G$ is a survival function of a random variable on $(-\infty,\infty]$ and $\gamma$ is the image measure of the L\'evy measure $\nu$ under the transformation $h: M^0_\infty\to \overline{M}^0_\infty,\ x\mapsto 1- \exp(-x(\cdot))$. This implies that for every $d\in\N\cup\{\infty\}$
\begin{align*}
\mu_d\lc E^d_\binfty\setminus (\bm t,\binfty] \rc&= \mu_{b,d} \lc (-\infty,\infty]^d\setminus (\bm t,\binfty]\rc  \\
&+\bigintssss_{\overline{M}^0_\infty} \otimes_{i=1}^d\pp_G\lc (-\infty,\infty]^d\setminus (\bm t,\binfty]\rc \gamma(\mathrm{d}G)  . 
\end{align*}
A similar argument as in the proof of Proposition~\ref{propositionexchangeabilityexponentmeasure} yields 
\begin{align*}
\mu_d(A)=\mu_{b,d}\lc A\rc  +\bigintssss_{\overline{M}^0_\infty} \otimes_{i=1}^d\pp_G\lc A\rc \gamma(\mathrm{d}G)  \text{ for all }A\in\bc(E^d_\binfty).
\end{align*}
It remains to verify the properties of $\gamma$. Obviously, $\gamma\lc 0_{D^\infty(\RR)}\rc=0$. Applying the inequalities $1-x\leq \min\{1,-\log(x)\}$ and $\min\{1,-\log(x)\}\leq e(1-x)$ for any $x\in[0,1]$ yields
\begin{align*}
\int_{\overline{M}^0_\infty} G(t)\gamma(\mathrm{d}G) =\int_{\overline{M}^0_\infty} 1-\gb(t)\gamma(\mathrm{d}G) \leq \int_{M^0_\infty} \min\{1,x(t)\} \nu(\mathrm{d}x)<\infty
\end{align*}
as well as
\begin{align*}
\int_{M^0_\infty} \min\{1,x(t)\} \nu(\mathrm{d}x)\leq e\int_{\overline{M}^0_\infty} 1-\gb(t)\gamma(\mathrm{d}G) = e\int_{\overline{M}^0_\infty} G(t)\gamma(\mathrm{d}G)  <\infty.
\end{align*} 
Therefore, the integrability condition of $\gamma$ is equivalent to the integrability condition of $\nu$.
\end{proof}

\subsection{Proof of Proposition \ref{propstrongidt}}
\begin{proof}
We observe that
$$H_t= \sum_{k \geq 1} -\log\Big\{ 1-G_k\Big( \frac{t}{S_k}\Big)\Big\} =\int_{[0,\infty) \times \overline{M}^0_\infty} -\log\lc 1-G\lc\frac{t}{S} \rc\rc N(\mathrm{d}(S,G))$$
is given by integration of the (measurable) function $(S,G)\mapsto -\log\lc 1-G(t/S)\rc$ w.r.t.\ the Poisson random measure $N$. Similarly, for $\bm z\in[0,\infty)^d$, $\sumd z_i H_{t_i}$ is given by the integration of $(S,G)\mapsto \sumd -z_i\log\lc 1-G(t_i/S)\rc$ w.r.t.\ the Poisson random measure $N$. Therefore, the Laplace transform of $H$ is given by the (one-dimensional) Laplace transform of an integral over a Poisson random measure. An application of \cite[Proposition 3.6]{resnickextreme} yields
\begin{align*}
    L(\bm z,\bm t)&=\ee\lk \exp\lc -\sumd z_i H_{t_i}\rc\rk \\
    &=\ee\lk \exp\lc-\bigintssss_{[0,\infty]\times \overline{M}^0_{\infty}}\sumd -z_i\log\lc 1-G\lc\frac{t_i}{S}\rc\rc N(\mathrm{d}(S,G))\rc \rk\\
    &=\exp\lc- \bigintssss_{[0,\infty]}\bigintssss_{ \overline{M}^0_{\infty}} 1-\prodd  \lc 1-G\lc\frac{t_i}{s}\rc\rc^{z_i} \rho(\mathrm{d}G)\kappa(\mathrm{d}s)\rc.
\end{align*}
Thus, $H$ is infinitely divisible and the exchangeable min-id sequence $\bmX$ associated with $H$ has exponent measure $\mu_{\kappa,\rho}$. We can express the survival function of $\bmX$ as
\begin{align*}
&\pp(\bmX>\bm t) =\exp \lc -\int_{\overline{M}^{0}_{\infty}}\int_0^{\infty}1-\prod_{i\in\N} \lc 1-G\lc \frac{t_i}{s}\rc\rc \kappa(\mathrm{d}s)\rho(\mathrm{d}G) \rc \\
&=\exp \lc -\int_{\overline{M}^{0}_{\infty}}\int_0^{\infty} \int_{[0,\infty]^\N} \id \bigg\{ y_i\leq \frac{t_i}{s} \text{ for some } i\in\N\bigg\} \lc\otimes_{i\in\N} \pp_G\rc(\mathrm{d}\bm y)\kappa(\mathrm{d}s)\rho(\mathrm{d}G) \rc   \\ 
&=\exp \lc -\int_{\overline{M}^{0}_{\infty}} \int_{[0,\infty]^\N}  \int_0^{\infty}\id \bigg\{ s\leq \max_{i\in\N} \frac{t_i}{y_i}  \bigg\} \kappa(\mathrm{d}s)\lc\otimes_{i\in\N} \pp_G\rc(\mathrm{d}\bm y)\rho(\mathrm{d}G) \rc  ,
\end{align*}
which finishes the argument. 
\end{proof}

\subsection{Proof of Proposition \ref{propositionintegratedidprocess}}
\begin{proof}
It is easy to see that $H^{(\kappa)}$ is infinitely divisible, since $x(\cdot)\mapsto \int_{0}^\cdot$ is a measurable map in $D^\infty\big([0,\infty)\big)$. By the generalized L\'evy--Ito representation \cite[Proposition 3.1 and Theorem 5.1]{rosinskiinfdivproc} there exists a version $V^\prime$ of $V$ such that 
$$ \lc V^\prime_s\rc _{s \geq 0}=\lc b_V(s)+ \int_{D^\infty(\RR)} x(s) N(\mathrm{d}x) \rc_{s \geq 0} ,$$
where $N$ denotes a Poisson random measure on $D^\infty\big([0,\infty)\big)$ with intensity $\nu_V$. Note that the compensating term in the generalized L\'evy--Ito representation can be omitted by \cite[Theorem 5.1]{rosinskiinfdivproc}. Moreover, $N$ can be chosen as a random measure on $D^\infty\big([0,\infty)\big)_+:=\{ x\in D^\infty\big([0,\infty)\big) \mid x(t)\geq 0 \text{ for all } t\geq 0\}$, which follows by similar arguments as in the proof of Proposition \ref{propositionltchronometer}. Thus,
\begin{align*}
\lc H^{(\kappa)}_t\rc_{t\geq 0}&=\lc \int_{0}^t V_s\kappa(\mathrm{d}s)\rc_{t\geq 0}\\
&\sim \lc \int_{0}^t \lc b_V(s)+ \int_{D^\infty\big([0,\infty)\big)_+} x(s) N(\mathrm{d}x)  \rc \kappa(\mathrm{d}s)\rc_{t\geq 0} .
\end{align*} 
Since $N$ is $\sigma$-finite and concentrated on non-negative functions we can use Fubini's Theorem to obtain
\begin{align*}
&\lc \int_{0}^t \lc b_V(s)+ \int_{D^\infty\big([0,\infty)\big)_+} x(s) N(\mathrm{d}x)  \rc \kappa(\mathrm{d}s)\rc_{t\geq 0}\\
=&\lc \int_{0}^t  b_V(s) \kappa(\mathrm{d}s) + \int_{D^\infty\big([0,\infty)\big)_+} \int_{0}^t x(s) \kappa(\mathrm{d}s) N(\mathrm{d}x)  \rc_{t\geq 0},
\end{align*}
which is a decomposition of $H^{(\kappa)}$ into a non-decreasing deterministic drift $b^{(\kappa)}(t):=\int_{0}^t  b(s)\kappa(\mathrm{d}s)$ and an integral over a Poisson random measure $\int_{D^\infty\big([0,\infty)\big)_+} x^{(\kappa)}(t) N(\mathrm{d}x)$, where $x^{(\kappa)}(t):=\int_{0}^t x(s) \kappa(\mathrm{d}s)$ is a non-decreasing function in $D^\infty\big([0,\infty)\big)_+$. Therefore, using the usual formula for the Laplace transform of an integral over a Poisson random measure \cite[Proposition 3.6]{resnickextreme}, we obtain for arbitrary $\bm z\in [0,\infty]^d$ and $\bm t\in[0,\infty)^d$
\begin{align*}
    \ee\lk \exp\lc-\sumd z_i H^{(\kappa)}_{t_i}\rc\rk &= \ee\lk \exp\lc-\sum z_i \lc b^{(\kappa)}_{t_i} +\int_{D^\infty\big([0,\infty)\big)_+} x^{(\kappa)}(t_i) N(\mathrm{d}x)\rc  \rc\rk\\
    &=\exp\bigg( -\sum z_i b^{(\kappa)}_{t_i}-\\
    &\int_{D^\infty\big([0,\infty)\big)_+} 1-\exp\lc - \sumd z_i x^{(\kappa)} (t_i)\rc \nu_V(\mathrm{d}x)\bigg).
\end{align*}
Note that 
\begin{align*}
  &\exp\lc -\sum z_i b^{(\kappa)}_{t_i} -\int_{D^\infty\big([0,\infty)\big)_+} 1-\exp\lc - \sumd z_ix^{(\kappa)} (t_i)\rc \nu_V(\mathrm{d}x)\rc\\
  &´\overset{(\star)}{=}\exp\lc -\sum z_i b^{(\kappa)}_{t_i} -\int_{M^0_\ell} 1-\exp\lc - \sumd z_i x (t_i)\rc \nu_\kappa(\mathrm{d}x)\rc, 
\end{align*}
since we use that $x^{(\kappa)}\in M^0_\ell$ and we only omit those terms in $(\star)$ for which $\exp\big( - \sumd $ $ z_i x^{(\kappa)} (t_i)\big)=1$. 

% To verify that $\nu$ is a valid L\'evy measure it remains to show that
% \begin{align*}
%     \int_{M^0_\ell} \min\{1,x(t)\}\nu_\kappa(\mathrm{d}x)&= \int_{M^0_\ell} \min\bigg\{ 1, \int_{0}^t x(s)\kappa(\mathrm{d}s)\bigg\}\nu_V(\mathrm{d}x)<\infty.
% \end{align*}
% \cite[Theorem 3.4]{rosinskiinfdivproc} shows that $\nu_V\lc \{ x\in D^\infty(\RR) \mid \lim_{s\to0}x(s)=0 \}^\complement\rc =0$, since $\{ x\in D^\infty(\RR) \mid \lim_{s\to0}x(s)=0 \}$ is a standard Borel space and an algebraic group under addition. Therefore for each $x\in \{ x\in D^\infty(\RR) \mid \lim_{s\to0}x(s)=0 \}$ there exists some $\bar{t}=\bar{t}(x)\in\RR$ such that $x(s)\leq 2x(\bar{t})$ for all $s\leq t$.
% Thus, $\int_{0}^t x(s)\kappa(\mathrm{d}s)\leq \kappa\big((0,t]\big)2x(\bar{t})$. Thus,
% \begin{align*}
%     &\int_{M^0_\ell} \min\bigg\{ 1, \int_{0}^t x(s)\kappa(\mathrm{d}s)\bigg\}\nu_V(\mathrm{d}x)\leq 
%     \int_{M^0_\ell} \min\bigg\{ 1,\kappa\big((0,t]\big)2x(\bar{t}) \bigg\}\nu_V(\mathrm{d}x)\\
%     &\int_{M^0_\ell} \min\bigg\{ 1,\kappa\big((0,t]\big)C(x) \bigg\}\nu_V(\mathrm{d}x)
%     <\infty.
%\end{align*}
\end{proof}

\end{appendix}

\bibliographystyle{dinat}
\bibliography{quellenminid}% common bib file


\end{document}